\documentclass[10pt]{amsart}
\usepackage[utf8]{inputenc}
\usepackage[margin=1in]{geometry}
\usepackage{amsfonts,amssymb,amsthm,amsmath,enumitem,comment,bbm}
\usepackage[usenames,dvipsnames]{xcolor}
\usepackage{tikz}
\usepackage{caption}
\usepackage{subcaption}
\usepackage{graphicx}
\usepackage[hidelinks]{hyperref}
\usepackage{mathabx}   

 \usepackage{enumitem}    

\newcounter{problempart}

\allowdisplaybreaks
\usepackage{scalerel,stmaryrd}

\newcommand{\N}{\mathbb{N}}
\newcommand{\R}{\mathbb{R}}
\newcommand{\Q}{\mathbb{Q}}
\newcommand{\Z}{\mathbb{Z}}

\newcommand{\C}{\mathcal{C}}

\newcommand{\E}{\mathbb{E}}

\newcommand{\M}{\mathcal{M}}
\newcommand{\cN}{\mathcal{N}}

\newcommand{\ve}{\varepsilon}
\newcommand{\Ll}{\mathcal{L}}
\newcommand{\Pp}{\mathbb P}
\newcommand{\f}{\frac}
\newcommand{\deq}{\overset{d}{=}}
\newcommand{\mbf}{\mathbf}

\newcommand{\wt}{\widetilde}

\newcommand{\Rup}{\R_{\uparrow}^4}

\newcommand{\W}{W}

\newcommand{\DLBusedc}{\Xi}
\newcommand{\dir}{\xi}

\newcommand{\sig}{{\scaleobj{0.8}{\boxempty}}} 
\newcommand{\sigg}{{\scaleobj{0.9}{\boxempty}}} 

\newcommand{\arr}{a}
\newcommand{\arrv}{\pmb{a}}
\newcommand{\srv}{s}
\newcommand{\srvv}{\pmb{s}}
\newcommand{\dep}{d}
\newcommand{\depv}{\pmb{d}}
\newcommand{\Qs}{\mathcal{U}}
\newcommand{\Qu}{\mathcal{Q}}

\def\bq#1{\pmb{#1}}
\newcommand{\ml}{\mathcal{X}}
\newcommand{\Vmap}{\mathcal{V}}

\def\Ind{\mathcal{E}_{ind}}
\newcommand{\cE}{\mathcal{E}}
\def\en{M}

\newtheorem{theorem}{Theorem}[section]
\newtheorem{proposition}[theorem]{Proposition}
\newtheorem{corollary}[theorem]{Corollary}
\newtheorem{lemma}[theorem]{Lemma}
\newtheorem{conjecture}[theorem]{Conjecture}

\theoremstyle{definition}
\newtheorem{definition}[theorem]{Definition}

\newtheorem{assumption}{Assumption}
\numberwithin{equation}{section}
\theoremstyle{remark}
\newtheorem{remark}[theorem]{Remark}

\newcommand{\be}{\begin{equation}}
\newcommand{\ee}{\end{equation}}

\def\tsp{\hspace{0.5pt}}  
\def\tspa{\hspace{0.7pt}}
\def\tspb{\hspace{0.9pt}}

\newcommand\abullet{{\raisebox{2pt}{\scaleobj{0.5}{\bullet}}}}  
\newcommand\aabullet{{\tspb\raisebox{2pt}{\scaleobj{0.5}{\bullet}}\,}}  
\newcommand\aaabullet{{\tspb\raisebox{1pt}{\scaleobj{0.5}{\bullet}}\tspb}}  
\newcommand\bbullet{{\raisebox{0.5pt}{\scaleobj{0.6}{\bullet}}}} 

\def\ind{\mathbf{1}}
\def\ddd{\displaystyle}

\def\timonote#1{{\color{teal}{(T: #1)}}}

\def\ofernote#1{{\color{blue}{(O:#1)}}}

\def\MTR{K}   

\newcommand{\lzb}{\llbracket}   
\newcommand{\rzb}{\rrbracket}   

\newcommand{\fl}[1]{\lfloor{#1}\rfloor} 
\newcommand{\ce}[1]{\lceil{#1}\rceil}

\def\bck#1{{\overleftarrow{#1}}}

\DeclareMathOperator\Cls{Cls} 
\DeclareMathOperator\TwoRare{TwoRare}
\DeclareMathOperator\ClsJmp{ClsJmp} 
\DeclareMathOperator\OneRare{OneRare}

 \usepackage{enumitem}   

\title{Scaling limit of the TASEP speed process}

\author{Ofer Busani}
\address{Ofer Busani, Universit\"at Bonn,
Endenicher Allee 60,
Bonn, Germany}
\email{busani@iam.uni-bonn.de}
\author{Timo Sepp{\"a}l{\"a}inen}
\address{Timo Sepp{\"a}l{\"a}inen, University of Wisconsin-Madison, Mathematics Department, Van Vleck Hall, 480
Lincoln Dr., Madison WI 53706-1388, USA.}
\email{seppalai@math.wisc.edu}
\author{Evan Sorensen}
\address{Evan Sorensen, University of Wisconsin-Madison, Mathematics Department, Van Vleck Hall, 480
Lincoln Dr., Madison WI 53706-1388, USA.}
\email{elsorensen@wisc.edu}


\begin{document}
\maketitle

	\begin{abstract}
  The TASEP speed process introduced by Amir, Angel and Valk\'o in 2011 is a simultaneous coupling of all the translation-ergodic invariant distributions of multiclass totally asymmetric simple exclusion processes (TASEPs). It is defined as the process of limiting speeds of second-class particles started from each lattice site so that initially each particle sees a full lattice behind and an empty lattice ahead.  We show that suitably scaled, the TASEP speed process converges weakly to the stationary horizon (SH), a stochastic process recently introduced and studied by the authors. Specifically, around each interior speed value, the family of continuously interpolated level curves of the TASEP speed process converges to a coupled family of Brownian motions with drift, and this limiting function-valued stochastic process is precisely SH. SH is  believed to be the universal scaling limit of Busemann processes in the KPZ universality class. Our results add to the evidence for this universality by connecting SH with multiclass particle configurations.  Previously SH has been associated with the exponential corner growth  model, Brownian last-passage percolation, and the directed landscape (DL). As a consequence of the DL connection, we show that, in a certain technical sense, the set of speed process values converges weakly to the set of exceptional directions of DL, and the convoys of equal speed process values converge to the Busemann difference profiles. 
	\end{abstract}
	
\tableofcontents	
	
	\section{Introduction}

 \subsection{Universality in KPZ}
	The Kardar-Parisi-Zhang (KPZ) universality class is a large collection of random growth models that share a common scaling limit called the KPZ fixed point, a continuous-time Markov process taking values in the space of upper semi-continuous functions on the reals. The meaning of the ``universality'' of the KPZ class has  gradually developed over the past quarter century, from the one-dimensional distribution \cite{Baik-Deift-Johansson-99,baik2000limiting}, through the functional one \cite{Prahofer-Spohn-02,sasamoto2005spatial,Baik-Ferrari-Peche-10,borodin2008transition}, as line ensembles  \cite{corwin2014brownian,dimitrov2021characterization},  as a Markov process \cite{KPZfixed,KPZ_equation_convergence,heat_and_landscape}, and finally as a ``directed metric'' \cite{Directed_Landscape,Dauvergne-Virag-21}.
	
	Recently, the first author introduced a new scaling limit, the \textit{stationary horizon} (SH) \cite{Busani-2021}.  SH is a continuous-function-valued cadlag process indexed by the real line. Its construction   was achieved, building on results from \cite{Fan-Seppalainen-20}, through a diffusive  scaling of the Busemann process of  exponential last-passage percolation (LPP).  Not long after and independently  the second and third author discovered the SH as the Busemann process of the Brownian LPP \cite{Seppalainen-Sorensen-21b}, and uncovered quantitative information about its finite-dimensional distributions.  Very broadly speaking, Busemann processes are random objects holding much of the information on infinite geodesics in metric-like models \cite{Newman, hoffman2008, Sepp_lecture_notes}. It was conjectured in \cite{Busani-2021} that the SH is the scaling limit of the Busemann process of models in the KPZ class. 
	
	LPP models in the KPZ class  belong to a family of metric-like models: they satisfy a form of the triangle inequality, but are not necessarily positive or symmetric. These models are believed to  share a common limiting behavior under the  
 $1\!:\!2\!:\!3$ 
 scaling, namely,  the directed landscape (DL) \cite{Directed_Landscape,Dauvergne-Virag-21}. The DL holds more information than the KPZ fixed point in the sense that it allows for the coupling of initial conditions.  In \cite{Busa-Sepp-Sore-22arXiv}, building on results from \cite{Rahman-Virag-21}, the authors of the present  paper showed that the SH is the Busemann process of the directed landscape, thus settling part of a conjecture from \cite{Dauvergne-2021}. The result proved to have valuable applications to the study of infinite geodesics in the DL. The work of~\cite{Rahman-Virag-21} also studied the scaling limit of the trajectory of a second class particle for the particle system known as TASEP (discussed below) and showed that it converges to the competition interface of the DL.

 \subsection{Exclusion processes}
	Among the many types of models in the KPZ class are interacting particle systems, in particular,  exclusion processes. These models consist of particles on $\Z$, each performing an independent rate 1 continuous-time random walk with jump kernel $p:\Z\times \Z\rightarrow [0,1]$ under the exclusion rule: a particle's attempted  jump is executed  
	if the target site is vacant, otherwise suppressed. In the Harris-type probabilistic graphical construction of such a process we  attach to each directed edge $(x,y)$ a Poisson clock of rate $p(x,y)$ that generates the jump attempts.  Since their introduction in the mathematical literature in the 1970s \cite{spitzer1991interaction} exclusion processes have been extensively studied \cite{liggett1985interacting}. Exclusion processes can be mapped into growing interfaces, which under some conditions  (including positive drift) are believed to be in the KPZ class \cite{Corwin-survey}.  
	
	The particular  case  $p(x,x+1)=1$  is   the \textit{totally asymmetric simple exclusion process} (TASEP).  Each particle attempts to make nearest-neighbor jumps to the right at rate one, and a jump is executed only if the  site to the right is empty. There is a coupling between exponential LPP and the TASEP, and so showing that one is in the KPZ class implies the membership of the other. However, this connection between metric-like models and particle systems does not hold in general. The two families of models are amenable to different techniques. For example, the proof of the convergence of the KPZ equation to the KPZ fixed point was achieved  through two different approaches, where \cite{KPZ_equation_convergence} is tailored for particles systems while \cite{heat_and_landscape} is more suitable for  LPP and polymer models.
	
	Previously SH has been found in the context of LPP models. In this paper we complement the picture by showing  that   \textit{SH appears as a scaling limit also in exclusion processes}.  
	While geodesics and Busemann functions might not have natural counterparts in exclusion processes,   one feature of the Busemann function is common, namely, its invariance  under the dynamics of the model. Stationary measures of one-dimensional  exclusion processes are well-known \cite[Chapter VIII]{liggett1985interacting}: under very  general assumptions on $p$, the i.i.d.\ Bernoulli product measures  $\nu^\rho$ on $\{0,1\}^\Z$ with particle density  $\rho\in[0,1]$ are the translation-invariant, extremal  stationary measures under the exclusion dynamics. 

 \subsection{Single and multitype stationary distributions of TASEP}
	The family $\{\nu^\rho\}_{\rho\in[0,1]}$ has been  instrumental for example in the study of hydrodynamic limits of exclusion processes \cite{kipnis1998scaling}. In \cite{benassi1987hydrodynamical,andjel1987hydrodynamic},  it was shown that when started from $\nu^{\lambda,\rho}$ (the product measure on $\Z$ with intensity $\lambda$ to the left of the origin and intensity $\rho$ to the right), the TASEP particle profile will converge to either a rarefaction fan or a moving shock depending on the values of $\rho$ and $\lambda$. When $\rho>\lambda$, i.e.\ the shock hydrodynamics,  \cite{ferrari1991microscopic}  showed the existence of a microscopic stationary  profile  as seen from the shock. These studies utilized couplings $\mu^{\lambda,\rho}$ of the measures $\nu^\rho$ and $\nu^\lambda$ that are themselves  stationary under the joint TASEP dynamics of two processes that evolve in   \textit{basic coupling}.  Basic  coupling means that two or more exclusion processes, each from their own initial state,  are run together with common Poisson clocks.

	The stationary measure  $\mu^{\lambda,\rho}$ is sometimes called the two-type stationary measure. This is because one can realize the basic coupling by introducing two types of particles on $\Z$:  first class particles  whose distribution is $\nu^\lambda$, and second class particles, so that, when classes are ignored, the distribution of first and second class particles together  is   $\nu^\rho$. The dynamics is such that first class particles have priority over second class particles in the sense that the latter are treated as holes by the former. Second class particles represent discrepancies and so track the flow of information across space-time. Hence in some sense they assume the role of LPP geodesics. In the   hydrodynamic limit their space-time trajectories trace the characteristics of the limiting scalar conservation law \cite{ferr-92,ferrari1995second,reza-95,sepp01IIclass}.
	
	The two-type stationary measures
	$\mu^{\lambda,\rho}$ generalize  to 
		multitype stationary measures $\mu^{\rho_1,\dotsc,\rho_n}$. These measures  and their Ferrari-Martin construction by queueing mappings  \cite{Ferrari-Martin-2007}   are central players in this paper.  

 \subsection{Second class particles and the speed process}
	In \cite{ferrari1995second}, it was shown that   the normalized position of the  second class particle started at the origin in the step initial condition  converges in probability to a random speed uniformly distributed on $[-1,1]$. This convergence was strengthened to an almost sure one in \cite{mountford2005motion}. In other words, the second class particle chooses a limiting speed or characteristic line uniformly at random.
		The situation was further studied in  
	\cite{ferrari2009collision} which showed that the probability that a second class particle overtakes a third class particle in the rarefaction fan is $2/3$. 
	To obtain the full joint distribution of the speeds of particles of infinitely many classes, Amir, Angel, and Valk\'o~\cite{Amir_Angel_Valko11} constructed   the  \textit{TASEP speed process} $\{U_i\}_{i\in\Z}$. For each $i\in\Z$, the uniformly random value  $U_i\in[-1,1]$ is the limiting speed of the second class particle that started in a step configuration centered at site $i$.  The reader is referred to \cite{Amir_Angel_Valko11} for some of the fascinating properties of the speed process.  More recent studies of speed processes appear in \cite{aggarwal2022asep,amir2021tazrp}.  For our purposes, the key features of  the TASEP speed process are that it itself  is again invariant under suitably formulated multiclass TASEP dynamics, 
	and it provides a simultaneous coupling of all the stationary multiclass measures $\mu^{\rho_1,\dotsc,\rho_n}$ for any number of particle classes.  
	

\subsection{TASEP speed process, SH, and DL}
	Our  main result   Theorem \ref{thm:conv} states that when suitably scaled around a speed $v\in(-1,1)$, the TASEP speed process converges in distribution to SH. In particular, we  connect the multitype stationary distributions of TASEP to SH through the TASEP speed process. 
 The information used in the scaling is the number of particles in  a lattice interval of order $N$ whose speed deviates from the centering $v$ by order $N^{-1/2}$. These particle counts are converted into continuous height  functions by the standard mapping that turns TASEP particle configurations into interfaces. The joint process of these height functions is then scaled diffusively.  

 Since  SH is the distribution of the Busemann process of DL, as a corollary we get a limit theorem that captures the convergence of the scaled and centered speed process values to the \textit{exceptional directions} of DL, and the convergence of the 
 interpolated cumulative convoys to the  Busemann difference profiles of DL. 
 The exceptional directions of DL are those into which the uniqueness and coalescence of semi-infinite geodesics fail.  These results are proved in terms of the weak convergence of a simple point measure of speed process values and interpolated convoys to the corresponding object in DL (Theorem \ref{thm:Xi}).


\subsection{Basic coupling versus LPP construction}
In response to several queries about this work, we emphasize that \textit{the result is not a  consequence of the SH limit of the Busemann process of exponential LPP given in  \cite{Busani-2021},  nor a consequence of the KPZ  limit of multiple TASEPs given in  \cite[Theorem 1.20]{Dauvergne-Virag-21}}. The underlying reason is the distinction between two constructions of TASEP: with Poisson clocks on the edges of $\Z$, and in terms of LPP on the planar lattice. These two constructions yield the same process when TASEP is started from a single initial condition. Less clear is the connection  between multiple TASEPs in basic coupling constructed with  Poisson clocks, and the dynamics of LPP applied to multiple initial conditions.  The multiclass distributions studied here are invariant for joint TASEPs in basic coupling, constructed in terms of Poisson clocks.  By contrast, the SH limit in  \cite{Busani-2021} and the joint KPZ limit in \cite{Dauvergne-Virag-21} utilize LPP.    For this same reason we do not yet have a space-time limit that would connect the temporal evolution of multiclass TASEP with a space-time process whose invariant distribution is SH. The multivariate KPZ fixed point, constructed in terms of the variational formula in the random environment of DL, does possess SH as an invariant distribution \cite[Theorem 2.1]{Busa-Sepp-Sore-22arXiv}.

\subsection{Notation and conventions}

We collect here some conventions for quick reference. 
$\Z_+=\{0,1,2,\dotsc\}$ and $\N=\{1,2,3,\dotsc\}$. Integer intervals are denoted by $\lzb a,b\rzb=\{n\in\Z: a\le n\le b\}$.  The space $C(\R)$ of continuous functions on $\R$ is equipped with its  Polish topology of uniform convergence on compact subsets of $\R$. 
The indicator function of an event $A$ is  denoted by $\ind[A]$ and $\ind_A$.  The floor and ceiling of a real $x$ are $\fl{x}=\max\{n\in\Z: n\le x\}$ and $\ce{x}=\min\{n\in\Z: n\ge x\}$.

For random variables $X$, $Y$ and  $Z$ and a probability measure $\mu$,  $X\sim Y$ and $X\deq Y$  mean that $X$ and $Y$ have the same distribution and  $Z\sim\mu$ means that $Z$ has distribution $\mu$. Convergence in distribution is denoted by $\Rightarrow$.  $Z\sim{\rm Ber}(\alpha)$ is the abbreviation for the Bernoulli distribution $P(Z=1)=\alpha=1-P(Z=0)$.  When the value $Z=0$ represents a vacant site in a particle configuration, in certain situations $Z=0$ is replaced by $Z=\infty$.  $X\sim$ Geom$(p)$ means that $P(X=k)=p(1-p)^k$ for $k\in\Z_+$, that is, the distribution of the number of failures until the first success with probability $p$.   

If $B(\aaabullet)$ is a standard Brownian motion, then for $c>0$ and $\lambda\in\R$,  $t\mapsto c B(t)+\lambda t$ is a Brownian motion with diffusivity $c$ and drift $\lambda$. 

The i.i.d.\ Bernoulli product measure $\nu^\alpha$ on the sequence  space $\{0,1\}^\Z$ satisfies  
$\nu^\alpha\{\eta: \eta(x_1)=\dotsm=\eta(x_m)=1\}=\alpha^m$ for any $m$ distinct sites $x_1,\dotsc,x_m\in\Z$ and with  generic elements of $\{0,1\}^\Z$ denoted by $\eta=\{\eta(x)\}_{x\in\Z}$. We call $\alpha$ the \textit{density} or the  \textit{intensity} of $\nu^\alpha$. As above, empty sites are denoted by both $0$ and $\infty$, depending on the context.  \textit{Translation invariance} of a probability measure on a sequence space  means invariance under the mapping $(T\eta)(x)=\eta(x+1)$. 

In queueing theory, a bi-infinite sequence is denoted by a boldface version of the same letter that denotes the entries, together with additional indices, as for example in $\bq{x}_k=\{x_k(j)\}_{j\in\Z}\in\{1,\infty\}^\Z$.  

Single-variable functions apply to sequences coordinatewise: if $u=\{u_i\}_{i\in\Z}\in\R^\Z$   and $\phi:\R\to\R$, then $\phi(u)=\{\phi(u_i)\}_{i\in\Z}$.

In proofs, constants such as $C$ and $c$ can change from line to line. 

\subsection{Acknowledgements}
We thank Duncan Dauvergne for pointing out a mistake in the first version of this paper and for helpful discussions. O.~Busani also thanks Pablo Ferrari  for a guide to the literature and and M\'arton Bal\'azs for helpful discussions.  The work of O.~Busani was funded by the Deutsche Forschungsgemeinschaft (DFG, German Research Foundation) under Germany’s Excellence Strategy--GZ 2047/1, projekt-id 390685813, and partly performed at  University of Bristol. T.\ Sepp\"al\"ainen was partially supported by National Science Foundation grant DMS-2152362 and by the Wisconsin Alumni Research Foundation. E.~Sorensen was partially supported by T.~Sepp{\"a}l{\"a}inen under National Science Foundation grant DMS-2152362.


 \section{Stationary horizon limit of the speed process} 
 \label{sec:main} 
	
	We first introduce the TASEP speed process and the stationary horizon (SH). Then we explain how the speed process is scaled and state the main result, namely,  that the scaled speed process converges weakly to SH  on a function-valued cadlag path space (Theorem~\ref{thm:conv} below).  
	
	\subsection{TASEP speed process} \label{sec:TSPintro}
	In the simplest TASEP dynamics each site of $\Z$  contains either a particle or a hole.  Each site has an  independent rate 1 Poisson clock. If at time $t$ the clock rings at site $x\in\Z$ the following happens. If there is a particle at site $x$ and no particle at site $x+1$ then the particle at site $x$ jumps to site $x+1$, while the other sites  remain unchanged. If there is no particle at site $x$ or there is a particle at site $x+1$ then the jump is suppressed. In other words, a particle can jump to the right only if the target site has no particle at the time of the jump attempt. This is the \textit{exclusion rule}. TASEP is a Markov process on the compact  state space $\{0,1\}^\Z$. Generic elements of $\{0,1\}^\Z$, or \textit{particle configurations}, are denoted by $\eta=\{\eta(x)\}_{x\in\Z}$, where $\eta(x)=1$  means that site $x$ is occupied by a particle and $\eta(x)=0$ that site $x$ is occupied by a hole, in other words, is empty. The infinitesimal  generator $\mathcal{L}$  of the process acts on functions  $f$ on $\{0,1\}^\Z$ that are supported on finitely many sites  via 
	\begin{equation}\label{Tg}
		\mathcal{L}f(\eta)=\sum_{x\in\Z}\eta(x)(1-\eta(x+1))[f(\eta^{x,x+1})-f(\eta)]
	\end{equation}
	where $\eta^{x,x+1}$ denotes the configuration after the contents of sites $x$ and $x+1$ have been exchanged: 
	\begin{equation*}
		\eta^{x,x+1}(z)=
		\begin{cases}
			\eta(z) & \text{ if $z\notin\{x,x+1\}$}\\
			\eta(x+1)& \text{ if $z=x$}\\
			\eta(x)& \text{ if $z=x+1$}.
		\end{cases}
	\end{equation*}
We do not work with the generator, but it serves as a convenient summary of the dynamics.

	For each density $\rho\in[0,1]$ the i.i.d.\ Bernoulli  distribution $\nu^\rho$ on $\{0,1\}^\Z$ with density $\rho$ is  the unique translation-invariant extremal stationary distribution of particle density $\rho$ under the TASEP dynamics. 
	
	There is a natural way to couple multiple  TASEPs from different initial conditions but with the same driving dynamics. Let $ \{\cN_x:x \in \Z\}$ be a $\Z$-indexed  collection of independent  rate $1$ Poisson  processes on $\R$. The clock at location $x$ rings at the times that correspond to points in $\cN_x$. One can then take two densities $0\leq \rho^1\leq \rho^2\leq 1$ and ask whether there exists a coupling measure $\pi^{\rho_1,\rho_2}$ on $\{0,1\}^\Z\times\{0,1\}^\Z$ with Bernoulli marginals $\nu^{\rho_1}$ and $\nu^{\rho_2}$  that  is stationary under the joint TASEP dynamics
	and ordered. In other words, the twin requirements are    that if initially  $(\eta^1,\eta^2)\sim \pi^{\rho_1,\rho_2}$,  then $(\eta^1_t,\eta^2_t)\sim\pi^{\rho_1,\rho_2}$ at all subsequent times $t\ge0$, and $\eta^1(x)\leq \eta^2(x)$ for all $x\in\Z$ with $\pi^{\rho_1,\rho_2}$-probability one. Such a two-component  stationary distribution exists and is unique \cite{Liggett76}.
	
	One reason for the  interest in stationary measures of more than one density comes from the connection between the TASEP dynamics on $k$ coupled profiles in the state space $(\{0,1\}^\Z)^k$ and the TASEP dynamics on particles with classes in $\lzb1,k\rzb=\{1,\ldots,k\}$, called multiclass or multitype  dynamics. In the $k$-type dynamics, each  particle has a class in $\lzb1,k\rzb$ that remains the same for all time. A  particle  jumps to the right, upon the ring of a Poisson clock, only if there is either a hole or a particle of lower class (higher label) to the right. If this happens, the lower class particle moves left.  The state space of  $k$-type dynamics is  $\{1,\ldots,k,\infty\}^\Z$, with generic configurations denoted again by $\eta=\{\eta(x)\}_{x\in\Z}$. A value   $\eta(x)=i\in\lzb1,k\rzb$ means that site $x$ is occupied by  a particle of class $i$, and  $\eta(x)=\infty$ means  that site $x$ is empty, equivalently, occupied by a   hole.  Denoting a hole by $\infty$ is convenient now because holes can be equivalently viewed as particles of the absolute  lowest class.  
	For $k=1$ the multitype dynamics is the same as basic TASEP.
	
	The next question is whether we can  couple all the invariant multiclass distributions  so that the resulting construction is still invariant under TASEP dynamics. This was achieved by \cite{Amir_Angel_Valko11}:    such couplings can be realized by applying   projections to an object they constructed and named the \textit{TASEP speed process}. We describe briefly the construction. To start,  each site  $i\in\Z$ is occupied by a particle of class $i$.  This creates the initial   profile $\eta_0\in \Z^\Z$ such that $\eta_0(i)=i$. Let $\eta_t$ evolve under  TASEP dynamics, now interpreted so that a  particle switches places with the particle to its right only if the particle to the right  is of lower class, that is, has a higher label. Note that now each site is always occupied by a particle of some integer label.   The limit  from \cite{mountford2005motion}  
	implies  that each  particle has a well-defined limiting    speed:  if $X_t(i)$ denotes  the time-$t$ position of the particle initially at site $i$,  then the following random limit exists almost surely: 
	\begin{equation}\label{speed8}
	U_i= \lim_{t\to\infty} t^{-1}  X_t(i). 
	\end{equation}      
	The process $\{U_i\}_{i\in\Z}$ is   the TASEP speed process. It is a random element of the space  $[-1,1]^\Z$. 

\begin{theorem}[{\cite[Theorem 1.5]{Amir_Angel_Valko11}}]  \label{thm:AAV}
    The TASEP speed process $\{U_i\}_{i\in\Z}$ is the unique invariant distribution of TASEP that is ergodic under translations of the lattice $\Z$ and such that each $U_i$ is uniformly distributed on $[-1,1]$. 
\end{theorem}
In the context of the theorem above, the TASEP state  $\eta=\{\eta(i)\}_{i\in\Z}$ is a real-valued sequence but the meaning of the dynamics is the same as before.  Namely, at each pair $\{i,i+1\}$ of nearest-neighbor sites, at the rings of a rate one exponential clock, the variables $\eta(i)$ and $\eta(i+1)$ are swapped if $\eta(i)<\eta(i+1)$, otherwise left unchanged.

 A key point  is that the
		 TASEP speed process projects to multitype stationary distributions. 
		
	
 \begin{theorem}[{\cite[Theorem 2.1]{Ferrari-Martin-2007}, \cite[Theorem 2.1]{Amir_Angel_Valko11}}] \label{thm:TASEPprod}
	Let $k\in\N$ be the number of classes. Let  $\bar\rho=(\rho_1,\ldots,\rho_k)\in(0,1)^k$ be a parameter vector such that  $\sum_{i=1}^k  \rho_i \le 1$. Then there is a  translation-invariant  stationary distribution $\widecheck\mu^{\bar\rho}$ for the $k$-type TASEP which is unique under the conditions {\rm(i)} and {\rm(ii)}, and also under the conditions {\rm(i)} and {\rm(ii')} below:  
 
 {\rm(i)}  $\widecheck\mu^{\bar\rho}\{\eta\in \{1,\ldots,k,\infty\}^\Z:\eta(x) = j\} = \rho_j$ for each site $x\in\Z$ and class $j\in\lzb1,k\rzb$; 
 
 {\rm(ii)} under $\widecheck\mu^{\bar\rho}$, for each $\ell\in\lzb1,k\rzb$, the distribution of the $\{0,1\}$-valued sequence  $\{\ind[\eta(x) \le \ell]\}_{x \in \Z}$ of indicators  is the i.i.d.\ Bernoulli measure $\nu^{\sum_{j=1}^{\ell} \rho_j}$ of intensity  $\sum_{j=1}^{\ell} \rho_j$;

 {\rm(ii')}  $\widecheck\mu^{\bar\rho}$ is  ergodic under the translation of the lattice $\Z$.
 
 Furthermore, $\widecheck\mu^{\bar\rho}$ is   extreme among translation-invariant stationary   measures of the $k$-type dynamics with jumps to the right. 
	\end{theorem}

 Theorem \ref{thm:TASEPprod} is not stated exactly in this form in either reference.  It can be proved with the techniques of Section VIII.3 of Liggett \cite{liggett1985interacting}. 
	
	\begin{lemma}[{\cite[Corollary 5.4]{Amir_Angel_Valko11}}]\label{lem:FU}
		Let $F:[-1,1] \rightarrow \{1,\dotsc,k, \infty\}$ be a  nondecreasing function and  $\lambda_j= \f{1}{2} {\rm Leb}\big(F^{-1}(j)\big)$, i.e.,  one-half the Lebesgue measure of the interval  mapped to the value $j\in\{1,\dotsc,k, \infty\}$.
		Then  the distribution of the $\{1,\dotsc,k, \infty\}$-valued sequence  $\{F(U_i)\}_{i\in\Z}$ is   the stationary measure $\widecheck\mu^{(\lambda_1,\dotsc,\lambda_k)}$ described in Theorem \ref{thm:TASEPprod} for  the $k$-type TASEP with jumps to the right. 
	\end{lemma} 
 For example, the case $k=1$ of  Lemma \ref{lem:FU} tells us that  to produce a particle configuration with Bernoulli distribution  $\nu^\rho$ from the TASEP speed process,  assign a particle to each site $x$ such that  $U_x\le2\rho-1$.   Lemma \ref{lem:FU} follows readily from Theorems \ref{thm:AAV} and  \ref{thm:TASEPprod} because the nondecreasing projection $F$ commutes with the pathwise dynamics. 

\begin{remark}[Jump directions] Throughout this Section \ref{sec:main} jumps in TASEP go to the right.  Later in Sections \ref{sec:fdd} and 
\ref{sec:tightness} we use the convention from \cite{Ferrari-Martin-2007} whereby TASEP jumps proceed left. This is convenient because then discrete  time in the queueing setting agrees with the order on $\Z$.  Notationally,  $\widecheck\mu^{\bar\rho}$ denotes the multiclass stationary measure under rightward jumps, as in Theorem~\ref{thm:TASEPprod} and Lemma~\ref{lem:FU} above, while 
  $\mu^{\bar\rho}$ will denote the stationary measure under leftward jumps. These  measures are simply reflections of each other (see Theorem~\ref{thm:FM}). 
\end{remark}

\subsection{The stationary horizon} \label{sec:SH} 
The stationary horizon (SH) is a process $G = \{G_\mu\}_{\mu \in \R}$ with values $G_\mu$ in the space $C(\R)$ of continuous $\R\to\R$ functions. $C(\R)$ has its Polish topology of uniform convergence on compact sets. The paths $\mu\mapsto G_\mu$ lie  in the Skorokhod space $D(\R,C(\R))$.   For each $\mu \in \R$, $G_{\mu}$ is a two-sided Brownian motion with diffusivity $\sqrt 2$ and drift $2\mu$. 
 With these conventions for the  diffusivity and drift, $G$  is the version of SH associated to the directed landscape and the KPZ fixed point, as developed in our previous paper \cite{Busa-Sepp-Sore-22arXiv}.  
The distribution of a $k$-tuple  $(G_{\mu_1},\dotsc,G_{\mu_k})$ can be realized as the  image of $k$  independent Brownian motions with drift. See Appendix~\ref{sec:stat_horiz} for a  description. 


	\subsection{Scaling limit of the speed process} \label{sec:scsp}
The space $\{0,1\}^\Z$ of TASEP particle configurations $\eta$ can be mapped bijectively  onto the space of  continuous  interfaces $f:\R\to\R$ such that $f(0)=0$, $|f(x)-f(x+1)|=1$ for all $x\in\Z$,  and $f(x)$  interpolates   linearly between integer points.  Define $\mathcal{P}:\{0,1\}^\Z\rightarrow C(\R)$  by stipulating that  on integers $i$ the image function $\mathcal{P}[\eta]$ is given by
	\begin{subequations} \label{eqn:TASEPh}
	\be
		\mathcal{P}[\eta](i)=
		\begin{cases}
			\sum_{j=0}^{i-1}(2\eta(j)-1), & i\in\N\\
			0, & i=0\\
			-\sum_{j={i}}^{-1}(2\eta(j)-1), & i\in -\N
		\end{cases}
		\ee
	and then extend $\mathcal{P}[\eta]$ to the reals by linear interpolation: 
	\be \text{for} \quad x\in \R\setminus \Z, \quad  \mathcal{P}[\eta](x)=(\ce x-x)\mathcal{P}[\eta](\fl x)+ (x-\fl x)\mathcal{P}[\eta](\ce x).
	\ee
	\end{subequations}
	TASEP can therefore be thought of as dynamics on continuous interfaces $f:\R \to \R$ such that for all $x \in \Z$, $f(x) \in\Z$ and  $f(x\pm1) \in \{f(x)-1,f(x) + 1\}$. When a particle at location $x$ lies immediately to the left of a hole at location $x + 1$, the interface  has a local maximum at location $x + 1$. When the particle changes places with the hole, the  local maximum becomes a local minimum.

	Let	$U=\{U_j\}_{j\in\Z}$ be the TASEP speed process and for $s\in\R$,  $\ind_{U < s}=\{\ind_{U_j< s}\}_{j\in\Z}$  a shorthand for the $\{0,1\}$-valued  sequence of indicators.   For each value of the centering $v \in (-1,1)$ and a scaling parameter $N\in\N$, use
	the mapping~\eqref{eqn:TASEPh} to define from the speed process a    $C(\R)$-valued process indexed by $\mu\in\R$: 
 \begin{equation} \label{H138}
     H^N_\mu(x) = H^{v,N}_\mu(x) =N^{-1/2}\,\mathcal{P}\big[\ind_{U \le v + \mu(1 -v^2) N^{-1/2}}\big]\Bigl(\f{2x}{1 - v^2}N\Bigr) - \f{2vx}{1-v^2}N^{1/2}, \qquad \mu, x\in\R.
 \end{equation}
		

Our main theorem is the process-level  weak limit of $H^{v,N} = \{H^{v,N}_\mu\}_{\mu \in \R}$.   The path space of $\mu\mapsto H^{v,N}_\mu$ is the Skorokhod space $D(\R,C(\R))$ of $C(\R)$-valued cadlag paths on $\R$, with its usual Polish topology. This is  discussed in Section~\ref{sec:D(R)}.  Here is our main result. 

	
	\begin{theorem}\label{thm:conv}
		Let $G$ be the stationary horizon. Then, for each $v \in (-1,1)$, as $N\rightarrow \infty$,  the  distributional limit $H^{v,N} \Rightarrow G$   holds on the path space  $D(\R, C(\R))$. 
		
	\end{theorem}
	
	The  proof of  Theorem \ref{thm:conv} is reached at the end of Section \ref{sec:tightness}. 
   As is typical,  the proof splits into two main steps: 
  (i) weak convergence of finite-dimensional distributions of  $H^{N}$ to the limiting object in Section \ref{sec:fdd}  and (ii) tightness of $\{H^{N}\}_{N\in\N}$ on $D(\R,C(\R))$ in Section \ref{sec:tightness}.   Both parts  use  the Ferrari-Martin  queueing representation of the  multitype stationary measures. 
  The first part shows that, in the limit, the queueing representation recovers  the queuing structure that defines the SH. 
	
	The   tightness of $\{H^{N}\}_{N\in\N}$  boils down to showing that, uniformly in $N$, $\mu\mapsto H^{N}_\mu$ does not have too many jumps on a compact interval. The main ingredients are the   \textit{reversibility} and \textit{interchangeability} of Markovian queues.  For a sequence of arrivals  $\arrv$ and  services $\srvv$, we write $Q(\arrv,\srvv)$ for the queue process, $D(\arrv,\srvv)$ for the departure process, and  $R(\arrv,\srvv)$ for the process of dual services.  Then  reversibility means that  the time-reversal of the  process $(D(\arrv,\srvv),R(\arrv,\srvv),{Q})$ has the same distribution as $(\arrv,\srvv,Q)$. Reversibility implies the Burke property for Markovian queues. Interchangeability is the property  $D(\arrv,\srvv^1,\srvv^2)=D(\arrv,R(\srvv^1,\srvv^2),D(\srvv^1,\srvv^2))$ where  $D(\arrv,\srvv^1,\srvv^2)$ is the departure process of queues $\srvv^1$ and $\srvv^2$ in tandem fed by the arrival process $\arrv$.  The queueing theory we use is covered  in Section \ref{sec:FM} and Appendix \ref{app:MM1}. 
	
	These two properties were used in \cite{Fan-Seppalainen-20} to construct the joint distribution of the Busemann process of exponential LPP, itself a key ingredient of the results in \cite{Busani-2021}.  In~\cite{Seppalainen-Sorensen-21b}, the authors used a continuous analogue of the same properties to describe the distribution of the Busemann process of Brownian LPP. It was observed in \cite{Busani-2021} that the Fan-Sepp\"al\"ainen construction in \cite{Fan-Seppalainen-20} can be obtained through an RSK-like procedure on random walks, named  \textit{stationary melonization} in \cite{Busani-2021}. Here,  the relevant version of  RSK (Robinson-Schensted-Knuth)  is an algorithm taking as input $N$ random walks $\srvv^1,\ldots,\srvv^N$ and returning $N$ non-intersecting/ordered paths, through  iterative application of a sorting map. 
	In that context,  the pair map   $(D,R)$ plays the role of the sorting map  
	and the interchangeability should be thought of as the isometry of the melonization procedure \cite{Biane-Bougerol-OConnell-2005,Directed_Landscape}. In contrast to the non-intersecting lines  output of  standard  RSK,  stationary melonization outputs  lines that agree pairwise  on a compact interval around the origin and branch off  outside of it. 
 

 \subsection{Beyond TASEP: general exclusion processes on $\Z$}
	
	 \cite{Amir_Angel_Valko11} conjectured that  an analogue of the TASEP speed process exists for ASEP, the exclusion process whose particles can jump to either of the two adjacent neighbors, but the symmetric case excluded. Assuming the conjecture, \cite{Amir_Angel_Valko11} derived properties of this putative process, including its  stationarity under the evolution. The existence question was recently settled in \cite{aggarwal2022asep}. In a related development, \cite{mart-20} constructed stationary distributions for multitype ASEP. 
  
  An analogue of our Theorem \ref{thm:conv} should hold for ASEP and even more generally for one-dimensional exclusion processes, provided the speed process or its analogue can be constructed.  In Section  \ref{sec:Asp} we take a step towards this extension, not  by constructing the speed process but by approaching the question from the other direction:  we  construct a stationary distribution for the exclusion process with continuum values whose projections are stationary distributions of multiclass particle processes. 

  In the next theorem we consider exclusion dynamics with a general jump kernel $p:\Z\times \Z\rightarrow [0,1]$  that satisfies the conditions of Section VIII.3 of Liggett \cite{liggett1985interacting}: 
translation invariance  $p(x,y)=p(0,y-x)$ and this  form of irreducibility: 
for each pair   $x,y\in\Z$ there exists $m\in\Z_+$
such that $p^{(m)}(x,y)+p^{(m)}(y,x)>0$ where $p^{(m)}$ is the $m$-step transition. 
  
  \begin{theorem}\label{exun}  
  There exists a random variable $W=\{W_i\}_{i\in\Z}\in [0,1]^\Z$ with uniform marginals   $W_i\sim \rm{U}[0,1]$ whose distribution is translation-invariant and  stationary under the generalized exclusion dynamics described below Theorem \ref{thm:AAV} but now using kernel $p$.  
{\rm(}The generator is given   in equation \eqref{L800}  in Section \ref{sec:Asp}.{\rm)}
If $F$ is an increasing function on $[0,1]$, then $F(W):=\{F(W_i)\}_{i\in\Z}$ is again a translation-invariant measure that is  stationary under these same  dynamics. Moreover, if $V\in[0,1]^\Z$ is  translation-ergodic with uniform marginals and its distribution is stationary under these dynamics, then $V\sim W$. 
  \end{theorem}

  
It is not a priori clear whether $W$ still contains information about the speeds of individual second class particles under the  general jump kernel $p$. However, it does follow  that if the   speed process $U^{\rm{spd}}$ exists and is stationary, as in \cite{Amir_Angel_Valko11,aggarwal2022asep},  there is a deterministic increasing function $\phi$ such that  $U^{\rm{spd}}\sim \phi(W)$ (Proposition \ref{NL} in Section \ref{sec:Asp}). For example, $\phi(v)=(1-2p)\frac{1+v}{2}$ in   ASEP with $p=p(x,x+1)$.	
  
  We conjecture that  the  stationary horizon  $G$  is a universal scaling limit of translation-invariant multiclass  stationary distributions.
	\begin{conjecture}\label{conj}
  In the setting of Theorem \ref{exun} fix a  suitable centering $v$.  Then there is a scaled version $H^{v,N}$ of the process  $H^v_\mu(x)=\mathcal{P}[\ind_{W\leq v + \mu}](x)$ such that $H^{v,N} \Rightarrow G$  on the path space  $D(\R,C(\R))$  as $N\to\infty$. 
	\end{conjecture}
	    
\section{Speed process and exceptional directions of the directed landscape}

A consequence of Theorem \ref{thm:conv} is that features of the TASEP speed process approximate, in distribution, certain geometrically relevant  features of the directed landscape (DL). Namely, 
(i)  the  set  $\Xi^{v,N}=\bigl\{  N^{1/2}\tfrac{U_i-v}{1-v^2}:  i\in\Z  \bigr\}$ of scaled and centered speed process values  approximates the set of exceptional directions of DL and (ii)  the suitably scaled and interpolated cumulative convoy associated to a speed process value is an approximation of the Busemann difference profile associated to the corresponding exceptional direction of DL.   At the end of this section  we formulate the  result (Theorem \ref{thm:Xi}) as the weak limit of a point measure based on the support $\Xi^{v,N}$, but a technical issue arises.  The set $\Xi^{v,N}$  is not discrete, and also the limiting set of exceptional directions of DL is  dense in $\R$.  Furthermore, the entire function space $C(\R)$ is a bounded set under the metric \eqref{d} below.    Hence these ingredients alone do not give us a point measure that is finite on bounded sets. To fix this  we add  a third component to the point measure, one whose almost sure continuity can be readily proved on the path space of the stationary horizon.  

  We begin with a brief description of  the directed landscape   and refer the reader to the papers  \cite{Busa-Sepp-Sore-22arXiv,Dauvergne-Virag-18,Dauvergne-Virag-21, Rahman-Virag-21} for more coverage.


\subsection{Directed landscape and its Busemann process}  

The directed landscape (DL) is a random continuous function $\Ll:\Rup \to \R$
on the domain $\Rup = \{(x,s;y,t) \in \R^4: s < t\}$ of time-ordered pairs of space-time points. 
 It arises as the scaling limit of various last-passage type  models in the KPZ universality class, and is expected to be a universal limit of such models.   DL satisfies 
\be \label{eqn:metric_comp}
\Ll(x,s;y,u) = \sup_{z \in \R}\{\Ll(x,s;z,t) + \Ll(z,t;y,u)\}
\ee
 for $(x,s;y,u) \in \Rup$ and $t \in (s,u)$, so it is  a directed LPP process.   It can also be viewed as a signed ``directed metric'', though the triangle inequality is reversed.  But it is still profitable to define geodesics.  A continuous path $g:[s,t] \to \R$  is
 a \textit{geodesic} if   every  partition $s = t_0 < t_1 < \cdots < t_k = t$ satisfies 
\[
\Ll(g(s),s;g(t),t) = \sum_{i = 1}^k \Ll(g(t_{i - 1}),t_{i - 1};g(t_i),t_i).
\]
  For  fixed $(x,s;y,t) \in \Rup$, there exists almost surely a unique geodesic between $
(x,s)$ and $(y,t)$ \cite[Sect.~12--13]{Directed_Landscape}.  A \textit{semi-infinite geodesic} with initial point  $(x,s) \in \R^2$ is a continuous path $g:[s,\infty) \to \R$ such that $g(s) = x$ and   the restriction of $g$ to each bounded interval $[s,t]\subseteq[s,\infty)$ is a geodesic between $(x,s)$ and $(g(t),t)$. It has {\it direction} $\dir \in \R$ if $\lim_{t \to \infty} g(t)/t=\dir$.

 Information about the geodesics of DL is contained in its \textit{Busemann process}  
 \[
\bigl\{\W_{\dir \sig}(x,s;y,t): \dir \in \R, \,  \sigg \in \{-,+\}, \, (x,s),(y,t) \in \R^2\bigr\}  . 
\]
This is a real-valued stochastic process indexed by a pair of (not necessarily time-ordered) space-time points $(x,s),(y,t) \in \R^2$, a direction $\dir \in \R$, and a sign $\sigg \in \{-,+\}$.   We summarize properties of this process from \cite{Busa-Sepp-Sore-22arXiv}. The statements all hold with probability one, across all the values of the parameters in question. 

 For each fixed $\xi\sigg$,   $\W_{\dir \sig}\in C(\R^4,\R)$. In the topology of $C(\R^4,\R)$ of uniform convergence on compact sets, 
 $\dir\mapsto\W_{\dir +}$ is right-continuous and $\dir\mapsto\W_{\dir -}$  left-continuous. The two functions $\W_{\dir\pm}$ agree for all but a countable dense subset $\DLBusedc$ of \textit{exceptional directions} $\dir$ of DL: 
\be\label{DL90} 
 \DLBusedc  =\{   \dir \in \R :  \exists (x,s),(y,t) \in \R^2 \text{ such that }  \W_{\dir -}(x,s;y,t) \neq \W_{\dir +}(x,s;y,t)\}.
 \ee 
 A fixed $\dir$ is never exceptional: $\Pp(\xi\in \DLBusedc)=0$ $\forall\xi\in\R$. 
 
The set  $\DLBusedc$ of exceptional directions lies at the heart of the uniqueness and coalescence of semi-infinite geodesics in DL (Theorem 2.5 in \cite{Busa-Sepp-Sore-22arXiv}).    For $\xi\notin\DLBusedc$,  all semi-infinite geodesics in  direction $\xi$ coalesce and, outside of a Lebesgue-null set of initial space-time points, the $\xi$-directed semi-infinite geodesic is unique.    By contrast, if $\xi\in\DLBusedc$, then from each initial point there are at least two $\xi$-directed semi-infinite geodesics that eventually separate and never meet again.    These geodesics form at least  two distinct coalescing families of $\xi$-directed semi-infinite geodesics.

For our purposes it is enough to consider the Busemann process $\W_{\dir \sig}(x,t;y,t)$ restricted to a fixed  time level $t\in\R$.  $\W$ is stationary and mixing under every translation of the space-time $\R^2$ so the choice of $t$ is arbitrary.   The connection between $\Ll$ and $\W$ is the Busemann limit:
for all $\dir \in \R$, $t \in \R$, $x < y$ in $\R$, and any sequence $(z_n,u_n)_{n\in\N}$ in $\R^2$ such that   $u_n \to \infty$ and $z_n/u_n \to \dir$ as $n \to \infty$,
 \be\label{W45}    \begin{aligned}
    \W_{\dir -}(y,t;x,t) &\le \liminf_{n \to \infty}  \bigl[ \Ll(y,t;z_n,u_n) - \Ll(x,t;z_n,u_n) \bigr]  \\  &\le \limsup_{n \to \infty}  \bigl[ \Ll(y,t;z_n,u_n) - \Ll(x,t;z_n,u_n) \bigr] \le \W_{\dir +}(y,t;x,t).  
    \end{aligned}\ee
    For $\xi\notin\DLBusedc$ the extreme left and right members coincide and the limit holds. 

The set of exceptional directions is the focus of our study. Define the difference profile 
    \be \label{fsdir}
J_{\dir}(x) = \W_{\dir +}(x,t;0,t) - \W_{\dir -}(x,t;0,t) \quad\text{ for } x\in\R, 
\ee
an identically zero function unless $\dir\in\DLBusedc$.   
$J_\xi$ is a nondecreasing function with $J_\dir(0)=0$.  For all choices of $t\in\R$,  $\dir\in\DLBusedc$  is equivalent to   $J_\xi(x)\nearrow\infty$
as $x\nearrow\infty$.  The realization of $J_\dir$ varies from one choice of $t$ to the next, but for each $t$ the random set of exceptional directions $\dir$ such that  $J_\dir\not\equiv0$ is the same $\DLBusedc$.   
Under Palm conditioning on the event $\dir\in\DLBusedc$,  $J_\xi$ vanishes on a random open neighborhood   $(-\bck{\tau}_{\!\!\dir} , \tau_\xi)$ around $x=0$, and beyond this interval, $x\mapsto J_\xi(x)$ for $x\ge\tau_\xi$ and $x\mapsto -J_\xi(-x)$ for $x\ge\bck{\tau}_{\!\!\dir}$  are two independent copies of Brownian local time \cite[Theorem 8.1]{Busa-Sepp-Sore-22arXiv}. 

 For $a>0$ and $\dir\in\R$ let 
\[  \tau_{a,\dir}=\inf\{ x>0:   J_{\dir}(x) \vee [-J_{\dir}(-x)]   
\ge a \}. \] 
  Thus $\tau_{a,\dir}<\infty$ iff $\dir\in\DLBusedc$.  


    Define the following point measure on $\R\times\R_+\times C(\R)$: 
\be\label{W68} 
\Lambda_a=\sum_{ \dir\tsp\in\tsp\DLBusedc} \delta_{(\dir\,, \, \tau_{a,\dir}\,,\, J_\dir)} .  
\ee
The moment bound  in Theorem \ref{thm:SH10}\ref{itm:exp} implies that $\Lambda_a$ is almost surely a locally finite point measure.   We regard it as an element of the space $\M(\R\times\R_+\times C(\R))$ of locally finite Borel measures on $\R\times\R_+\times C(\R)$. This space is endowed with its Polish vague 
topology \footnote{The vague topology is defined by integration against bounded continuous test functions with bounded support. This is the terminology of Kallenberg \cite{kall-17-book}. Daley and Vere-Jones \cite{dale-vere-08}  reserve the term vague topology for locally compact spaces. In their language, $\M(\R\times\R_+\times C(\R))$ is a space of boundedly finite Borel measures with the $w^\#$ (weak-hash) topology.}.  
Since our result involves only the distribution of $\Lambda_a$, the choice of $t$ is immaterial and we omit it from the notation.

The connection between the TASEP speed process and the DL Busemann process goes through SH. 
For each $t \in \R$,  the following equality in distribution holds between random elements of the Skorokhod space $D(\R,C(\R))$:
\be\label{WG} 
\{\W_{\dir +}(\aabullet,t;0,t)\}_{\dir \in \R} \deq \bigl\{G_{\dir}(\aabullet) \bigr\}_{\dir \in \R},
\ee
where $G$ is the version of the stationary horizon described in Section \ref{sec:SH} and Appendix~\ref{sec:stat_horiz},   with diffusivity $\sqrt 2$ and drifts $2\xi$. In particular, in Theorem \ref{thm:conv} we can replace the limit $G$ with $\{\W_{\dir +}(\aabullet,t;0,t)\}_{\dir \in \R}$. It is in this sense that 	Theorem \ref{thm:conv} yields Theorem \ref{thm:Xi} below as a corollary. 


\subsection{Scaled and centered speed process values and their convoys} 

 We turn to discuss the approximating objects from the speed process.   Following \cite{Amir_Angel_Valko11}, for a given $m\in\Z$, call the index set $\C_m=\{ i\in\Z:  U_i=U_m\}$ the \textit{convoy} 
 of $U_m$.    By \cite[Theorem 1.8]{Amir_Angel_Valko11}, conditional on the value of $U_m$,  $\C_m-m$ is  bi-infinite and in fact a zero-density  renewal process.  
 
  Fix a centering $v\in(-1,1)$.    Define the difference function 
  \be\label{Jv} 
  J^{v, N}_{\mu}(x)  =  H^{v, N}_{\mu}(x) - H^{v, N}_{\mu-}(x), \qquad \mu,x\in\R. 
\ee  
 Let $\Xi^{v, N}$ be the set of   jump locations of $H^{v, N}_\bbullet$, in other words, the set of $\mu$ such that $J^{v, N}_{\mu} $ is not the identically zero function: 
   \be\label{Xi800} \begin{aligned}
\Xi^{v, N}&=\{  \mu\in\R:   H^{v, N}_{\mu-}\ne  H^{v, N}_{\mu}\} 
 =\bigl\{  N^{1/2}\tfrac{U_m-v}{1-v^2}:  m\in\Z  \bigr\}  . 
\end{aligned}\ee
The definition \eqref{H138} of $H^{v, N}_\mu$  shows  that  $\mu$ is a jump point of $H^{v, N}_{\bbullet}$ iff $U_m= v+(1-v^2)\mu N^{-1/2}$
for some $m\in\Z$.  This gives the  second equality above.  
Thus either $\mu\notin\Xi^{v, N}$ in which case $J^{v, N}_{\mu} $ is identically zero, or 
  \be\label{Jv3} 
  \text{for }  \mu=N^{1/2}\tfrac{U_m-v}{1-v^2} \,, \quad   J^{v, N}_{\mu}(x)  =    
   \begin{cases} 
   2N^{-1/2}  \ddd\sum_{i=0}^{{2x}{(1-v^2)^{-1}}N-1}
   \tsp\ind_{U_i= U_m}, &x\ge 0  \\[8pt] 
  -2N^{-1/2}  \ddd\sum_{i={2x}{(1-v^2)^{-1}}N}^{-1}  
   \tsp\ind_{U_i= U_m}, &x<0. 
  \end{cases}   
\ee
The sums on the right are exact only when the summation limits are integers. Otherwise the precise formula requires  the interpolation done in \eqref{eqn:TASEPh}.   The point is to illustrate that a nonzero function $J^{v, N}_{\mu}$ is the continuously interpolated cumulative convoy of the speed process value $v+(1-v^2)\mu N^{-1/2}\in\{U_i:i\in\Z\}$. 



By definition,  $\Xi^{v, N}$ is   a  \textit{set} and not a sequence indexed by $m$.  No repetition among the elements of $\Xi^{v, N}$ is intended, even though  every particular member of the second  formulation in \eqref{Xi800} appears  for infinitely many  distinct $m$-values.  

$\Xi^{v, N}$ is a dense subset of the interval $[-N^{1/2}\tfrac{1+v}{1-v^2}, \,N^{1/2}\tfrac{1-v}{1-v^2}]$  and hence not suitable as the support of a random point measure. To remedy this we add a second component to each point  that distributes the points  sparsely enough across a half-plane. 
 For $a>0$ define 
\be\label{Xi808} 
\sigma^{v, N}_{\mu,a}=\inf\bigl\{ x>0:   [H^{v, N}_{\mu}(x) - H^{v, N}_{\mu-}(x)] \vee [H^{v, N}_{\mu-}(-x) - H^{v, N}_{\mu}(-x)] \, \ge \,  a\bigr\} .
\ee
In terms of the speed process,   in the same approximate sense  as in \eqref{Jv3}, 
\be\label{Xi818}   
\text{for }  \mu=N^{1/2}\tfrac{U_m-v}{1-v^2}, \quad 
\sigma^{v, N}_{\mu,a} = \inf \Bigl\{  x>0:       \sum_{i=0}^{\frac{2x}{1-v^2}N-1} \!\!\!\!\ind_{U_i= U_m} 
\bigvee     \sum_{i=\frac{-2x}{1-v^2}N}^{-1} \!\!\!\!\ind_{U_i= U_m}\ge \tfrac12a N^{1/2}   \Bigr\} .
\ee 
 In particular, $\sigma^{v, N}_{\mu,a}<\infty$ iff $\mu\in\Xi^{v,N}$ iff $v+(1-v^2)\mu N^{-1/2}$ is among the speed process values $\{U_i:i\in\Z\}$. 
Define the following simple point measure on $\R\times\R_+\times C(\R)$: 
\be\label{Xi816}  \Lambda^{v, N}_a = \sum_{\mu\tspa\in\tspa\Xi^{v,N}} \delta_{(\mu,  \,\sigma^{v, N}_{\mu,a},\,   J^{v, N}_{\mu}  )} . \ee 
$ \Lambda^{v, N}_a$ is a locally finite point measure because  a bound $\sigma^{v, N}_{\mu,a}\le M$ bounds the number of terms in the sums in \eqref{Xi818}, and hence only finitely many distinct speed process values can appear.    As a measurable function of the speed process $\{U_i\}$,  $ \Lambda^{v, N}_a$ is  a random  element of the space $\M(\R\times\R_+\times C(\R))$. 

We can now state the theorem.  The limit measure $\Lambda_{a}$ is the one from \eqref{W68}. 

\begin{theorem}\label{thm:Xi}   Fix $v\in(-1,1)$.  Then for all 
$a>0$,  we have the distributional limit   $\Lambda^{v, N}_{a} \Rightarrow \Lambda_{a} $ as $N\to\infty$, in the vague topology of  the  space $\M(\R\times\R_+\times C(\R))$. 
\end{theorem} 
\begin{proof}
 Recall a basic fact of weak convergence: suppose that  $\mathcal X$ and $\mathcal Y$ are metric spaces,   $h:\mathcal X\to\mathcal Y$ is a Borel function with discontinuity set $\mathcal D$,  $X_n\Rightarrow X$ are $\mathcal X$-valued  random variables, and $P(X\in \mathcal D)=0$. Then $h(X_n)\Rightarrow h(X)$.

 We apply this fact to  the weak limit of  Theorem~\ref{thm:conv} and the point measures $\Lambda^{v, N}_{a}$ and $\Lambda_{a} $ as  functions on the path space. The auxiliary material used here is in Appendix \ref{sec:DSH}. First restrict the path space $D(\R, C(\R))$ to the smaller closed subspace $D_{SH}$ defined in \eqref{DSH7} that takes advantage of the monotonicity satisfied by the processes $H^{v, N}$ and $G$.  The distributions of $H^{v, N}$ and $G$ are supported by $D_{SH}$. Point measures  $\Lambda^{v, N}_{a}$ and $\Lambda_{a} $ are both instances of the general definition 
 \eqref{La87} on the space $D_{SH}$. Let $\mathcal D_a$ be the discontinuity set  of $\Lambda_a: D_{SH} \to\M(\R\times\R_+\times C(\R))$. The first inequality below comes from Lemma~\ref{lm:psi87} applied to $\mathcal Z=\Q$, the second comes because $\mu \mapsto G_\mu(q)$ is a pure jump process (Theorem~\ref{thm:SH10}\ref{itm:SH_j}), and the last equality comes because $G_{\mu_2}(q) - G_{\mu_1}(q)$ has no nonzero atoms (Theorem~\ref{thm:SH10}\ref{SHrelfect} and \ref{itm:SHdist}): 
\begin{align*}
\Pp(G \in \mathcal D_a) &\le \Pp\Big(\bigcup_{q\tsp\in\tsp\Q} \bigl\{\exists \mu\in\R \text{ such that }   |G_\mu(q)-G_{\mu-}(q)|=a\bigr\}\Big)\\
&\le \Pp\Big(\bigcup_{q\tsp\in\tsp\Q} \bigcup_{\substack{\mu_1 < \mu_2 \\\text{ both in }\Q}} \bigl\{\tspb |G_{\mu_2}(q)-G_{\mu_1}(q)|=a\tspa\bigr\}\Big) = 0. \qedhere
\end{align*}
\end{proof}

The theorem gives a precise meaning to the notion that the  scaled and centered speed process values approximate  the exceptional directions of DL and in the limit the convoys converge to Busemann difference profiles.   Since this theorem is inherited from Theorem \ref{thm:conv},  the choice of centering $v\in(-1,1)$  (again) vanishes in the limit.

\section{Finite-dimensional convergence} \label{sec:fdd}

We turn to the proof of Theorem~\ref{thm:conv}.  

\subsection{The space $D(\R,C(\R))$}\label{sec:D(R)}	

	
	$C(\R)$ is  the space of continuous functions on the real line equipped with the complete separable  metric
	\begin{align} \label{d}
		d(f,g)=\sum_{n=1}^{\infty}2^{-n}\frac{d_n(f,g)}{1+d_n(f,g)}
	\end{align} 
where \begin{align} \label{dn}
		d_n(f,g)=\sup_{x\in[-n,n]}|f(x)-g(x)|.
	\end{align} 	
	  Since    $d_n(f,g)\leq d_{n+1}(f,g)$, we have the following useful bound: 
	\begin{align}\label{maprox}
		d(f,g)\leq d_n(f,g)+2^{-n} \quad \forall n\in\N.
	\end{align}
The space $D(\R,C(\R))$ is the space of cadlag functions $\R \to C(\R)$, equipped with Skorokhod topology.

We observe why the path $\mu\mapsto H^N_\mu$ defined in \eqref{H138} lies in $D(\R,C(\R))$. Restriction of $x\mapsto H^N_\mu(x)$ to a  bounded interval $[-x_0, x_0]$ is denoted by  $H^{N,x_0}_\mu=H^N_\mu\vert_{[-x_0, x_0]}$.  Then note that for $\mu<\rho$, $H^{N,x_0}_\mu\ne H^{N,x_0}_\rho$	if and only if $U_j\in(v + \mu(1 - v^2) N^{-1/2},v  + (1 - v^2)\rho N^{-1/2}]$ for some  
$j\in\lzb\,\fl{-\,\f{2x_0}{1 - v^2} N}, \ce{\f{2x_0}{1 - v^2} N}\,\rzb$.  Since this range of indices is finite, for each $\mu \in \R$ and $x_0 > 0$ there exists $\ve>0$ such that $H^{N,x_0}_\mu= H^{N,x_0}_\rho$ for $\rho\in[\mu, \mu+\ve]$ and $H^{N,x_0}_\rho= H^{N,x_0}_\sigma$ for $\rho, \sigma\in[\mu-\ve, \mu)$. 

\subsection{Ferrari-Martin representation of multiclass measures} \label{sec:FM}

This section describes the queueing construction of stationary multiclass measures from \cite{Ferrari-Martin-2007}. We use the convention of \cite{Ferrari-Martin-2007} that TASEP particles jump to the left rather than to the right, because this choice leads to the more natural queuing set-up where time flows on $\Z$ from left to right. This switch is then accounted for when we apply the results of this section.  

\subsubsection{Queues with a single customer stream}
	Let $\Qs_1:=\{1,\infty\}^\Z$ 
	be the space of configurations of particles on $\Z$ with the following interpretation:  a configuration  $\bq{x}=\{x(j)\}_{j\in\Z}\in \Qs_1$  has a particle at time $j\in\Z$ if $x(j)=1$, otherwise   $\bq{x}$ has a hole at time $j\in\Z$.  
	Let $\arrv,\srvv\in\Qs_1$. Think of $\arrv$ as arrivals of customers to a queue, and of $\srvv$ as the available  services in  the queue. For $i\le j\in \Z$ let $a^{\leq1}[i,j]$  be the number of customers, that is, the number of $1$'s in $\bq{a}$, that  arrive to the queue during time interval $[i,j]$. Similarly let $s[i,j]$ be the number of services available  during time interval $[i,j]$. The queue length at time $i$ is then given by 
	\begin{equation}\label{Q}
		Q_i=\sup_{j: j\leq i}\big(a^{\leq1}[j,i]-s[j,i]\tspb\big)^+.
	\end{equation}
In principle this makes sense for arbitrary sequences $\arrv$ and $\srvv$ if one allows infinite queue lengths $Q_i=\infty$. However, in our treatment $\arrv$ and $\srvv$ are always such that 
queue lengths are finite.  We will not repeat this point in the sequel.

	The departures from the queue come from the mapping $\depv=D(\arrv,\srvv):\Qs_1\times\Qs_1 \rightarrow \Qs_1$,  given by 
	\begin{equation}\label{eq6}
		\dep(i)=
		\begin{cases}
		1 & \text{$s(i)=1$ and either $Q_{i-1}>0$ or $a(i)=1$},\\
		\infty & \text{otherwise}.
		\end{cases}
	\end{equation}
	In other words, a customer leaves the queue at time $i$ (and $d(i)=1$) if there is a service  at time $i$ and either the queue is not empty or a customer just arrived at time $i$.
			The sequence $\bq{u}:=U(\arrv,\srvv)$ of unused services  is given by a mapping  $U:\Qs_1\times\Qs_1\rightarrow \Qs_1$  defined by 
	\begin{equation}\label{defU}
	u(j)=
	\begin{cases}
	1 & \text{if $\srv(j)=1$,  $Q_{j-1}=0$, and $\arr(j)=\infty$},\\
	\infty & \text{otherwise}.
	\end{cases}
	\end{equation}
	Last, we define the map $R:\Qs_1\times\Qs_1\rightarrow \Qs_1$ as $\bq{r}=R(\arrv,\srvv)$ with 
	\begin{equation}\label{defR} 
		r(j)=
		\begin{cases}
		1 & \text{ if either $a(j)=1$ or $u(j)=1$},\\
		\infty & \text{ if $a(j)=u(j)=\infty$}.  
		\end{cases}.
	\end{equation}

Extend the departure operator $D$ to  queues in tandem.   Let  $D^1(\bq{x})=\bq{x}$ be the identity, and for $n\ge 2$,  
\begin{equation}\label{eq18}
\begin{aligned}
D^2(\bq{x}_1,\bq{x}_2)&=D(\bq{x}_1,\bq{x}_2)\\
D^3(\bq{x}_1,\bq{x}_2,\bq{x}_3)&=D\big(D^2(\bq{x}_1,\bq{x}_2),\bq{x}_3\big)\\
&	\ \; \vdots \\
D^n(\bq{x}_1,\bq{x}_2,\dotsc,\bq{x}_n)&=D\big(D^{n-1}(\bq{x}_1,\bq{x}_2,\dotsc,\bq{x}_{n-1}),\bq{x}_n\big).
\end{aligned}
\end{equation}
We may omit the superscript and simply write $D(\bq{x}_1,\bq{x}_2,\dotsc,\bq{x}_n)$.

\medskip


	\subsubsection{Queues with priorities}
	Now consider queues with customers of different classes.  For $m\in\N$, let $\Qs_m:=\{1,2,\dotsc,m,\infty\}^\Z$ be the space of configurations of particles on $\Z$ with classes in $\lzb1,m\rzb=\{1,2,\dotsc,m\}$. A lower label indicates higher class and, as before,  the value $\infty$ signifies an empty time slot.  To illustrate the notation for  an arrival sequence $\arrv\in \Qs_m$,   the  value  $a(j)=k\in\lzb1,m\rzb$ means that    a customer of class $k$ arrives at time $j\in\Z$,  while   $a(j)=\infty$ means   no  arrival  at time $j$. Define 
	\begin{equation*}
		a^{\leq k}[j]=
		\begin{cases}
		1 &\text{if $a(j)\leq k$},\\
		0 &\text{if $a(j)>k$}.
		\end{cases}
	\end{equation*}
	 Consistently with earlier definitions,  $a^{\leq k}[i,j]=\sum_{l=i}^{j}a^{\leq k}[l]$ is the number of customers in classes $\lzb1,k\rzb$ that arrive to the queue in the time interval $[i,j]$. Let $\srvv\in \Qs_1$ be the sequence of available services. The number of customers in classes $\lzb1,k\rzb$ in the queue at time $i$ is then 
	\begin{equation*}
		Q^{\leq k}_i(\arrv,\srvv)=\sup_{j: j\leq i}\big(a^{\leq k}[j,i]-s[j,i]\big)^+, \qquad i\in \Z.
	\end{equation*}
	 The multiclass departure map $\depv = F_m(\arrv,\srvv):\Qs_m\times \Qs_1 \rightarrow \Qs_{m+1}$ is defined so that  customers of  higher class (lower label) are served first.  These are the rules: 
	\begin{equation}\label{fm}
		\begin{cases}
		d(i)\leq k & \text{for $k\in\lzb1,m\rzb$ if $s(i)=1$ and either $Q^{\leq k}_{i-1}>0$ or $a(i)\leq k$}, \\
		d(i)=m+1 & \text{if $s(i)=1$,  $Q^{\leq m}_{i-1}=0$, and $a(i)= \infty$},\\
		d(i)=\infty &\text{if }  s(i)=\infty .
		\end{cases}
	\end{equation}
	The map $F_m$ works as follows.  The queue is fed with arrivals $\arrv\in\Qs_m$ of customers in classes $1$ to $m$. Suppose a service is available at time $i\in\Z$ ($s(i)=1$). Then  the customer of the highest class (lowest label in $\lzb1,m\rzb$)   in the queue at time $i$,  or  just arrived at time $i$, is served at time $i$, and its label becomes the value of $d(i)$. If no customer arrived at time $i$ ($a(i)=\infty$) and the queue is empty ($Q^{\leq m}_{i-1}=0$), then the unused service  $s(i)=1$ is converted into a departing  customer of class $m+1$:  $\dep(i)=m+1$. If there is no  service available at time $i\in\Z$ ($s(i)=\infty$), then no customer leaves at time $i$ and  $\dep(i)=\infty$.
	
	In particular, for $m=1$, the output    $\depv=F_1(\arrv,\srvv)$ satisfies 
	\be\label{Q670} 
	d(i)=\begin{cases} 1, &D_i(\arrv,\srvv)=1\\ 2, &U_i(\arrv,\srvv)=1 \\ \infty, &s(i)=\infty.   \end{cases}
	\ee
 	
	For  $n\in\N$  define the space  $\ml_n=\Qs_1^n=\{1,\infty\}^{\Z\times \{1,\dotsc,n\}}$ of $n$-tuples of sequences. Let $\bar{\lambda}=(\lambda_1,\dotsc,\lambda_n)\in(0,1)^n$ be a parameter vector such that  $\sum_{r=1}^n \lambda_r\le1$. Define  the product measure $\nu^{\bar{\lambda}}$ on  $\ml_n$ so that if   $\bar{\bq{x}}=(\bq{x}_1,\dotsc,\bq{x}_n)\sim\nu^{\bar{\lambda}}$ then the sequences  $\bq{x}_k$ are independent and  each   $\bq{x}_k$ has the i.i.d.\ product  Bernoulli distribution  with intensity $\sum_{i=1}^{k}\lambda_i$.  From this input we define a new process $\bar{\bq{v}}=(\bq{v}_1,\dotsc,\bq{v}_n)$ such that each $\bq{v}_m\in \Qs_m$   by the iterative  formulas 
	\begin{equation}\label{eq1}
		\begin{aligned}
		\bq{v}_1&=\bq{x}_1 \qquad\text{and} \\
		\bq{v}_m&=F_{m-1}(\bq{v}_{m-1},\bq{x}_m) \qquad \text{for } \ m= 2, \dotsc,n.
		\end{aligned}
	\end{equation}
	We denote this map  by $\bar{\bq{v}}=\Vmap(\bar{\bq{x}})=(\Vmap_1(\bar{\bq{x}}),\dotsc,\Vmap_n(\bar{\bq{x}}))$. For a vector $\bar{\lambda}=(\lambda_1,\dotsc,\lambda_n)$ and $\bar{\bq{x}}\sim \nu^{\bar{\lambda}}$, define the distribution  $\mu^{\bar{\lambda}}$ as the image of $\nu^{\bar{\lambda}}$ under this map:  
	\begin{equation}\label{mu}
\mu^{\bar{\lambda}}=\nu^{\bar{\lambda}}\circ \Vmap_n^{-1}  \ \iff \ 	\Vmap_n(\bar{\bq{x}})\sim 	\mu^{\bar{\lambda}}.
	\end{equation}
	\begin{theorem}[\cite{Ferrari-Martin-2007}\label{thm:FM}, Theorem 2.1]  
 For each $m\in\lzb1,n\rzb$, 
		the distribution  of $\bq{v}_m$   under $\mu^{\bar{\lambda}}$ is 
    the unique translation-ergodic
  stationary distribution of the $m$-type TASEP on $\Z$ with leftward jumps and  with density $\lambda_r$ of particles of class $r\in\lzb1,m\rzb$. 
   The distribution of the reversed configuration $\{\bq{v}_m(-i)\}_{i \in \Z}$ is the unique  distribution $\widecheck\mu^{\bar\lambda}$ described in Theorem \ref{thm:TASEPprod}, in other words, the unique translation-ergodic stationary distribution of the $m$-type TASEP on $\Z$ with rightward  jumps, with density $\lambda_r$ of particles of class $r \in \lzb 1,m \rzb$. 
\end{theorem}
\begin{remark}
The statement about the TASEP with rightward jumps is not included in \cite{Ferrari-Martin-2007}, but its proof is straightforward. Reflecting the index does not change the density of the particles, so the values $\lambda_r$ are preserved. Consider an $m$-type TASEP with left jumps $\{\eta_t\}_{t \ge 0}$ defined by the  Poisson clocks  $\{\cN_i\}_{i \in \Z}$ and started from initial profile $\eta_0 \sim \bq{v}_m$. Let $\{\widecheck \eta_t\}_{t \ge 0}$ be TASEP with right jumps defined by the   Poisson clocks $\{\cN_{-i}\}_{i \in \Z}$ and started from initial profile $\{\widecheck \eta_0(i)\}_{i \in \Z} := \{\eta_0(-i)\}_{i \in \Z}$, which has distribution $\{\bq{v}_m(-i)\}_{i \in \Z}$. Then, in the process $\eta_t$ a particle jumps from site $i$ to site $i - 1$ exactly when a particle in the process $\widecheck \eta_t$ jumps from site $-i$ to site $-i + 1$. By the invariance of $\eta_0$ under TASEP with left jumps,
\[
\{\widecheck \eta_t(i)\}_{i \in \Z} = \{\eta_t(-i)\}_{i \in \Z} \deq \{\eta_0(-i)\}_{i \in \Z} = \{\widecheck \eta_0(i)\}_{i \in \Z}, 
\]
so $\{\bq{v}_m(-i)\}_{i \in \Z}$ is the invariant measure for TASEP with right jumps and densities $\lambda_r$. 
\end{remark}

	For  $\bq{x}\in\Qs_n$, define 
\begin{equation*}
\Cls^{[i,j]}_m(\bq{x})=\#\{l\in [i,j]:{x}(l)\leq m\}, \qquad m\in\lzb1,n\rzb.
\end{equation*}  
$\Cls^{[i,j]}_m(\bq{x})$  records the number of customers in classes  $\lzb1,m\rzb$   during time interval  $[i,j]$  in the sequence  $\bq{x}$. Note that a customer of class $m$ appears in $[i,j]$ iff $\Cls^{[i,j]}_m(\bq{x})>\Cls^{[i,j]}_{m-1}(\bq{x})$, with the convention $\Cls_0\equiv0$. The key technical lemma is that the iteration in \eqref{eq1} can be represented by tandem queues.

\begin{lemma}\label{lem:Cl}
Let $n\in\N$ and $\bar{\bq{x}}=(\bq{x}_1,\dotsc,\bq{x}_n)\in\ml_n$. 
	Let $\bq{v}_n=\Vmap_n(\bar{\bq{x}})$, where $\Vmap_n$ is given in \eqref{eq1}.   Define
	 \begin{equation*}
	\begin{aligned}
	\depv^{n,i}:=D(\bq{x}_{i},\bq{x}_{i+1},\dotsc,\bq{x}_n) \quad \text{for } \ i=1,\dotsc,n-1,  \quad \text{and} \quad 
	\depv^{n,n}:=\bq{x}_n.
	\end{aligned}
	\end{equation*}
	Then for all time intervals $[i,j]$, 
	\begin{equation}\label{ind}
		\big(\Cls^{[i,j]}_1(\bq{v}_n),\dotsc,\Cls^{[i,j]}_n(\bq{v}_n)\big)
		=  \big(d^{n,1}[i,j],\dotsc,d^{n,n}[i,j]\big).
	\end{equation}
\end{lemma}
\begin{proof}
	The proof goes by induction on $n$, with  base case $n=2$. From   \eqref{Q670}, $\depv^{2,1}=D(\bq{x}_1,\bq{x}_2)$ registers the first class departures out of the queue $F_1(\bq{x}_1,\bq{x}_2)$ while $\depv^{2,2}=\bq{x}_2$ is the combined number of first and second class customers coming out of the queue.  The case $n=2$ of \eqref{ind}  has been verified.

	Assume   \eqref{ind} holds for some $n=k\ge2$. 	This means that  for each $m\in\lzb1,k\rzb$, $\depv^{k,m}$ registers the   customers in classes $\lzb1,m\rzb$ in $\bq{v}_k$. In the next step, $\bq{v}_{k+1}=F_k(\bq{v}_k, \bq{x}_{k+1})$ and 
	\begin{equation*}
		\big(\depv^{k+1,1},\dotsc,\depv^{k+1,k+1}\big)=\big(D(\depv^{k,1},\bq{x}_{k+1}),\dotsc,D(\depv^{k,k},\bq{x}_{k+1}),\bq{x}_{k+1}\big). 
	\end{equation*}
	Since the same service process $\bq{x}_{k+1}$ acts in both queuing maps, the outputs match in the sense that for each $m\le k$,  $\depv^{k+1,m}=D(\depv^{k,m},\bq{x}_{k+1})$ registers the customers in classes $\lzb1,m\rzb$ in $\bq{v}_{k+1}$.   
	Under $F_k(\bq{v}_k, \bq{x}_{k+1})$, unused services become departures of class $k+1$. Hence  every service event of $\bq{x}_{k+1}$ becomes a departure of some class in $\lzb1,k+1\rzb$. This verifies the equality $\Cls^{[i,j]}_{k+1}(\bq{v}_{k+1})={x}_{k+1}[i,j]=d^{k+1,k+1}[i,j]$ of the last coordinate. Thereby the validity of \eqref{ind} has been extended from $k$ to $k+1$. 
\end{proof}

\subsection{Convergence of queues}

\medskip \noindent 
This section shows the finite-dimensional weak  convergence of the TASEP speed process,  using the representation of stationary distributions in terms of  queuing mappings. To do this, we derive a convenient representation for the random walk defined by the departure mapping $D$ (Equation~\eqref{DQ}). Consistently with the count notation $x[i,j]$ introduced above for  $\bq x \in \Qs_1$, abbreviate 
   $x[i] =  x[i,i]=\ind_{x(i) = 1}$. With this convention,  configurations $\bq{x}$ can also be thought of as members of the sequence space  $\{0,1\}^\Z$. Recall the operation $\mathcal 
 P$ from~\eqref{eqn:TASEPh}.
   \begin{lemma} 
   For $i \in \Z$,
   \be \label{DQ}
    \mathcal P[D(\arrv,\srvv)](i) = \mathcal P[\srvv](i) + \sup_{-\infty < j \le 0}[\mathcal P[\srvv](j) -\mathcal P[\arrv](j)] - \sup_{-\infty < j \le i}[\mathcal P[\srvv](j) -\mathcal P[\arrv](j)].
   \ee
   \end{lemma}
  \begin{proof}
      Recall the definition of $D$ from~\eqref{eq6}. Observe that
      \begin{equation} \label{gen:Qa}
    D(\arrv,\srvv)[j,i]=Q_{j-1}-Q_i+a[j,i],
\end{equation}
because any arrival that  cannot be accounted for in $Q_i$ must have left by time $i$.
Use also the empty interval convention  $x[i + 1,i] = 0$. Then, from~\eqref{Q}, we can equivalently write 
   \be \label{Qialt}
   Q_i = \sup_{j:j \le i + 1}\bigl(a^{\le 1}[j,i] - s[j,i]\bigr). 
   \ee
Now, observe that for $\bq{x} \in \Qs_1$,
\be \label{xincs}
2x[j,i] - (i - j + 1) = \sum_{k = j}^i (2x[k] - 1) = \mathcal P[\bq{x}](i + 1) - \mathcal P[\bq{x}](j)
\ee
Combining~\eqref{gen:Qa}--\eqref{xincs} and the definition  $\mathcal P[\bq{x}](0) = 0$, gives for $i > 0$,
\begin{align*}
\mathcal P[D(\arrv,\srvv)](i) &= \sum_{k = 0}^{i - 1} (2D(\arrv,\srvv)[k] - 1) = 2D(\arrv,\srvv)[0,i-1] - i \\
&=  2a[0,i-1] - i + 2Q_{-1} - 2Q_{i -1} \\
&\overset{\eqref{xincs}}{=} \mathcal P[\arrv](i) - \mathcal P[\arrv](0) + 2\sup_{-\infty < j \le 0}[a^{\le 1}[j,-1] - s[j,-1]] - 2\sup_{-\infty < j \le i }[a^{\le 1}[j,i-1] - s[j,i-1]] \\
&= \mathcal P[\arrv](i) + \sup_{-\infty < j \le 0}[2a^{\le 1}[j,-1] +j - (2s[j,-1] + j)] \\
&\qquad\qquad - \sup_{-\infty < j \le i}[2a^{\le 1}[j,i-1] - (i - j) - (2s[j,i-1] - (i - j))] \\
&\overset{\eqref{xincs}}{=} \mathcal P[\arrv](i) + \sup_{-\infty < j \le 0}[\mathcal P[\arrv](0) - \mathcal P[\arrv](j) - \mathcal P[\srvv](0) + \mathcal P[\srvv](j) ] \\
&\qquad\qquad\qquad - \sup_{-\infty < j \le i}[\mathcal P[\arrv](i) - \mathcal P[\arrv](j) - \mathcal P[\srvv](i) + \mathcal P[\srvv](j) ] \\
&= \mathcal P[\srvv](i) + \sup_{-\infty < j \le 0}[\mathcal P[\srvv](j) -\mathcal P[\arrv](j)] - \sup_{-\infty < j \le i}[\mathcal P[\srvv](j) -\mathcal P[\arrv](j)].
\end{align*}
The case $i < 0$ follows an analogous proof. 
  \end{proof} 
   




\begin{proposition}\label{prop:fdd}   Fix the centering $v \in (-1,1)$. 
Then the scaled TASEP speed process  $H^N$ of  \eqref{H138} satisfies the weak convergence 
$(H_{\mu_1}^N,\dotsc,H_{\mu_k}^N)$ $\Rightarrow$ $(G_{\mu_1},\dotsc,G_{\mu_k})$  on $C(\R)^k$ for any finite sequence 
$(\mu_1,\dotsc,\mu_k)\in \R^k$.  
\end{proposition}
\begin{proof}

Without loss of generality, take $\mu_1 < \mu_{2} < \cdots < \mu_k$.  
For  $N > |\mu_1|^{3} \vee |\mu_k|^{3}$, consider the following nondecreasing map $F:[-1,1] \to \{1,\ldots,k\} \cup \{\infty\}$:
 \[
 \text{For } U \in [-1,1],\; F(U) =
 \begin{cases}
 1, & U \le v + \mu_1 (1 - v^2) N^{-1/2} \\
 2,& v + \mu_1(1 - v^2) N^{-1/2} < U \le v + \mu_2(1 - v^2) N^{-1/2} \\
 \vdots & \\
 k, &v + \mu_{k - 1}(1 - v^2)N^{-1/2} < U \le v +  \mu_k(1 -v^2) N^{-1/2} \\
 \infty, &U > v + \mu_k(1 -v^2) N^{-1/2}.
\end{cases}
 \]
 By considering the output of this map as classes, Lemma~\ref{lem:FU} implies that $\{F(U_i)\}_{i \in \Z}$
 is distributed as the stationary distribution for $k$-type TASEP with right jumps and densities
 \[
 \bar \lambda = \biggl(\f{1 + v + \mu_1(1 - v^2) N^{-1/2}}{2},\f{(\mu_2 - \mu_1)(1 - v^2)N^{-1/2}}{2},\ldots,\f{(\mu_{k} - \mu_{k-1})(1 -v^2)N^{-1/2}}{2}\biggr) \;\in\;(0,1)^k. \] 
 The reflection of Theorem~\ref{thm:FM} and translation invariance then imply that $\{F(U_{-i - 1})\}_{i \in \Z}$
 has the stationary distribution $\mu^{\bar\lambda}$ for TASEP with left jumps.    
Lemma~\ref{lem:Cl} implies that
 \be \label{bercoup}\begin{aligned} 
&\bigl(\ind_{U_{-i - 1} \le v + \mu_1(1 - v^2) N^{-1/2}},\ldots,  \ind_{U_{-i-1} \le v + \mu_{k-1} (1 -v^2) N^{-1/2}}\,, \ind_{U_{-i-1} \le v + \mu_k (1 -v^2) N^{-1/2}}\bigr)_{i \in \Z}\\
&\qquad\qquad\qquad
\deq \bigl(D(\bq{x}_1^N,\ldots,\bq{x}_k^N)[i]\,,\ldots,D(\bq{x}^N_{k - 1},\bq{x}_k^N)[i]\,,\bq{x}_k^N[i]\bigr)_{i \in \Z},
\end{aligned} \ee
 where $(\bq{x}_1^N,\ldots,\bq{x}_k^N) \sim \nu^{\bar \lambda}$.
 Remark \ref{rmk:L/R8} at the end of the section gives an alternative way to justify the index reversal on the left-hand side above when $v = 0$.  

Before proceeding with the proof, we give a roadmap. First, by definition of $\mathcal P$,  if $\widecheck{U}_i = U_{-i - 1}$, then for $x \in \R$,
\be \label{chU}
\begin{aligned}
H_{\mu}^N(x) &=  N^{-1/2}\mathcal P[\ind_{U \le v + \mu(1 - v^2) N^{-1/2}}]\Bigl(\f{2x}{1 - v^2}N\Bigr) -\f{2vx}{1 -v^2}N^{1/2} \\
&= -N^{-1/2}\mathcal P[\ind_{\widecheck U \le v + \mu(1 - v^2) N^{-1/2}}]\Bigl(-\f{2x}{1-v^2}N\Bigr) - \f{2vx}{1- v^2}N^{1/2}.
\end{aligned}
\ee
Our goal is to show the weak limit 
\be \label{convflip}
\begin{aligned}
&\biggl(-N^{-1/2}\mathcal P[\bq{x}_k^N]\Bigl(\f{2\tsp\aabullet}{1 - v^2}N\Bigr) + \f{2v\tsp\aabullet}{1 - v^2}N^{1/2}\,, \\[3pt]
&\qquad -N^{-1/2}\mathcal P[D(\bq{x}^N_{k-1},\bq{x}_k^N)]\Bigl(\f{2\tsp\aabullet}{1 - v^2}N\Bigr) + \f{2v\tsp\aabullet}{1 - v^2}N^{1/2},\ldots, \\[3pt] 
&\qquad
-N^{-1/2}\mathcal P[D(\bq{x}_1^N,\ldots,\bq{x}_k^N)]\Bigl(\f{2\tsp\aabullet}{1 - v^2}N\Bigr) + \f{2v\tsp\aabullet}{1 - v^2}N^{1/2}\biggr) \\
&\qquad\qquad  
\Longrightarrow (G_{-\mu_k},\ldots, G_{-\mu_1}).
\end{aligned}
\ee
From \eqref{convflip},  \eqref{bercoup} and~\eqref{chU} follows $(H_{\mu_1}^N,\ldots,H_{\mu_k}^N) \Longrightarrow (G_{-\mu_1}(-\tsp\aaabullet),\ldots, G_{-\mu_k}(-\tsp\aaabullet))$. This limit has the same distribution as $(G_{\mu_1},\ldots,G_{\mu_k})$ by Theorem~\ref{thm:SH10}\ref{SHrelfect}. As mentioned previously,   these reflections   in the proof are a consequence of having   time flow left to right in  the queuing setting. 

We prove \eqref{convflip}.
 By construction, for $j \in \lzb 1,k\rzb$, $\{{x}^N_j[i]\}_{i \in \Z}$ is an i.i.d. Bernoulli sequence with intensity $\sum_{\ell \le j} \lambda_\ell = \f12{(1 +v +  \mu_j(1 -v^2)N^{-1/2})}$.  Hence, for $j \in \lzb 1,k \rzb$, 
 \be \label{xconv}
 -N^{-1/2} \mathcal P[\bq{x}_j^N]\Bigl(\f{2\aabullet}{1 - v^2}N\Bigr) + \f{2v\tsp\aabullet}{1 - v^2}N^{1/2}
 \ee
 converges in distribution  to a Brownian motion with diffusivity $\sqrt 2$ and drift $-2\mu_j$. 
 To elevate this to the  joint convergence of~\eqref{convflip}, we utilize the queueing mappings in \eqref{bercoup} and the transformations $\Phi^k$ from Appendix  \ref{sec:stat_horiz} that construct SH.

By Skorokhod representation (\cite[Thm.~11.7.2]{dudl}, \cite[Thm.~3.1.8]{ethi-kurt}), we may couple $\{\bq{x}_j^N\}_{j = 1,\ldots,k}$ and independent Brownian motions $\{B_j\}_{j \in \lzb 1,k \rzb}$ with diffusivity $\sqrt 2$ and drift $-2\mu_j$ so that, with probability one, for $j  \in \lzb 1,k \rzb$,~\eqref{xconv}
converges uniformly on compact sets to $B_j$. Let $\Pp$ be the law of this coupling. (To be precise, the sequences  $\bq{x}_j^N$ are  functions of the  converging processes~\eqref{xconv}, which  Skorokhod representation couples with their  limiting  Brownian motions.)


By Appendix  \ref{sec:stat_horiz}, for reals $\mu_1  < \cdots < \mu_k$, the $C(\R)^k$-valued marginal  $(G_{-\mu_k},\dotsc,G_{-\mu_1})$ of SH can be constructed as follows:
\begin{align*}
 G_{-\mu_k}&=\Phi^1(B_k)=B_k, \qquad 
 G_{-\mu_{k-1}}=\Phi^2(B_k, B_{k-1})=\Phi(B_k, B_{k-1}),  \\
  G_{-\mu_{k-2}}&=\Phi^3(B_k, B_{k-1}, B_{k-2})=\Phi(B_k, \Phi(B_{k-1}, B_{k-2})), \dotsc, \\
  G_{-\mu_1}&=\Phi^k(B_k, B_{k-1}, \dotsc,  B_1) =  \Phi(B_k, \Phi^{k-1}(B_{k-1}, \dotsc,  B_1)). 
\end{align*}
The map $\Phi$ as defined   in~\eqref{Phialt} is given by 
\[
\Phi(f,g)(y) = f(y) + \sup_{-\infty <x \le y }\{g(x) - f(x)\} - \sup_{-\infty < x \le 0}\{g(x) - f(x)\}. 
\]
In particular, $\Phi(f,g)$ is a well-defined  continuous random function when $f$ and $g$ are Brownian motions and $f$ has a strictly  smaller drift than $g$. 

 By a union bound, it suffices to show that, under this coupling,  for each $\ve > 0, a > 0$, and each $j = 0,\ldots,k - 1$, 
\be \label{convPhi}\begin{aligned} 
\limsup_{N \to \infty}\Pp\biggl(\;\sup_{x \in [-a,a]}\Bigl|-N^{-1/2}\mathcal P[D(\bq{x}_{k - j}^N,\ldots, \bq{x}_k^N)] & \Bigl(\f{2 x}{1 - v^2}N\Bigr) + \f{2vx}{1 - v^2}N^{1/2}  \\
&- \Phi^{j+1}(B_k,\ldots,B_{k - j})(x)\Bigr| > \ve  \biggr) = 0,
\end{aligned} \ee
We show this by induction on $j$. The base case $j = 0$ follows by the almost sure uniform convergence on compact sets of~\eqref{xconv} to $B_j$. Now, assume the statement holds for some $j - 1 \in \{0,\ldots,k -2\}$. Recall from  definition~\eqref{eq18} that $D(\bq{x}_{k - j}^N,\ldots,\bq{x}_{k}^N) = D(D(\bq{x}_{k - j}^N,\ldots,\bq{x}_{k - 1}^N),\bq{x}_k^N)$. The proof is completed by Lemma~\ref{lem:convD} below.  
\end{proof}

 \begin{lemma} \label{lem:convD}
Let $v \in (-1,1)$. For each $N>0$, let $\arrv^N$ and $\srvv^N$ be $\{0,1\}^\Z$-valued i.i.d. sequences such that  the intensity of $\srvv^N$ is strictly greater than the intensity of $\arrv^N$. Assume further that these sequences are coupled together with Brownian motions $B_1,B_2$ with diffusivity $\sqrt 2$ and drifts $-2\mu_1 > -2\mu_2$ so that, for each $\ve > 0$ and $a > 0$,
\be \label{ipcon}
\begin{aligned}
\limsup_{N \to \infty} &\Pp\Biggl(\sup_{x \in [-a,a]}\Bigl|-N^{-1/2} \mathcal P[\arrv^N]\Bigl(\f{2x}{1 - v^2}N\Bigr) + \f{2vx}{1 - v^2}N^{1/2} - B_1(x) \Bigr| \\
&\qquad\qquad\qquad
\vee \Bigl|-N^{-1/2} \mathcal P[\srvv^N]\Bigl(\f{2x}{1 - v^2}N\Bigr) + \f{2vx}{1 - v^2}N^{1/2} - B_2(x) \Bigr| > \ve \Biggr) = 0.
\end{aligned}
\ee
 Then, for every $\ve > 0$ and $a > 0$, 
\be \label{2ptprob}
\limsup_{N \to \infty} \Pp\Biggl(\sup_{x \in [-a,a]}\Bigl|-N^{-1/2} \mathcal P[D(\mbf a^N,\mbf s^N)]\Bigl(\f{2x}{1 - v^2}N \Bigr) + \f{2vx}{1 - v^2}N^{1/2} - \Phi(B_2,B_1)(x) \Bigr| > \ve  \Biggr) = 0.
\ee

 \end{lemma}
 \begin{proof}
 From~\eqref{DQ}, we have, for $2N x \in \Z$,
 \be \label{Prew}
 \begin{aligned}
 &-N^{-1/2} \mathcal P[D(\arrv^N,\srvv^N)]\Bigl(\f{2x}{1 - v^2}N\Bigr) + \f{2vx}{1 - v^2}N^{1/2} \\
 &=-N^{-1/2} \mathcal P[\srvv^N]\Bigl(\f{2x}{1 - v^2} N\Bigr) + \f{2vx}{1 -v^2}N^{1/2} \\
&\qquad\qquad + N^{-1/2}\sup_{-\infty < j \le 2Nx/(1 - v^2)}[\mathcal P[\srvv^N](j) - \mathcal P[\arrv^N](j)] 
- N^{-1/2}\sup_{-\infty < j \le 0}[\mathcal P[\srvv^N](j) - \mathcal P[\arrv^N](j)] \\
 &= -N^{-1/2} \mathcal P[\srvv^N]\Bigl(\f{2x}{1 - v^2}N\Bigr)  +\f{2vx}{1 -v^2}N^{1/2} \\
 &\qquad\qquad  
 + \sup_{-\infty < j \le 2Nx/(1 - v^2)}[- N^{-1/2} \mathcal P[\arrv^N](j)- (-N^{-1/2}\mathcal P[\srvv^N](j)) ] \\
&\qquad\qquad
-\sup_{-\infty < j \le 0}[- N^{-1/2} \mathcal P[\arrv^N](j)- (-N^{-1/2}\mathcal P[\srvv^N](j))], 
 \end{aligned}
 \ee 
Hence, from~\eqref{Prew} and the assumed convergence of $N^{-1/2} \mathcal P[\srvv^N](2N\tsp\aabullet)$ in probability~\eqref{ipcon}, to prove~\eqref{2ptprob}, it suffices to show that, for each $a > 0$ and $\ve > 0$,
\be\label{3004} \begin{aligned} 
\limsup_{N \to \infty}\Pp\biggl(\,\sup_{x \in [-a,a]}\Bigl|\sup_{-\infty < y \le [2N x/(1 - v^2)]}\bigl[- N^{-1/2} \mathcal P[\arrv^N](y)- &(-N^{-1/2}\mathcal P[\srvv^N](y )) \bigr] \\
&\quad - \sup_{-\infty < y \le x}[B_1(y) - B_2(y)]  \Bigr| > \ve\biggr) = 0.
\end{aligned} \ee
Note that there is a drift term for both the walks $\mathcal P[\arrv^N]$ and $\mathcal P[\srvv^N]$ that cancels when they are subtracted. 
For shorthand, let 
\[
X^N(y) = - N^{-1/2} \mathcal P[\arrv^N](y)- (-N^{-1/2}\mathcal P[\srvv^N](y )) 
\]
For the $a$ in the hypothesis of the lemma and arbitrary $S > a$, let $E_{N,a,S}$  be the event where these three  conditions all hold: 
\begin{enumerate}[label={\rm(\roman*)}, ref={\rm\roman*}]   \itemsep=3pt
\item $\sup_{-\infty < y \le [-2N a/(1 -v^2)]}\bigl[X^N(y) \bigr] = \sup_{[-2NS/(1 -v^2)]\le y \le [-2Na/(1 - v^2)]}\bigl[X^N(y) \bigr]$.
\item $\sup_{-\infty < y \le -a}[B_1(y) - B_2(y)] = \sup_{-S \le y \le -a}[B_1(y) - B_2(y)]$.
\item $\sup_{x \in [-a,a]}\Bigl|\sup_{[-2NS/(1 - v^2)] \le y \le [2N x/(1 - v^2)]} X^N(y) - \sup_{-S \le y \le x}[B_1(y) - B_2(y)] \Bigr| \le \ve$.
\end{enumerate}
For every $S > a$, the event in~\eqref{3004} is contained in $E_{N,a,S}^c$. By assumption~\eqref{ipcon} and Lemma~\ref{lem:convprob} (applied to the random walk $-\mathcal P[\arrv^N] + \mathcal P[\srvv^N]$ with $m = \mu_2 - \mu_1$, $\sigma = 2$, $\varphi(N) = N^{-1/2}, \xi(N) = 2N/(1 - v^2)$, and $B = B_1(\aabullet/4) - B_2(\aabullet/4)$),  $\lim_{S \to \infty} \limsup_{N \to \infty} \Pp(E_{N,a,S}^c) = 0$, completing the proof. 
\end{proof}

\begin{remark}\label{rmk:L/R8} 
For $v = 0$, one can alternatively arrive at \eqref{bercoup}  by   considering the speed process for TASEP with left jumps. As in \eqref{speed8}, let  $X_i(t)$ be the position of the right-going particle with label $i$ that starts at $X_i(0)=i$ and define the right-going speed process by $U_i=\lim_{t\to\infty} t^{-1}X_i(t)$.  To flip the space direction, define left-going particles $\wt X_i(t)=-X_{-i}(t)$ and the corresponding speed process $\wt U_i=\lim_{t\to\infty} t^{-1}\wt X_i(t)=-U_{-i}$. Reversing the lattice direction reversed the priorities of the labels (for the walks $\wt X$, lower label means lower priority), so the non-decreasing  projection $F$ to left-going multiclass stationary measures has to be applied to speeds $-\wt U_i=U_{-i}$. 

The distributional equality  $\{\wt U_i\}_{i\in\Z}\deq\{U_{i}\}_{i\in\Z}$ from \cite[Proposition 5.2]{Amir_Angel_Valko11}  implies that both speed processes have the same SH limit. This can also be verified by replacing $U$ with $\wt U$ in \eqref{H138}, rearranging, taking the limit, and using the reflection property Theorem \ref{thm:SH10}\ref{SHrelfect} of SH.
\end{remark}

\medskip 
  
	\section{Tightness of the scaled TASEP speed process} \label{sec:tightness}	

Throughout this section the centering $v\in(-1,1)$ in \eqref{H138} is fixed. We mostly omit  dependence on $v$ from the notation.  We show tightness of the process 
\be \label{Href}
\mu\mapsto \wt H_\mu^N(\aabullet) = -H_\mu^N(-\tsp\aaabullet),
\ee
whose tightness is equivalent to that of $H^N$. By~\eqref{bercoup} and~\eqref{chU}, the queuing setup of Section~\ref{sec:FM} applies directly to the distribution of this process.

\subsection{Modulus of continuity} 

Let $\bar{\lambda}=(\lambda_1,\dotsc,\lambda_4)$ and  $\bar{\bq{x}}\sim \nu^{\bar{\lambda}}$. Let $\bq{v}_4=\Vmap_4(\bar{\bq{x}})$. For integers $i\leq j$ define the event
\begin{equation*}
	\TwoRare^{[i,j]}(\bar{\lambda})=\bigl\{\ind_{\Cls^{[i,j]}_1(\bq{v}_4)<\Cls^{[i,j]}_2(\bq{v}_4)}+\ind_{\Cls^{[i,j]}_2(\bq{v}_4)<\Cls^{[i,j]}_3(\bq{v}_4)}+\ind_{\Cls^{[i,j]}_3(\bq{v}_4)<\Cls^{[i,j]}_4(\bq{v}_4)}\geq2\bigr\}.
\end{equation*}
$\TwoRare^{[i,j]}(\bar{\lambda})$ is the event that  $\bq{v}_4$ has customers of at least two 
different classes among the classes $\{2,3,4\}$ in the time interval $[i,j]$.  
The event itself does not depend on $\bar\lambda$ but we include $\bar\lambda$ in the notation to keep in mind  the parameters under which we are calculating. The next lemma will be useful when classes $\{2,3,4\}$ are rare.

\begin{lemma}\label{lem:TR} 
	 Fix $\bar{\lambda}=(\lambda_1,\dotsc,\lambda_4)\in(0,1)^4$ such that $\sum_{l=1}^{4}\lambda_l\le1$. Fix  integers $i\leq j$. Let $\lambda^*=\max\{\lambda_2,\lambda_3,\lambda_4\}$ and  $\Delta=\frac{\lambda^*}{(1-\sum_{l=1}^{3}\lambda_l)\lambda_1}$. Then for any integer  $\MTR\geq 2$ and $\rho\in(0,1)$, 
	\begin{equation}\label{tr}\begin{aligned} 
	&\Pp\big(\TwoRare^{[i,j]}(\bar{\lambda})\big)
	\; \leq \;  72\sum_{s=1}^\infty \big(1-(1-\Delta )^{s\MTR}\big)^2  \\
	&\qquad \qquad
	\times \exp\Bigg\{-\frac{2\Big(\Big[\tfrac {(s-1)\MTR}{24}-\big|\sum_{m=1}^4\lambda_m-\rho\tspb \big|\cdot(j-i+1)-\big|\lambda_1-\rho\tsp \big|\cdot(j-i+1)\Big]_+\,\Big)^2}{9(j-i+1)}\Bigg\}.
\end{aligned} 	\end{equation}
\end{lemma}

\begin{proof}
Lemma \ref{lem:Cl} gives this distributional equality: 
	\begin{align*}
			&\big(\Cls^{[i,j]}_1(\bq{v}_4),\Cls^{[i,j]}_2(\bq{v}_4),\Cls^{[i,j]}_3(\bq{v}_4),\Cls^{[i,j]}_4(\bq{v}_4)\big)\\[3pt]  
			&\qquad\qquad 
			\sim \big(D(\bq{x}_1,\bq{x}_2,\bq{x}_3,\bq{x}_4)[i,j],D(\bq{x}_2,\bq{x}_3,\bq{x}_4)[i,j],D(\bq{x}_3,\bq{x}_4)[i,j]
			,{x}_4[i,j]\big).
\end{align*} 
We reformulate the tandem queuing mappings above  by repeated applications of \eqref{q259}. 
\begin{equation}\label{Q3}
	I_3:=D(\bq{x}_3,\bq{x}_4), \,\,\,J_3:=R(\bq{x}_3,\bq{x}_4).  
\end{equation}
\begin{equation}\label{Q2}
	I_2:=D(\bq{x}_2,\bq{x}_3,\bq{x}_4)=D(\bq{x}_2,J_3,I_3), \,\,\, J_2:=R\big(D(\bq{x}_2,J_3),I_3\big). 
\end{equation}
\begin{equation*}
	\begin{aligned}
		I_1:&=D(\bq{x}_1,\bq{x}_2,\bq{x}_3,\bq{x}_4)
		=D(\bq{x}_1,\bq{x}_2,J_3,I_3)
		=D\big(\bq{x}_1,R(\bq{x}_2,J_3),D(\bq{x}_2,J_3),I_3\big)\\
	&=D\big(\bq{x}_1,R(\bq{x}_2,J_3),R\big(D(\bq{x}_2,J_3),I_3\big),D\big(D(\bq{x}_2,J_3),I_3\big)\big)=D\big(\bq{x}_1,R(\bq{x}_2,J_3),J_2,I_2\big). 
	\end{aligned}
\end{equation*}
Abbreviate the queue lengths produced by  these mappings  at time $i-1$ as follows: 
\[  
\Qu_3:=Q_{i-1}(\bq{x}_3,\bq{x}_4), \quad 
\Qu_2:=Q_{i-1}\big(D(\bq{x}_2,J_3),I_3\big)
\quad\text{and}\quad 
\Qu_1:=Q_{i-1}(D(\bq{x}_1,R(\bq{x}_2,J_3),J_2),I_2).
\] 
Next we express the event $\TwoRare^{[i,j]}(\bar{\lambda})$ in terms of the queue lengths   $(\Qu_1,\Qu_2,\Qu_3)$ and auxiliary walks that start at time $i$ and are defined for times $m\ge i$ as  follows: 
	\begin{align*}
	S_3(m)&:={x}_4[i,m]-{x}_3[i,m] \\
		S_2(m)&:=I_3[i,m]-D(\bq{x}_2,J_3)[i,m], 
		\\
		S_1(m)&:=I_2[i,m]-D\big(\bq{x}_1,R(\bq{x}_2,J_3),J_2\big)[i,m].
	\end{align*}
Write  $S^*_k(m):=\max_{i\leq l\leq m} S_k(l)$ for  the running maximum of the walk $S_k$.

First we claim that 
\begin{equation}\label{eq13}
	\{D(\bq{x}_3,\bq{x}_4)[i,j]
	<{x}_4[i,j]\}\subseteq \{\Qu_3<S^*_3(j)\}.
\end{equation}
To see this, note that the map  $(\bq{x}_3,\bq{x}_4)\mapsto D(\bq{x}_3,\bq{x}_4)[i,j]$  labels each service time in $\bq{x}_4$ as a departure or an unused service.  On the event on the left side of \eqref{eq13}, there are unused service times in ${x}_4[i,j]$. If $l\in[i,j]$ is such that ${x}_4(l)$ is an unused service, then the queue must have emptied prior to time $l$, and thereby $Q_{i-1}(\bq{x}_3,\bq{x}_4)+{x}_3[i,l]<{x}_4[i,l]$. This implies \eqref{eq13}. The same argument implies
\begin{equation*}
	\begin{aligned}
		&\{D(\bq{x}_2,\bq{x}_3,\bq{x}_4)[i,j]
	<D(\bq{x}_3,\bq{x}_4)[i,j]\}\subseteq \{\Qu_2<S^*_2(j)\}\\[3pt] 
\text{and }\quad 	& \{D(\bq{x}_1,\bq{x}_2,\bq{x}_3,\bq{x}_4)[i,j]
	<D(\bq{x}_2,\bq{x}_3,\bq{x}_4)[i,j]\}\subseteq \{\Qu_1<S^*_1(j)\}.
	\end{aligned}
\end{equation*}

The development up to this point gives us the following bound: 
\begin{equation}\label{eq14}
	\Pp\big(\TwoRare^{[i,j]}(\bar{\lambda})\big)\leq \Pp\big(\ind_{\Qu_1<S^*_1(j)}+\ind_{\Qu_2<S^*_2(j)}+\ind_{\Qu_3<S^*_3(j)}\geq 2\big).
\end{equation}
To take advantage of \eqref{eq14}, we need two more ingredients: (i)  a process that dominates $S^*_1$, $S^*_2$, and $S^*_3$ and that is independent of $\{\Qu_l\}_{l\in\{1,2,3\}}$ and (ii) the independence of $\Qu_1,\Qu_2$ and $\Qu_3$.

 For $\bq{x}\in\Qs_1$, $\rho \in(0,1)$, and integers $i\leq j$, alter the drift of the walk ${x}[i,j]$ by defining
\begin{equation}\label{rw7}
	{x}^\rho[i,j]={x}[i,j]-\rho(j-i+1).
\end{equation} 
For $n\in\N$ and $i\in\Z$, we use the altered queuing map $D_{i,0}$ from \eqref{eq42}. Then we have the bound
\begin{equation}\label{eq20}
	\begin{aligned}
	\sup_{l\in[i,m]}S_2(l)&=\sup_{l\in[i,m]}\Big[D(\bq{x}_3,\bq{x}_4)[i,l]-D(\bq{x}_2,R(\bq{x}_3,\bq{x}_4))[i,l]\Big]\leq \sup_{l\in[i,m]}\Big[{x}_4[i,l]-D_{i,0}(\bq{x}_2,\bq{x}_3))[i,l]\Big]\\
	&\stackrel{\eqref{bq3b}}{\leq}  2\sum_{k=2}^4\sup_{l\in [i,m]}|{x}_{k}^{\rho}[i,l]|.
	\end{aligned}
\end{equation}
The first inequality used part  \ref{bq0} of Lemma \ref{lem:q} and \eqref{bq1}--\eqref{bq4}.  For $S_1$ we have the bound
\begin{equation}\label{eq21}
	\begin{aligned}
	\sup_{l\in[i,m]}S_1(l)&=\sup_{l\in[i,m]}\Big[D(\bq{x}_2,\bq{x}_3,\bq{x}_4)[i,l]-D\big(\bq{x}_1,R(\bq{x}_2,J_3),J_2\big)[i,l]\Big]\\
	&\leq \sup_{l\in[i,m]}\Big[{x}_4[i,l]-D_{i,0}\big(\bq{x}_1,\bq{x}_2,D_{i,0}(\bq{x}_2,\bq{x}_3)\big)[i,l]\Big]\leq  6\sum_{k=1}^4\sup_{l\in [i,m]}|{x}_{k}^{\rho}[i,l]|,
	\end{aligned}
\end{equation}
where we used (via Lemma~\ref{lem:q})
\be\label{t100} \begin{aligned}
D\big(\bq{x}_1,R(\bq{x}_2,J_3),J_2\big)[i,m]&= D\big(\bq{x}_1,R(\bq{x}_2,J_3),R\big(D(\bq{x}_2,J_3),I_3\big)\big)[i,m]\\
&\geq D_{i,0}\big(\bq{x}_1,\bq{x}_2,D(\bq{x}_2,J_3)\big)[i,m]\geq D_{i,0}\big(\bq{x}_1,\bq{x}_2,D_{i,0}(\bq{x}_2,\bq{x}_3)\big)[i,m].
\end{aligned}\ee
Combining the bounds for the walks gives, for $m\in\{1,2,3\}$, 
\begin{equation}\label{eq31}
	S^*_m(j)\leq S^*(j):=6\sum_{k=1}^4\max_{l\in [i,j]}|{x}_{k}^{\rho}[i,l]|,  \qquad j\ge i . 
\end{equation}
The process  $S^*$ is a function of the inputs after time $i-1$ and hence independent of $\{\Qu_l\}_{l\in\{1,2,3\}}$.  \eqref{eq14} implies 
\begin{equation}\label{eq27}
\Pp\big(\TwoRare^{[i,j]}(\bar{\lambda})\big)\leq \Pp\big(\ind_{\Qu_1<S^*(j)}+\ind_{\Qu_2<S^*(j)}+\ind_{\Qu_3<S^*(j)}\geq 2\big). 
\end{equation}

We turn to verify  the independence of $\{\Qu_1,\Qu_2, \Qu_3\}$. From \eqref{eq11} and \eqref{Q3}
\begin{equation*}
	\{I_3(l)\}_{l<i}, \,\,\, 	\{J_3(l)\}_{l<i}, \,\,\,\bq{x}_1 \,\,\text{ and } \,\,\,\bq{x}_2 \,\,\, \text{ are jointly independent of $\Qu_3$.}
\end{equation*}
From \eqref{Q2}, the pair  $(\{I_2(l)\}_{l<i},\{J_2(l)\}_{l<i})$ is a function of $\{I_3(l)\}_{l<i}$, $\{J_3(l)\}_{l<i}$,  and  $\bq{x}_2$. This implies that $\Qu_3$ is independent of
\begin{equation}\label{eq22}
	\big(\bq{x}_1,\bq{x}_2,\{I_2(l)\}_{l<i},\{J_2(l)\}_{l<i},\{I_3(l)\}_{l<i},\{J_3(l)\}_{l<i}\big).
\end{equation}
As $(\Qu_2,\Qu_1)$ is a function of \eqref{eq22}, we conclude that $\Qu_3$ is independent of $(\Qu_2,\Qu_1)$. We are left to show that $\Qu_1$ is independent of $\Qu_2$. First note that the map
\begin{equation}\label{eq26}
	(\bq{x}_2,\bq{x}_3,\bq{x}_4) \mapsto \big(R(\bq{x}_2,J_3),D(\bq{x}_2,J_3),I_3\big)
\end{equation}
is obtained by applying the pair map $(R,D)$
twice, first to $(\bq{x}_3,\bq{x}_4)$ to obtain $(\bq{x}_2,J_3,I_3)$ and then to  $(\bq{x}_2,J_3)$. By \eqref{disId}    each application of $(R,D)$ leaves  three components of the output  vector independent. In particular,
\begin{equation}\label{eq24}
	R(\bq{x}_2,J_3) \,\,\,\text{ is independent of  } \,\,\,\big(D(\bq{x}_2,J_3),I_3\big).
\end{equation}
From \eqref{Q2} and \eqref{eq11}, 
\begin{equation}\label{eq23}
\text{the pair } \ 	(\{J_2(l)\}_{l<i},\{I_2(l)\}_{l<i})\,\,\, \text{ is independent of } \,\,\,\Qu_2.
\end{equation}  
Combining \eqref{eq24} and \eqref{eq23}, 
\begin{equation}\label{eq25}
	\big(\bq{x}_1,R(\bq{x}_2,J_3),\{J_2(l)\}_{l<i}\,,\{I_2(l)\}_{l<i}\big)\,\,\,\text{ is independent of  } \,\,\,\Qu_2.
\end{equation}
As  a function of the collection of random variables in parentheses above, 
$\Qu_1$ is independent of $\Qu_2$.    The independence of $\{\Qu_1,\Qu_2, \Qu_3\}$ has been proved.

By the Burke property (Lemma \ref{lm:Burke}),   $\Qu_i\sim \text{Geom}\Big(\frac{\lambda_{i+1}}{(1-\sum_{l=1}^{i}\lambda_l)(\sum_{l=1}^{i+1}\lambda_l)}\Big)$. As in the statement of the lemma we are proving, let $\lambda^*=\max\{\lambda_2,\lambda_3,\lambda_4\}$ and 
\[ \Delta=\frac{\lambda^*}{(1-\sum_{l=1}^{3}\lambda_l)\lambda_1} \; \ge \;   \frac{\lambda_{i+1}}{(1-\sum_{l=1}^{i}\lambda_l)(\sum_{l=1}^{i+1}\lambda_l)} \quad\text{ for } i\in\{1,2,3\}. 
\]
Let $\hat{\Qu}_1,\hat{\Qu}_2,\hat{\Qu}_3$ be i.i.d.\  random variables with distribution $\text{\rm Geom}(\Delta)
$. Since   the probability of success increased, $(\Qu_1,\Qu_2,\Qu_3)$ stochastically dominates $(\hat{\Qu}_1,\hat{\Qu}_2,\hat{\Qu}_3)$. From this and  a union bound, 
\begin{align*}
\Pp\big(\ind_{\Qu_1<s}+\ind_{\Qu_2<s}+\ind_{\Qu_3<s}
\geq 2\big) &\leq \Pp\big(\ind_{\hat{\Qu}_1<s}+\ind_{\hat{\Qu}_2<s}+\ind_{\hat{\Qu}_3<s}\geq 2\big)\leq 3\Pp\big(\hat{\Qu}_1<s,\hat{\Qu}_2<s\big)\\
&\le 3\Pp\big(\hat{\Qu}_1<s\big)^2.
\end{align*} 
Substitute the last bound into  \eqref{eq27} to get, for $\MTR\geq2$, 
\begin{equation}\label{eq10}
\begin{aligned}
&\Pp\big(\TwoRare_j^i(\bar{\lambda})\big)\leq 3\sum_{s=1}^\infty\Pp\big(\hat{\Qu}_1<s\big)^2\,\Pp\big(S^*(j)=s\big)\\
&=3\sum_{s=0}^\infty\sum_{l=1}^{\MTR}\Pp\big(\hat{\Qu}_1<s\MTR+l\big)^2\,\Pp\big(S^*(j)=s\MTR+l\big)\leq 3\sum_{s=0}^\infty\Pp\big(\hat{\Qu}_1<s\MTR+\MTR\big)^2\sum_{l=1}^{\MTR}\Pp\big(S^*(j)=s\MTR+l\big)\\
&\leq  3\sum_{s=1}^\infty\big(1-(1-\Delta )^{s\MTR}\big)^2\,\Pp\big(S^*(j)> (s-1)\MTR\big).
\end{aligned}
\end{equation}
It remains to control the $S^*$ tail  probability above. 
By the definition  \eqref{eq31} of $S^*$, for $\tau>0$, 
\begin{equation}\label{eq29}
	\Pp\big(S^*(j)>\tau\big)=\Pp\Big(6\sum_{k=1}^4\max_{l\in [i,j]}|{x}^\rho_k[i,l]|>\tau\Big)\leq \sum_{k=1}^4\Pp\big(\max_{l\in [i,j]}|{x}^\rho_k[i,l]|>\tau/24\big).
\end{equation}
Each probability in the last sum above is bounded as follows. Let $t\ge0$.  
\be\label{eq29.5}\begin{aligned}
    &\Pp\Big\{\max_{l\in [i,j]}|{x}^\rho_k[i,l]|>t+\Big|\sum_{m=1}^k\lambda_m-\rho \Big|(j-i+1)\Big\}\\
    &\leq \Pp\Big\{\max_{l\in [i,j]}\big|{x}^\rho_k[i,l]-(l-i+1)\big(\sum_{m=1}^k\lambda_m-\rho\big) \big|>t\Big\}\\
    &=\Pp\Big\{\max_{l\in [i,j]}\big|{x}_k[i,l]-(l-i+1)\sum_{m=1}^k\lambda_m \big|>t\Big\}  
    \le 3\,\max_{l\in[i,j]}\Pp\Big\{\big|{x}_k[i,l]-(l-i+1)\sum_{m=1}^k\lambda_m \big|>t/3\Big\}\\
    &\leq 6\,\max_{l\in[i,j]} \exp\Big(-\frac{2t^2}{9(l-i+1)}\Big)
    = 6 \exp\Big(-\frac{2t^2}{9(j-i+1)}\Big). 
\end{aligned}\ee
The first inequality above is elementary and the first equality cancels the $\rho$-terms.  
Etemadi's inequality \cite[Theorem 22.5]{bill-95-prob-meas} moves the maximum outside the probability.  The last inequality is 
Hoeffding's \cite[Theorem 2.8]{bouc-lugo-mass-13}. 

Apply the last   bound  to the  $k$-term in the last sum of \eqref{eq29} with 
$t=\tfrac \tau{24}-|\sum_{m=1}^k\lambda_m-\rho \tspb|\cdot(j-i+1)$.  If $t\ge0$ then the bound gives the $k$-term after the first inequality below. If $t<0$ then the bound below is automatically valid because it bounds a  probability with $6e^0$.  
\begin{equation}\label{eq30}
		\begin{aligned}
		&\Pp\big(S^*(j)>\tau\big)\leq 6\sum_{k=1}^4\exp\Bigg(-\frac{2\big(\big[\tfrac \tau{24}-\big|\sum_{m=1}^k\lambda_m-\rho \big|\cdot(j-i+1)\big]_+\big)^2}{9(j-i+1)}\Bigg)\\
		&\leq 24\exp\Bigg(-\frac{2\Big(\big[\tfrac \tau{24}-\big|\sum_{m=1}^4\lambda_m-\rho \big|\cdot(j-i+1)-|\lambda_1-\rho |\cdot(j-i+1)\big]_+\Big)^2}{9(j-i+1)}\Bigg).
		\end{aligned}
\end{equation}
Above we used the inequality  $\big|\sum_{m=1}^k\lambda_m-\rho \big|\leq \big|\sum_{m=1}^4\lambda_m-\rho \big|+|\lambda_1-\rho|$, valid  for all $k\in\{1,2,3,4\}$ because the convex function $|x-\rho|$ achieves its maximum at an endpoint of an interval. Substitute  \eqref{eq30}  into  \eqref{eq10} to obtain the desired estimate \eqref{tr}. The proof of the lemma is complete.  
\end{proof}


We introduce notation for discretizing continuous customer classes. 
Let $\mu_0\geq 1$ and $\en\in\N$. Define 
\be\label{eq37}\begin{aligned}
\cE=\cE(\en)&=[-\mu_0,\mu_0]\cap \{i2^{-\en}:i\in\Z\}\,, \qquad  
\Ind(\en)=2^\en\cE=\{-\fl{2^M\mu_0},\dotsc,\fl{2^M\mu_0}\} \,, \\[3pt]
i_{\text{min}}(\en)&=\min\Ind(\en)=-\fl{2^M\mu_0}
\qquad\text{and}\qquad
i_{\text{max}}(\en)=\max\Ind(\en)=\fl{2^M\mu_0}.
\end{aligned}\ee
The interval $[-\mu_0,\mu_0]$ remains fixed in the calculations while $\en$ varies, but 
the dependence on $\en$ will also be typically suppressed from the notation. 
Note the bound on the size of $\cE$: 
\be\label{cEsize}  |\cE| =i_{max}-i_{min}+1 \le 2^{M+1}\mu_0+1. \ee
For $|\cE|$ different customer classes  define the  vector
$\bar{\lambda}^{M,N}=(\lambda_1^{M,N},\dotsc,\lambda_{|\cE|}^{M,N})$ of Bernoulli densities that are small perturbations of density $\tfrac{1+v}2$: 
\begin{equation}\label{eq35}\begin{aligned}
&\lambda_1^{M,N}=\tfrac{1+v}2+ (1-v^2) i_{\text{min}}2^{-M}N^{-1/2}\\[3pt] 
& \lambda_i^{M,N}=(1-v^2)2^{-M} N^{-1/2} \qquad \text{ for } \  i\in\{2,3,\dotsc,|\cE|\}.
\end{aligned}\end{equation}
The densities are centered around $\tfrac{1+v}2=\Pp(U_j\le v)$, corresponding to the centering of the speed process around $v$. 
 Let $\bq{v}^{M,N}\in\Qs_{|\cE|} $ have the invariant  multiclass distribution $\mu^{\bar\lambda^{M,N}}$,  as defined in  \eqref{mu}. 
In particular, for $m\in\lzb1,|\cE|\,\rzb$, customers of classes $\lzb1,m\rzb$ have density 
$\sum_{i=1}^m \lambda_i^{M,N}=\tfrac{1+v}2+ (1-v^2)(i_{\text{min}}+m-1)2^{-M}N^{-1/2}$. 
 Let  $x_0\geq 1$ and  $x_0^N=\tfrac{2x_0}{1-v^2}N$. 
 
 For $l \in\lzb1,|\cE|-3\rzb$, define the event  that among   the  three consecutive classes $\{l+1,l+2,l+3\}$,   at least two appear  in $\bq{v}^{M,N}$ in the time interval $[-x_0^N,x_0^N]$:  
	\begin{align*}
	\mathcal{A}^{M,N}_{l}:=
	\Big\{\ind\{\Cls^{[-x_0^N,x_0^N]}_{l}(\bq{v}^{M,N})<&\Cls^{[-x_0^N,x_0^N]}_{l+1}(\bq{v}^{M,N})\}+\ind\{\Cls^{[-x_0^N,x_0^N]}_{l+1}(\bq{v}^{M,N})<\Cls^{[-x_0^N,x_0^N]}_{l+2}(\bq{v}^{M,N})\}\\
	&\qquad\qquad+\ind\{\Cls^{[-x_0^N,x_0^N]}_{l+2}(\bq{v}^{M,N})<\Cls^{[-x_0^N,x_0^N]}_{l+3}(\bq{v}^{M,N})\}\geq2\Big\}.
	\end{align*}
In our development  class 1 is not rare and hence it is omitted from the options above.  
Let 
	\begin{equation}\label{eq32}
		\mathcal{A}^{M,N}=\bigcup_{l=1}^{|\cE|-3}\mathcal{A}^{M,N}_{l}
	\end{equation}
 be   the event that among some set of   three consecutive classes in $\lzb2,|\cE|\tspb\rzb$,  at least two appear  in the time interval $[-x_0^N,x_0^N]$. 
 

\begin{lemma}\label{lem:mjub}
	For  $\mu_0,x_0\geq 1$ there exists a constant $C=C(v,\mu_0,x_0)$ such that, whenever  $N>(8\mu_0)^2$ $\vee$ $2^{-2M+8}\bigl(\tfrac{1+v}{1-v}\bigr)^2$,
 we have the bound  \tsp   $\Pp\big(\mathcal{A}^{M,N}\big)\leq  C(v,\mu_0,x_0)2^{-M}.$
\end{lemma}
\begin{proof}
	
	We apply the estimate from Lemma \ref{lem:TR} to each event in the union \eqref{eq32}.  Let $i'\in \lzb1,|\cE|-3\rzb$ and utilize   the map $\Phi$ from \eqref{Phi} to relabel the classes $\{i', i'+1,i'+2,i'+3\}$ as $\{1,2,3,4\}$:  the sequence  $\bq{v}_4=\Phi[\bq{v}^{M,N},(i',i'+1,i'+2,i'+3)]$ obeys the parameter vector $\bar{\rho}^{M,N}=(\rho^{M,N}_1,\rho^{M,N}_2,\rho^{M,N}_3,\rho^{M,N}_4)$ with  coordinates 
\begin{equation*}
	\begin{aligned}
	\rho_1^{M,N}=\tfrac{1+v}2+ (1-v^2)(i_{\text{min}}+i'-1)2^{-M}N^{-1/2} \quad\text{and}\quad 
	\rho_l^{M,N}=(1-v^2)2^{-M}N^{-1/2} \quad \text{ for } \ l\in\{2,3,4\}.
	\end{aligned}
\end{equation*}
By  Lemma \ref{lem:Phi}, this operation preserves the multiclass distribution, and thereby  
\begin{equation}\label{eq33}
	\Pp(\mathcal{A}^{M,N}_{i'})=\Pp\big(\TwoRare^{[-x_0^N, x_0^N]}(\bar{\rho}^{M,N})\big). 
\end{equation}

 Apply  \eqref{tr} with $\MTR=N^{1/2}$, $\lambda_m=\rho^{M,N}_m$, $[i,j]=[-\fl{x_0^N}, \fl{x_0^N}]$,  and $\rho=\frac{1+v}2$.  Our hypothesis gives $\mu_0 N^{-1/2}<1/4$ which we use repeatedly.   In this setting, the success probability $\Delta$ satisfies 
 \begin{align*}
\Delta=\frac{(1-v^2)2^{-M}N^{-1/2}}{(1-\sum_{l=1}^{3}\rho^{M,N}_l)\rho^{M,N}_1}   
\le \frac{(1-v^2)2^{-M}N^{-1/2}}{(\frac12-\frac{v}2 -(1-v^2) \mu_0N^{-1/2})^2} < 
 2^{-M+4}\tspb \tfrac{1+v}{1-v} \tspb N^{-1/2}
 =2^{-M+4}\tspb b(v) \tspb N^{-1/2}
 \end{align*}
 where we abbreviated $b(v)=\tfrac{1+v}{1-v}$. 
Combine this with the inequality $(1+\frac{x}{n})^n\geq(1+x)$ for $|x|\leq n$. Then for $N\ge2^{-3M+12}b(v)^3$, 
\begin{align*}
\big(1-(1-\Delta)^{sN^{1/2}}\big)^2\le \big(1-(1-2^{-M+4}b(v)N^{-1/2} )^{sN^{1/2}}\big)^2   \le 
2^{-2M+8}b(v)^2 s^2. 
\end{align*}
Note that 
\[ 
	 \Big|\sum_{m=1}^4 \rho^{M,N}_m-\tfrac{1+v}2\Big| \vee \big|\rho^{M,N}_1-\tfrac{1+v}2\big|<   \mu_0 N^{-1/2}.
\] 
With these auxiliary bounds, \eqref{tr} yields this estimate:  
\begin{equation*}\begin{aligned} 
	&\Pp\big(\TwoRare^{[-x_0^N, x_0^N]}(\bar{\rho}^{M,N})\big)\\
	&\qquad 
	\leq 72\sum_{s=1}^\infty
	\big(1-(1-\Delta)^{sN^{1/2}}\big)^2
	\exp\Bigg\{-\frac{2\Big(\Big[\tfrac {(s-1)N^{1/2}}{24} -2\mu_0N^{-1/2}(2x_0^N+1)\Big]_+\Big)^2}{9(2x_0^N+1)}\Bigg\} \\
	&\qquad 
	\leq 72\cdot2^{-2M+8}b(v)^2\sum_{s=1}^\infty s^2  \exp\Bigg\{-\frac{2\Big(\big[\tfrac {s-1}{24} -2\mu_0(4x_0(1-v^2)^{-1}+1)\big]_+\Big)^2}{9(4x_0(1-v^2)^{-1}+1)}\Bigg\}\leq C'(v, \mu_0,x_0)2^{-2M}.
\end{aligned} \end{equation*}
 

A union bound in \eqref{eq32}, bound \eqref{cEsize} on $|\cE|$,  equality \eqref{eq33}, and the bound in the last display yield 
\begin{equation*}
	\Pp\big(\mathcal{A}^{M,N}\big)
	\leq (|\cE|-3)  C'(v,\mu_0,x_0)2^{-2M}
	\leq   C''(v,\mu_0,x_0)2^{-M}.
\end{equation*}
The proof is complete.
\end{proof}


Recall $\wt H^N$ from \eqref{Href}, and the restriction $\wt H^{N,x_0}_\mu=\wt H^N_\mu\vert_{[-x_0, x_0]}$ 
for $x_0\geq 1$.
$\mu\mapsto \wt H^{N,x_0}_\mu$ is a function taking values in $C[-x_0,x_0]$. We say $\mu\in\R$ is a \textit{jump point} of the function $\wt H^{N,x_0}$ if $\wt H^{N,x_0}_{\mu-}\neq \wt H^{N,x_0}_\mu$, i.e.\ there exists $x\in[-x_0,x_0]$ such that 
\begin{equation*}
	\lim_{h		\searrow \,0}\wt H^{N,x_0}_{\mu-h}(x)\neq \wt H^{N,x_0}_\mu(x).
\end{equation*} 
For $\mu_0\geq 1$,  define the random variable that registers the distance between the closest pair of distinct jump points of the process $\wt H^{N,x_0}$ as 
\begin{equation*}
	\ClsJmp_N^{x_0,\mu_0}:=\inf\{|\mu_1-\mu_2|:\text{$\mu_1,\mu_2$ are distinct jump points of $\mu\mapsto \wt H^{N,x_0}_\mu$ in $(-\mu_0,\mu_0)$}\}. 
\end{equation*}
Set $\ClsJmp_N^{x_0,\mu_0}=\infty$ if there is at most one jump in $(-\mu_0,\mu_0)$. 
\begin{lemma}\label{lm:cljp}
	For all  $\mu_0,x_0\geq 1$ and $\delta>0$  there exists a constant $C=C(v,\mu_0,x_0)$ such that, whenever $N> (8\mu_0)^2\vee 2^{8}\delta^2\bigl(\tfrac{1+v}{1-v}\bigr)^2$, 
 we have $\Pp\big(\ClsJmp_N^{x_0,\mu_0}\leq \delta\big)\leq C\delta$. 
\end{lemma}
\begin{proof}
We deduce the bound from the decomposition
\be\label{CJ700} 
\Pp\big(\ClsJmp_N^{x_0,\mu_0}\leq \delta\big) \le   \sum_{M=M_0}^\infty \Pp\big(2^{-M}<\ClsJmp_N^{x_0,\mu_0}\leq 2^{-M+1}\big)
\ee
where $M_0$ satisfies $2^{-M_0}<\delta\le 2^{-M_0+1}$.  
	 On the event $2^{-M}<\ClsJmp_N^{x_0,\mu_0}\leq 2^{-M+1}$,    there exists $i_e\in \Ind$ (recall \eqref{eq37})  such that two jump points $\mu_1,\mu_2$ satisfy one of these two cases:
		\begin{align*}
			\text{\rm Case 1:} \quad &\mu_1\in \big[2^{-M}i_e,2^{-M}(i_e+1)\big)\text{\,\,and\,\,\,}\mu_2\in \big[2^{-M}(i_e+1),2^{-M}(i_e+2)\big),  \\[3pt]  
			\text{\rm Case 2:} \quad &\mu_1\in \big[2^{-M}i_e,2^{-M}(i_e+1)\big)\text{\,\,and\,\,\,}\mu_2\in \big[2^{-M}(i_e+2),2^{-M}(i_e+3)\big).
	\end{align*}

Define the queuing configuration  $\bq{w}^{M,N}\in \Qs_{|\cE|}$ by 
	\begin{equation*}
		\begin{cases}
			w^{M,N}(j)\leq k & \text{ if  \,\,$U_j\leq v+(1-v)^2(i_{min}+k-1)2^{-M}N^{-1/2}$ and $  k\in\{1,\dotsc,|\cE|\}$}\\
			w^{M,N}(j)=\infty & \text{ if  \,\,$U_j>  v+(1-v)^2i_{max}\,2^{-M}N^{-1/2}$ }. 
		\end{cases}
	\end{equation*}
	That is, $\bq{w}^{M,N}$   assigns a customer of class  $k\in\{2,\dotsc,|\cE|\}$ to position  $j$ if and only if \[ v+(1-v)^2(i_{min}+k-2)2^{-M}N^{-1/2}<U_j\leq v+(1-v)^2(i_{min}+k-1)2^{-M}N^{-1/2}. \]
	From Lemma \ref{lem:FU},   $\bq{w}^{M,N}\sim \mu^{\bar{\lambda}^{M+1,N}}$ for $\bar{\lambda}^{M+1,N}$ in \eqref{eq35}. The superscript is $M+1$ instead of $M$ because for $U_j$ probability = interval length$/2$. The two cases above imply two cases for the number of different classes in $\bq{w}^{M,N}$. There exists $j_e\in \{1,\dotsc,|\cE|-3\}$ such that one of the two cases below holds:  
		\begin{align}\label{eq38}
			\Cls^{[-x_0^N,x_0^N]}_{j_e}(\bq{w}^{M,N}) &<\Cls^{[-x_0^N,x_0^N]}_{j_e+1}(\bq{w}^{M,N})<\Cls^{[-x_0^N,x_0^N]}_{j_e+2}(\bq{w}^{M,N}),  \\[3pt]  
		 \label{eq39}
			\Cls^{[-x_0^N,x_0^N]}_{j_e}(\bq{w}^{M,N}) &<\Cls^{[-x_0^N,x_0^N]}_{j_e+1}(\bq{w}^{M,N})<\Cls^{[-x_0^N,x_0^N]}_{j_e+3}(\bq{w}^{M,N})
		\end{align}
	Recall \eqref{eq32} and note that 
	\begin{equation*}
		\{\text{there exists $j_e\in \{1,\dotsc,|\cE|-3\}$ such that  either \eqref{eq38} or \eqref{eq39} holds}\}\subseteq \mathcal{A}^{M+1,N}.
	\end{equation*} 
With $\delta>2^{-M}$, our hypothesis on $N$ satisfies the assumption of  Lemma \ref{lem:mjub}. We get 
	\begin{equation*}
		\Pp\big(2^{-M}<\ClsJmp_N^{x_0,\mu_0}\leq 2^{-M+1}\big)
		\le \Pp(\mathcal{A}^{M+1,N}) \leq  C'(v,\mu_0,x_0)2^{-M}.
	\end{equation*}
Substitute this back into \eqref{CJ700} to complete the proof of the lemma.  
\end{proof}
 
We verify the first piece of process-level weak convergence.  Recall the modulus  $\omega$ in~\eqref{moc}.

\begin{proposition}\label{prop:MC}
	For every $\epsilon>0$ and $\mu_0>0$, \;
 $\ddd\lim_{\delta \rightarrow 0} \limsup_{N\rightarrow \infty} \Pp\big\{\omega(\wt H^N,\mu_0,\delta)>\epsilon\big\}=0$.
\end{proposition}
\begin{proof}

Pick $m\ge\log(2\epsilon^{-1})$ and recall the metric  $d_m$ in \eqref{dn} for restrictions to $[-m,m]$. Define the restricted modulus 
\be\label{wm} \begin{aligned}
		\omega^m(\wt H^N,\mu_0,\delta)=\inf\big\{\max_{1\leq i \leq n} \theta^m_{\wt H^N}[t_{i-1},t_i):\exists n\geq 1, \; &-\mu_0=t_0<t_1< \dotsc<t_n=\mu_0 \\
		&\text{such that $t_i-t_{i-1}>\delta$ for all $i\leq n$}\big\} 
\end{aligned}\ee
where 
\begin{align*}
    \theta^m_{\wt H^N}[a,b)&=\sup_{\mu,\nu\in[a,b)}d_m\big(\wt H^N_\mu, \wt H^N_\nu\big)
= \sup_{\mu,\nu\in[a,b)} \sup_{|x|\le m} \bigl\lvert\wt H^N_\mu(x)- \wt H^N_\nu(x)\bigr\rvert  \\
&\le \bigl(\wt H^N_{b-}(m)- \wt H^N_a(m) \bigr) 
- \bigl( \wt H^N_{b-}(-m)- \wt H^N_a(-m)\bigr).
\end{align*}
The last inequality used monotonicity of $\wt H^N_\mu(x)- \wt H^N_\nu(x)$ in $\mu$, $\nu$ and $x$.  Since $\mu\mapsto\wt H^{N,m}_\mu$ is a jump function, $\theta^m_{\wt H^N}[a,b)$ vanishes precisely when  there is no jump in the open interval $(a,b)$. 

Starting with \eqref{maprox} write 
\be\label{H1009} \begin{aligned}
    	&\Pp\big(\omega(\wt H^N,\mu_0,\delta)>\epsilon\big)
 \leq \Pp\big(\omega^m(\wt H^N,\mu_0,\delta)>{\epsilon}/2\big) \le  
 \Pp\bigl\{\ClsJmp_N^{m,\mu_0}< 2\delta\bigr\} \\[4pt]
 &\quad   + 
     \Pp\bigl\{\sup_{\mu_1,\mu_2\in[\mu_0,\mu_0+2\delta]}d(\wt H^{N,m}_{\mu_1},\wt H^{N,m}_{\mu_2})>\epsilon/2\bigr\}
 + 
     \Pp\bigl\{\sup_{\mu_1,\mu_2\in[\mu_0-2\delta,\mu_0]}d(\wt H^{N,m}_{\mu_1},\wt H^{N,m}_{\mu_2})>\epsilon/2\bigr\}. 
\end{aligned}\ee
The second inequality is justified by the following  observations.  Suppose that 
\[  \sup_{\mu_1,\mu_2\in[\mu_0,\mu_0+2\delta]}d(\wt H^{N,m}_{\mu_1},\wt H^{N,m}_{\mu_2})\le\epsilon/2 
\quad\text{ and }\quad 
\sup_{\mu_1,\mu_2\in[\mu_0-2\delta,\mu_0]}d(\wt H^{N,m}_{\mu_1},\wt H^{N,m}_{\mu_2})\le\epsilon/2, 
\]
and if there is more than one jump in  $(-\mu_0, \mu_0)$, the jumps are separated by at least $2\delta$ from each other. 
Then let the interior partition points $t_1,\dotsc, t_{n-1}$ in \eqref{wm} be exactly the jump locations in  $(-\mu_0, \mu_0)$. (The event $\omega^m(\wt H^N,\mu_0,\delta)>0$ forces at least one jump to occur so $n\ge2$.)  This is an acceptable partition if $t_1>-\mu_0+\delta$ and  $t_{n-1}<\mu_0-\delta$. If the latter condition fails, redefine $t_{n-1}=\frac12(t_{n-2}+\mu_0)\wedge(\mu_0-\frac32\delta)$ to have $\mu_0-t_{n-1}\in(\delta, 2\delta)$ and $t_{n-1}-t_{n-2}>\delta$. Redefine $t_1$ analogously if it is too close to $-\mu_0$. This is all feasible if $\delta$ is small enough relative to $\mu_0$. 

Now $\theta^m_{\wt H^N}[t_{i-1},t_i)=0$ for $1<i<n$, $\theta^m_{\wt H^N}[t_0,t_1)\le \epsilon/2$, and $\theta^m_{\wt H^N}[t_{n-1},t_n)\le \epsilon/2$. Together these imply 
$\omega^m(\wt H^N,\mu_0,\delta)\le{\epsilon}/2$.

 The claim of the proposition 
 follows from \eqref{H1009}  because  by the jump estimate in Lemma \ref{lm:cljp} and the stochastic continuity in Proposition \ref{prop:SC},  $\Pp(\omega(\wt H^N,\mu_0,\delta)>\epsilon)\leq C(v,\mu_0,\epsilon)\delta$ for large enough $N$. 
\end{proof}

\subsection{Stochastic continuity}
Fix $\bar{\lambda}=(\lambda_1,\lambda_2)$, and let  $\bq{v}_2\sim\mu^{\bar{\lambda}}$. For integers $i\leq j$ we define the event
\begin{equation*}
	\OneRare^{[i,j]}(\bar{\lambda})=\big\{\Cls^{[i,j]}_1(\bq{v}_2)<\Cls^{[i,j]}_2(\bq{v}_2)\big\}.
\end{equation*}
$\OneRare^{[i,j]}(\bar{\lambda})$ is the event where $\bq{v}_2$ has 
 at least one second class customer in the time interval $[i,j]$. The following bound is analogous to Lemma \ref{lem:TR}.
\begin{lemma}\label{lem:ORb}
	Fix $\bar{\lambda}=(\lambda_1,\lambda_2)$ and integers $i\leq j$,  and set $\Delta=\frac{\lambda_2}{(1-\lambda_1)(\lambda_1+\lambda_2)}$. For any $R\geq 2$ and $\rho\in(0,1)$.
	\begin{equation}\label{tr2}\begin{aligned} 
		&\Pp\big(\OneRare^{[i,j]}(\bar{\lambda})\big)\leq 2\sum_{s=1}^\infty \big(1-(1-\Delta )^{sR}\big)\\
		&\qquad\times 
		\exp\Bigg\{-\frac{\Big(\Big[\tfrac {(s-1)R}{2}-|\lambda_1+\lambda_2-\rho|\cdot (j-i+1)-|\lambda_1-\rho|\cdot(j-i+1)\Big]_+\Big)^2}{9(j-i+1)}\Bigg\}.
\end{aligned} 	\end{equation}
\end{lemma}
\begin{proof}
	The proof is very similar to the one of Lemma \ref{lem:TR}, so we only point out how to adapt it here. Let $\bar{\bq{x}}=(\bq{x}_1,\bq{x}_2) \sim \nu^{\bar{\lambda}}$ and  $\bq{v}_2\sim \Vmap_2(\bar{\bq{x}})$ so that $\big(\Cls^{[i,j]}_1(\bq{v}_2),\Cls^{[i,j]}_2(\bq{v}_2)\big)\sim \big( D(\bq{x}_1,\bq{x}_2)[i,j]
	,{x}_2[i,j]\big)$. Similar to \eqref{eq13}, we have
	\begin{equation*}
		\OneRare^{[i,j]}(\bar{\lambda})\subseteq  \{\Qu<S^*(j)\}
	\end{equation*}
	where $S^*$ is the running maximum of the random walk
	$ 
		S(m):={x}_2[i,m]-{x}_1[i,m ]
	$
	and $\Qu\sim \text{Geom}\big(\frac{\lambda_2}{(1-\lambda_1)(\lambda_1+\lambda_2)}\big)$ is   independent of $S^*$. Similar to \eqref{eq10}, for $R\geq 2$ we have
	\begin{equation}\label{eq41}
		\Pp\big(\OneRare^{[i,j]}(\bar{\lambda})\big)\leq  \sum_{s=1}^\infty\Pp\big(\Qu\leq sR\big)\Pp\big(S^*(j)> (s-1)R\big).
	\end{equation}
	Using the bound
	\begin{equation*}
		\Pp\big(S^*(j)>t\big)\leq \sum_{k=1}^2\Pp\big(\sup_{l\in[i,j]}|{x}^\rho_k[i,l]|>t/2\big),
	\end{equation*}
	and bounds similar to \eqref{eq29.5}--\eqref{eq30} in \eqref{eq41} we obtain \eqref{tr2}.
\end{proof}
\begin{proposition}\label{prop:SC}
	For every $\mu\in\R$ and $0<\epsilon<1$, there exists  constants $C(v,\mu, \epsilon)$ and $N_0(v,\mu, \epsilon)$ such that for any $0<\delta<1$ and  $N\ge N_0$,   
	\begin{equation*}
		\Pp\Big(\sup_{\mu_1,\mu_2\in(\mu-\delta,\mu+\delta]}d(\wt H^N_{\mu_1},\wt H^N_{\mu_2})>\epsilon\Big)\le C\delta.
	\end{equation*}
\end{proposition}
\begin{proof}
	Let $n=\lceil\log(\tfrac{\epsilon}{2})\rceil$. Then
	\begin{equation}\label{T439} 
		\begin{aligned}
			&\Pp\Big(\sup_{\mu_1,\mu_2\in(\mu-\delta,\mu+\delta]}d\big(\wt H^N_{\mu_1},\wt H^N_{\mu_2}\big)>\epsilon\Big)\leq \Pp\Big(\sup_{\mu_1,\mu_2\in(\mu-\delta,\mu+\delta]}d_{n-1}\big(\wt H^N_{\mu_1},\wt H^N_{\mu_2}\big)>\tfrac{\epsilon}{2}\Big)\\
		& \leq \Pp\Big(\big[\wt H^N_{\mu+\delta}(n-1)-\wt H^N_{\mu+\delta}(-n+1)\big]-\big[\wt H^N_{\mu-\delta}(n-1)-\wt H^N_{\mu-\delta}(-n+1)\big]>0\Big)\\
		&\le \Pp\Big\{ \,\exists i\in\bigl[-\tfrac{2n}{1-v^2}N,\tfrac{2n}{1-v^2}N\bigr]:  U_i\in\bigl(v+\tfrac{1-v^2}{N^{1/2}}(\mu-\delta), v+\tfrac{1-v^2}{N^{1/2}}(\mu+\delta)\bigr] \, \Big\}. 
		\end{aligned}
	\end{equation}
The second inequality used  monotonicity.  To turn this into a probability of a two-class queuing configuration,  discretize  the classes as follows: 
\begin{align*}
U_i\in \bigl(-1\,, v+(1-v^2)(\mu-\delta)N^{-1/2}\bigr] &\ \longrightarrow\  \text{ class }  1\\
    U_i\in \bigl(v+(1-v^2)(\mu-\delta)N^{-1/2}\,, v+(1-v^2)(\mu+\delta)N^{-1/2}\bigr] &\ \longrightarrow \ \text{ class }  2\\
U_i\in \bigl(  v+(1-v^2)(\mu+\delta)N^{-1/2}\,, 1\bigr) &\ \longrightarrow\  \text{ class }   \infty.    
\end{align*}
The probabilities of the classes are recorded in the parameter vector  \[ \bar{\lambda}=(\lambda_1, \lambda_2)=\big(\tfrac{1+v}2+\tfrac12(1-v^2)(\mu-\delta)N^{-1/2}\,,(1-v^2)\delta N^{-1/2}\big).\]    
The last event in \eqref{T439} is the existence of a second class customer in time interval $\bigl[-\tfrac{2n}{1-v^2}N,\tfrac{2n}{1-v^2}N\bigr]$. Thus we have 
\begin{equation*}
	\begin{aligned}
		&\Pp\Big(\sup_{\mu_1,\mu_2\in(\mu-\delta,\mu+\delta]}d\big(\wt H^N_{\mu_1},\wt H^N_{\mu_2}\big)>\epsilon\Big)\leq \Pp\Big\{\OneRare^{\bigl[-\tfrac{2n}{1-v^2}N,\tfrac{2n}{1-v^2}N\bigr]}(\bar{\lambda})\Big\}\\
		&\qquad \qquad 
		\leq C\delta \sum_{s=1}^\infty s \exp\Bigg(-\frac{\big[\tfrac {(s-1)}{2}-C(4n(1-v^2)^{-1}+1) \big]_+^2}{9(2n(1-v^2)^{-1}+1)}\Bigg)\leq C(v,\mu,\epsilon) \delta.
	\end{aligned}
\end{equation*}
The  penultimate inequality applied \eqref{tr2} with 
$[i,j]=\bigl[-\tfrac{2n}{1-v^2}N,\tfrac{2n}{1-v^2}N\bigr]$, 
$\rho=\tfrac{1+v}2$, 
$\Delta=\frac{\lambda_2}{(1-\lambda_1)(\lambda_1+\lambda_2)}$ $\leq$ $C(v,\mu)\delta N^{-1/2}$,  and $R=N^{1/2}$.
\end{proof}

\begin{proof}[Proof of Theorem \ref{thm:conv}]
	To show that $H^N$ converges to some element $H\in D(\R,C(\R))$, it is enough to show that the three items of Lemma \ref{lem:conv} hold for $H^N$. Item \eqref{ss} follows from Proposition \ref{prop:SC}. Item \eqref{cfd} follows from Proposition \ref{prop:fdd}. Item \eqref{fdc} follows from Proposition \ref{prop:MC}. From Proposition~\ref{prop:fdd}, the limiting object $H$ has the same finite-dimensional  distributions as the SH, which implies that $H=G$. 
\end{proof}


\section{Coupled multiclass measures  for general exclusion processes}

\label{sec:Asp}
Presently a speed process has been associated to three particle systems:   TASEP  \cite{Amir_Angel_Valko11},    ASEP  \cite{aggarwal2022asep}, and for the totally asymmetric zero range process (TAZRP) a result in this spirit was obtained in \cite{amir2021tazrp}.  The speed process   records the asymptotic speeds of  particles of ordered classes and it couples all the translation-invariant multitype stationary measures. 
To set the stage for extensions of our main results beyond nearest-neighbor exclusion processes, in this section we construct a coupling of multiclass  stationary distributions for a  general translation-invariant one-dimensional exclusion process and then prove Theorem \ref{exun}.  At the end of the section we connect this object to a speed process, \textit{assuming} that the latter exists and is stationary. 


Fix a probability kernel $p:\Z\times \Z\rightarrow [0,1]$
that satisfies the assumptions stated above Theorem \ref{exun}, namely, translation invariance  $p(x,y)=p(0,y-x)$ and that 
for each pair   $x,y\in\Z$ there exists $m\in\Z_+$
such that $p^{(m)}(x,y)+p^{(m)}(y,x)>0$. 

The generator $L^{ep}$ of the exclusion process on the particle configuration space $\{0,1\}^\Z$ is 
\begin{equation}\label{L780} 
	L^{ep}f(\eta)=\sum_{x,y\in\Z}p(x,y)\eta(x)(1-\eta(y))[f(\eta^{x,y})-f(\eta)]
\end{equation}
where
\begin{equation}\nonumber
	\eta^{x,y}(z)=
	\begin{cases}
		\eta(z) & z\notin\{x,y\}\\
		\eta(y) & z=x\\
		\eta(x) & z=y.
	\end{cases}
\end{equation}
We construct this process by attaching to each directed edge $(x,y)$ a Poisson clock of rate $p(x,y)$ whose rings trigger jump attempts. A jump from $x$ to $y$  is completed if there is a particle at $x$ and none at $y$. 
This dynamics generalizes naturally to a multiclass version: the contents of sites  $x$ and  $y$  are exchanged  iff the particle at $x$ has a lower label (higher priority)  than the particle at $y$.

We let $L^{cep}$  denote  the generator of the corresponding  continuum exclusion process  with state space $[0,1]^\Z$ and the same kernel $p$:  \begin{equation}\label{L800}
	L^{cep}f(\xi)=\sum_{x,y\in\Z}p(x,y)[f(\wt{\xi}^{\,x,y})-f(\xi)], \qquad \xi\in[0,1]^\Z, 
\end{equation}
where
\begin{equation}
	\wt{\xi}^{\,x,y}(z)=
	\begin{cases}
		\xi(z) & z\notin\{x,y\}\\
		\max\{\xi(x),\xi(y)\} & z=x\\
		\min\{\xi(x),\xi(y)\} & z=y.
	\end{cases}
\end{equation}
This process follows the same Poisson clocks on the directed edges. 
  When the clock of edge $(x,y)$ rings, 
the values $\xi(x)$ and $\xi(y)$ are exchanged if $\xi(x)<\xi(y)$, otherwise kept unchanged.

The existence and  uniqueness properties of translation-invariant  stationary distributions of  multicomponent 
exclusion processes continue to hold under this more general transition kernel $p$, by the same proofs based on  Section VIII.3 of Liggett \cite{liggett1985interacting}.  
Given 
 an increasing $k$-vector $\bar{\rho}=(\rho_1,\ldots,\rho_k)\in [0,1]^k$ of densities, 
 there exists a unique measure $\mu^{\bar{\rho}}$ on $(\{0,1\}^\Z)^k$ with Bernoulli marginals  $\nu^{\rho_1},\ldots,\nu^{\rho_k}$ such that $\mu^{\bar{\rho}}$ is translation-invariant and stationary under the joint evolution of $k$  exclusion processes with generator $L^{ep}$, coupled through common Poisson clocks (basic coupling). Moreover, if $(\eta_1,\ldots,\eta_k)\sim \mu^{\bar{\rho}}$ then 
\begin{equation}\nonumber
	\mu^{\bar{\rho}}\big(\eta_1\leq \eta_2\leq \cdots\leq \eta_k\big)=1.
\end{equation}

Define the vector $\bar{\rho}^N=(\rho^N_1,\dotsc,\rho^N_{2^N})\in[0,1]^{2^N}$ by 
\begin{equation}\nonumber
	\rho^N_i=i2^{-N}  \ \text{ for } \  i\in\{1,\dotsc,2^N\}.
\end{equation}
Let $\bar{\eta}^N\in(\{0,1\}^\Z)^{2^N}$ denote  a $2^N$-component random particle configuration with the $\bar{\rho}^N$-stationary distribution: 
\begin{equation}\label{g2}
	\bar{\eta}^N=(\eta^N_1,\ldots,\eta^N_{2^N})\sim \mu^{\bar{\rho}^N}.
\end{equation}
Note that density one implies that  $\eta^N_{2^N}(x)=1$ $\forall x\in\Z$. 
Map $\bar{\eta}^N$ bijectively into a multitype configuration  $W^N=\{W^N(x)\}_{x\in\Z}$ with values in $\{i2^{-N}: i\in\lzb1,2^N\rzb\}$ by  
\begin{equation}\label{g1}
	W^N(x):=\min\{j2^{-N}:  j\in\lzb1,2^N\rzb, \, {\eta}_j^N(x)=1\}.
\end{equation}
In words, $W^N(x)$ is the smallest density $\rho^N_{j_0}$ in the vector $\bar{\rho}^N$ such that the profile $\eta^N_{j_0}\sim\nu^{\rho^N_{j_0}}$ has a particle at site $x$. 
Marginally $W^N(x)$ is uniform on the set $\{j2^{-N}:  j\in\lzb1,2^N\rzb\}$.  

The next result states that for each $N$, $W^N$ is stationary under the exclusion dynamics of $L^{cep}$.   Denote by $\mathcal{W}^N$ the map $\bar{\eta}^N\mapsto W^N$ defined  in \eqref{g1}. 

\begin{lemma} \label{lm:W1} 
Let $\bar{\eta}^N_t\sim \mu^{\bar{\rho}^N}$ be a stationary process  of $N$ components  evolving in basic coupling. Then 
	$W^N_t=\mathcal{W}^N(\bar{\eta}^N_t)$ is a stationary process evolving under the dynamics specified by  generator  $L^{cep}$ in \eqref{L800}.  The distribution of $W^N$ is the unique stationary one in the following sense: if  $V$ is translation-ergodic on the sequence space  $\{j2^{-N}:  j\in\lzb1,2^N\rzb\}^\Z$ with uniform marginals  and stationary under the generator  $L^{cep}$, then $V\sim W^N$. 
\end{lemma}
\begin{proof}  The stationarity follows because the map $\mathcal{W}^N$ commutes with the pathwise evolution under the Poisson clocks. This is readily verified through a picture, see  Figure \ref{Fig:Com}.  The point  is that when a jump from $x$ to $y$ is attempted,  $W^N(x)$ and $W^N(y)$ are exchanged iff $W^N(x)<W^N(y)$, while particles in the configuration $\bar{\eta}^N$ move from $x$ to $y$ iff there are more particles at $x$ than at $y$. As a consequence, the relation $W^N=\mathcal{W}^N(\bar{\eta}^N)$ is preserved by each jump.  

The inverse of the map in \eqref{g1}, that is,  $\eta_j(x)=\ind\{j2^{-N}\ge V(x)\}$,  turns the distribution of  $V$ into a multicomponent stationary distribution $\wt\mu$ on the space $(\{0,1\}^\Z)^{2^N}$.  Its marginals are translation-ergodic stationary distributions for the exclusion process \eqref{L780}, hence i.i.d.\ Bernoulli distributions by Theorem 3.9(a) of \cite{liggett1985interacting}.  By the uniqueness of multicomponent stationary measures discussed above, $\wt\mu$ must equal $\mu^{\bar\rho^N}$ and hence $V\sim W^N$. 
\end{proof}

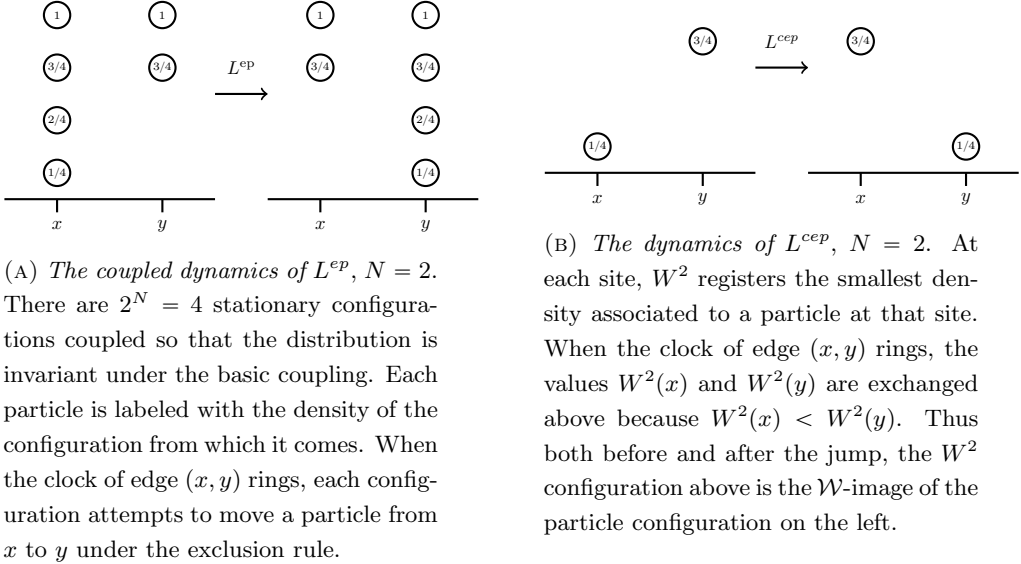
\begin{figure}
			\centering
			 \begin{subfigure}{0.35\textwidth}
			\begin{tikzpicture}[thick,scale=0.7, every node/.style={scale=0.7}]
				\draw (-2,0) -- (2,0);
				\draw (-1,0) -- (-1,-0.25);
				\node[]  at (-1,-.5){$x$};
				\node[]  at (1,-.5){$y$};
				\draw (1,0) -- (1,-0.25);
				\node[circle,draw,inner sep=0mm,minimum size=0.5cm] (c) at (-1,.5){\tiny$1/4$};
				\node[circle,draw,inner sep=0mm,minimum size=0.5cm,minimum size=0.5cm] (c) at (-1,1.5){\tiny$2/4$};
				\node[circle,draw,inner sep=0mm,minimum size=0.5cm] (c) at (-1,2.5){\tiny$3/4$};
				\node[circle,draw,inner sep=0mm,minimum size=0.5cm] (c) at (-1,3.5){\tiny$1$};

				\node[circle,draw,inner sep=0mm,minimum size=0.5cm] (c) at (1,2.5){\tiny$3/4$};
				\node[circle,draw,inner sep=0mm,minimum size=0.5cm] (c) at (1,3.5){\tiny$1$};
				\draw[->] (2,2) -- (3,2);
				\node[] (c) at (2.5,2.5){$L^{\text{ep}}$};
				\begin{scope}[shift={(5,0)}]
						\draw (-2,0) -- (2,0);
					\draw (-1,0) -- (-1,-0.25);
					\node[]  at (-1,-.5){$x$};
					\node[]  at (1,-.5){$y$};
					\draw (1,0) -- (1,-0.25);
					\node[circle,draw,inner sep=0mm,minimum size=0.5cm] (c) at (1,.5){\tiny$1/4$};
					\node[circle,draw,inner sep=0mm,minimum size=0.5cm] (c) at (1,1.5){\tiny$2/4$};
					\node[circle,draw,inner sep=0mm,minimum size=0.5cm] (c) at (-1,2.5){\tiny$3/4$};
					\node[circle,draw,inner sep=0mm,minimum size=0.5cm] (c) at (-1,3.5){\tiny$1$};
					
					\node[circle,draw,inner sep=0mm,minimum size=0.5cm] (c) at (1,2.5){\tiny$3/4$};
					\node[circle,draw,inner sep=0mm,minimum size=0.5cm] (c) at (1,3.5){\tiny$1$};
					\end{scope}
					\end{tikzpicture}
				\caption{\textit{The coupled dynamics of $L^{ep}$}, $N=2$. There are $2^N=4$ stationary configurations coupled so that the distribution is invariant under the basic coupling. Each particle is labeled with the density of the configuration from which it comes. When the clock of edge $(x,y)$ rings, each configuration attempts to move  a particle  from $x$ to $y$ under the exclusion rule.}
				 \end{subfigure}\hspace{4em}%
			 \begin{subfigure}{0.35\textwidth}
			 		\begin{tikzpicture}[thick,scale=0.7, every node/.style={scale=0.7}]
				\begin{scope}[shift={(0,-6)}]
					\draw (-2,0) -- (2,0);
					\draw (-1,0) -- (-1,-0.25);
					\node[]  at (-1,-.5){$x$};
					\node[]  at (1,-.5){$y$};
					\draw (1,0) -- (1,-0.25);
					\node[circle,draw,inner sep=0mm,minimum size=0.5cm] (c) at (-1,.5){\tiny$1/4$};
					\node[circle,draw,inner sep=0mm,minimum size=0.5cm] (c) at (1,2.5){\tiny$3/4$};
				\draw[->] (2,2) -- (3,2);
				\node[] (c) at (2.5,2.5){$L^{cep}$};
					\begin{scope}[shift={(5,0)}]
						\draw (-2,0) -- (2,0);
						\draw (-1,0) -- (-1,-0.25);
						\node[]  at (-1,-.5){$x$};
						\node[]  at (1,-.5){$y$};
						\draw (1,0) -- (1,-0.25);
						\node[circle,draw,inner sep=0mm,minimum size=0.5cm] (c) at (1,.5){\tiny$1/4$};
						
						\node[circle,draw,inner sep=0mm,minimum size=0.5cm] (c) at (-1,2.5){\tiny$3/4$};

					\end{scope}
				\end{scope}
			\end{tikzpicture}
		\caption{\textit{The dynamics of $L^{cep}$}, $N=2$.  At each site, $W^2$ registers the smallest density associated to a particle at that site. When the clock of edge $(x,y)$ rings,  the values $W^2(x)$ and $W^2(y)$ are exchanged above because $W^2(x)<W^2(y)$. Thus both before and after the jump, the $W^2$ configuration above is the $\mathcal W$-image of the particle configuration on the left. }
		 \end{subfigure}
			\caption{\small The effect of a Poisson clock ring on edge $(x,y)$, to illustrate the commutation of $\mathcal{W}^N$ with the evolution of $\bar\eta^N$ (on the left in diagram (A)) and $W^N$ (on the right in diagram (B)).  
   }.
   \label{Fig:Com}
		\end{figure}

\begin{lemma}\label{lm:W2}
	There exists a random configuration  $W\in[0,1]^{\Z}$ such that $W^N\Rightarrow W$ as  $N\to\infty$. $W$ is translation-invariant and has uniform marginals. 
\end{lemma}
\begin{proof}
 Let $\bar\eta^{N+m}\sim\mu^{\bar\rho^{N+m}}$ as in \eqref{g2}. Since projection commutes with the evolution and preserves Bernoulli marginals, by the uniqueness discussed above   we can define a version of $\bar\eta^{N}\sim\mu^{\bar\rho^{N}}$ by the projection 
\be\label{W56} \nonumber \bar{\eta}^N=\big({\eta}^{N+m}_{1\cdot 2^m},{\eta}^{N+m}_{2\cdot 2^{m}},{\eta}^{N+m}_{3\cdot 2^{m}},\ldots,{\eta}^{N+m}_{2^{N}\cdot 2^{m}}\big). 
\ee
Thus we have  a coupling $\mu_{N,N+m}$ of $\bar{\eta}^N$ and $\bar{\eta}^{N+m}$ such that 
\begin{equation}\label{eq1'}
\mu_{N,N+m}\Big\{
\big({\eta}^{N+m}_{1\cdot 2^m},{\eta}^{N+m}_{2\cdot 2^{m}},{\eta}^{N+m}_{3\cdot 2^{m}},\ldots,{\eta}^{N+m}_{2^{N}\cdot 2^{m}}\big)
= \bar{\eta}^N\Big\}=1.
\end{equation}
As   $i\mapsto \bar{\eta}^{N+m}_i(x)$ is nondecreasing and $\{0,1\}$-valued,  it follows from \eqref{eq1'} that 
\begin{equation}\nonumber
	\mu_{N,N+m}\Big( \sup_{x\in\Z}|W^N(x)-W^{N+m}(x)|\leq 2^{-N}\Big)=1.
\end{equation}

Thus $d_{\text{Prok}}(W^{N},W^{N+m})\leq 2^{-N}$ where $d_{\text{Prok}}$ is the Prokhorov metric on the space  of probability measures  on $[0,1]^\Z$ and we equip the space $[0,1]^\Z$ with the product metric $d(\zeta,\eta)=\sum_{x\in\Z} 2^{-|x|-2} |\zeta(x)-\eta(x)|$. 
Thus, $\{W^N\}_{N\in\N}$ is Cauchy under $d_{\text{Prok}}$ and by completeness there is a random variable $W$ such that $d_{\text{Prok}}(W^{N},W)\to0$.   $W$ inherits translation-invariance and uniform marginals from the  $W^N$s. 
\end{proof}


	\begin{proof}[Proof of Theorem \ref{exun}]
Translation-invariance and uniform marginals of $W$ are in Lemma \ref{lm:W2}. 
  If $f$ is a continuous local function on $[0,1]^\Z$, then 
  $L^{cep}f$ is a bounded continuous function.  Hence by the invariance of Lemma \ref{lm:W1},  
  $\E[L^{cep}f(W)]=\lim_{N\to\infty} \E[L^{cep}f(W^N)]=0$, and the invariance of the distribution of $W$ has been proved.

   

It remains to establish the uniqueness of $W$. For $N\in \N$, define the function
\begin{equation}
	F^N(v)=\sum_{i=1}^{2^N} i2^{-N}\cdot\ind_{\big((i-1)2^{-N},\, i2^{-N}\big]}(v), \qquad v\in [0,1].
\end{equation}
Define $W^N$ and $W^{N+m}$  as images \eqref{g1}  of $\bar\eta^{N}$ and $\bar\eta^{N+m}$ in the coupling \eqref{eq1'}.  Then $F^N(W^{N+m})=W^{N}$. Since the coordinates of $W$ are uniform, $W$ avoids the discontinuity set of $F^N$ almost surely. By sending $m\to\infty$ we get the distributional equality
$F^N(W)\sim W^N$.

Suppose $V$ is translation-ergodic and distributed according to a   stationary measure for  $L^{cep}$ with marginals uniform on $[0,1]$. We must show that $V\sim W$. Define $V^N=F^N(V)$ and note that $V^N$ is translation-ergodic and  $V^N\rightarrow V$ a.s. (and therefore in distribution) as $N\rightarrow \infty$.   It is therefore enough to show that $V^N\sim W^N$ for every $N\in \N$.
This follows from the uniqueness part of Lemma \ref{lm:W1}. 
	\end{proof}
 
The distribution of $W$ in Theorem \ref{exun} was constructed as a limit of its discretizations without using multiclass particles. We remark here that some of the ideas used in the  proof of Theorem \ref{exun} have appeared in the literature before. Specifically, a discretization and a limiting scheme similar to the one used in the proof of Theorem \ref{exun} was used in \cite{mart-20} for the ASEP on the torus.  Next we relate  $W$   to the speed process $U^{\text{spd}}$.
Since a speed process has not been constructed in the generality of this section, we proceed by assuming its existence and stationarity under the dynamics.  Then we show that  $U^{\text{spd}}\sim\phi(W)$ for a determistic map $\phi$. 
Starting from the profile $\eta(i)=i$ for $ i\in \Z$, apply the multiclass exclusion dynamics with kernel $p(\abullet,\abullet)$ and stipulate that particle $i$ has priority over all particles $j>i$. 
  Let $X_i(t)$ be the position of particle $i$ at time $t$. 
\begin{assumption}\label{assu1}
	With probability one and for some $M>0$, the following limit exists
	\begin{equation}
		U^{\text{spd}}_i:=\lim_{t\rightarrow \infty}\frac{X_i(t)}{t}\in[-M,M]. 
	\end{equation}

\end{assumption}  
\begin{assumption}\label{assu2}
	The distribution of the speed process $U^{\text{spd}}:=\{U^{\text{spd}}_i\}_{i\in\Z}$ is stationary under the multitype exclusion dynamics.
\end{assumption}
Assumptions \ref{assu1}--\ref{assu2}  are natural and they hold for the ASEP 
\cite{Amir_Angel_Valko11,aggarwal2022asep}.
 
\begin{definition}
	Let $F_{\text{spd}}$ denote the CDF of $U^{\text{spd}}_0$.  The process $U^W\in [-M,M]^\Z$ of the exclusion process with dynamics $L^{ep}$
 is defined by 
	\begin{equation}
		U^W:= F_{\text{spd}}^{-1}(W):=\{F_{\text{spd}}^{-1}(W_i)\}_{i\in\Z}
	\end{equation}
where
 $F_{\text{spd}}^{-1}$ is the generalized inverse function.
\end{definition}
\begin{corollary}
	$U^W$ is translation-invariant and stationary under $L^{cep}$.  
\end{corollary}


The stationarity follows because the pathwise dynamics commutes with any coordinatewise applied nondecreasing function. 
 Our final result connects  $W$ with $U^{\text{spd}}$.
\begin{proposition}\label{NL}
	Suppose Assumptions \ref{assu1}--\ref{assu2} hold. Then $U^{\text{spd}}\sim U^W$.
\end{proposition}
\begin{proof}
     Assumption \ref{assu1} implies that $U^{\text{spd}}$ is translation-ergodic   (the idea is in  \cite[Proposition 5.1]{Amir_Angel_Valko11}).  $V:=F_{\text{spd}}\big(U^{\text{spd}}\big)$ is translation-ergodic, stationary under  $L^{cep}$, and has uniform marginals on $[0,1]$. By Theorem \ref{exun} $V\sim W$. This implies the result. 
\end{proof}

\appendix

\section{Random walk}
We first state a random walk lemma that comes from p.~519--520 in \cite{resnick}.
See also Chapter VIII, Section 6 in~\cite{Asmussen-1987}. In~\cite{resnick} and~\cite{Asmussen-1987} the result is stated for $\mu_k < 0$ and $\sigma_N \to 1$ and supremum is taken over positive time. Our formulation    follows by Brownian scaling and by replacing $x$ with $-x$.



 
\begin{lemma} \label{lem:genres}
Let $\mu_N$ be a sequence of strictly positive numbers with $\mu_N \to 0$. Let $\sigma_N$ be a sequence satisfying $\sigma_N \to \sigma > 0$. Let $\varphi(N)$ be a sequence satisfying $\mu_N/\varphi(N) \to m > 0$. For each $N$, let $\{X_{N,i}:i \in \Z\}$ be a collection of i.i.d. random variables with mean $\mu_N$ and variance $\sigma_N^2$. Further, suppose that the sequence $\{X_{N,0}^2:N \ge 1\}$ is uniformly integrable. Let $S^N(m)$ be defined as
\be \label{SNm}
S^N(m) = \begin{cases}
 -\sum_{i = m}^{-1} X_{N,i} &m \le 0 \\
  \sum_{i = 0}^{m - 1} X_{N,i} &m \ge 0
\end{cases}
\ee
with   $S^N(0) = 0$. Let $B$ be a Brownian motion with diffusion coefficient $1$ and zero drift. Then, the following convergence in distribution holds:
\be \label{convBMgen}
\sup_{-\infty < x \le 0} \varphi(N) S^N(\ce{x}) \underset{N \to \infty}{\Longrightarrow} \sup_{-\infty < x \le 0} \{\sigma B(x) + m x\}
\ee
\end{lemma}
\begin{remark}
It is immediate that on the left-hand side of~\eqref{convBMgen}, one can replace $x$ with $\lceil \xi(N) x\rceil$ for any strictly positive sequence $\xi(N)$.  
\end{remark}

Let $[x]$ denote the integer closest to $x$ with $|[x]| \le |x|$. 

\begin{lemma}\label{lem:convprob}
Consider the setting of Lemma~\ref{lem:genres}. Let $\xi(N)$ be a sequence satisfying $\varphi(N)^2 \xi(N) \to R > 0$. Then, for each $S < T \in \R$,
\be \label{rwconv}
\begin{aligned}
&\,\lim_{N \to \infty} \Pp\Bigl[\,\sup_{-\infty < x \le [S \xi(N)]} \varphi(N) S^N(x) > \sup_{[S \xi(N)] \le x \le [T \xi(N)]} \varphi(N) S^N(x)\Bigr] \\
&\qquad\qquad 
= \Pp\bigl[\sup_{-\infty < x \le S}\{\sigma B(R x) + mRx\} > \sup_{S \le x \le T}\{\sigma B(Rx) + mRx\}\bigr].
\end{aligned}
\ee

\end{lemma}
\begin{proof}
For a function $f:\R \to \R$ let $f(x,y) = f(y) - f(x)$. Observe that 
\begin{align*}
&\Pp\Bigl[\,\sup_{-\infty < x \le [S \xi(N)]} \varphi(N) S^N(x) > \sup_{[S \xi(N)] \le x \le [T \xi(N)]} \varphi(N) S^N(x)\Bigr] \\
&= \Pp\Bigl[\,\sup_{-\infty < x \le [S \xi(N)]} \varphi(N) S^N([S\xi(N)],x) > \sup_{[S \xi(N)] \le x \le [T \xi(N)]} \varphi(N) S^N([S\xi(N)],x)\Bigr].
\end{align*}
Now, note that $\sup_{-\infty < x \le [S \xi(N)]} \varphi(N) S^N([S\xi(N)],x)$ and $\sup_{[S \xi(N)] \le x \le [T \xi(N)]} \varphi(N) S^N([S\xi(N)],x)$ are independent. By convergence of random walk to Brownian motion with drift (with respect to  the topology of uniform convergence on compact sets) , we get that 
\[
\sup_{[S \xi(N)] \le x \le [T \xi(N)]} \varphi(N) S^N(RS,[S\xi(N)],x) \Longrightarrow \sup_{S \le x \le T} \{\sigma B(RS, Rx) + R(m-S)x\}.
\]
By shift invariance of random walk and Lemma~\ref{lem:genres},
\begin{align*}
&\,\sup_{-\infty < x \le [S \xi(N)]} \varphi(N) S^N([S\xi(N)],x) \deq \sup_{-\infty < x \le 0} \varphi(N) S^N(x)  \\
&\Longrightarrow \sup_{-\infty < x \le 0}\{\sigma B(x) + mx\} = \sup_{-\infty < x \le 0} \{\sigma B(Rx) + mRx\} \\
&\deq \sup_{-\infty < x \le S} \{\sigma B(RS,Rx) + m(R-S)x\}. 
\end{align*}
By independence, we have shown the following joint convergence:
\be \label{jc}
\begin{aligned}
&\, \Bigl(\sup_{-\infty < x \le [S \xi(N)]} \varphi(N) S^N([S \xi(N)],x), \sup_{[S \xi(N)] \le x \le [T \xi(N)]} \varphi(N) S^N([S \xi(N)],x)\Bigr)  \\
&\Longrightarrow \Bigl(\sup_{-\infty < x \le S}\{\sigma B(RS,R x) + m(R-S)x\}, \sup_{S \le x \le T}\{\sigma B(RS,Rx) + m(R-S)x\}\Bigr).
\end{aligned}
\ee

The right-hand side of~\eqref{jc} consists of two independent random variables with continuous distribution. Therefore, 
\begin{align*}
&\quad \lim_{N \to \infty} \Pp\Bigl[\,\sup_{-\infty < x \le [S \xi(N)]} \varphi(N) S^N(x) > \sup_{[S \xi(N)] \le x \le [T \xi(N)]} \varphi(N) S^N(x)\Bigr] \\
&= \lim_{N \to \infty}\Pp\Bigl[\,\sup_{-\infty < x \le [S \xi(N)]} \varphi(N) S^N([S \xi(N)],x) > \sup_{[S \xi(N)] \le x \le [T \xi(N)]} \varphi(N) S^N([S \xi(N)] ,x)\Bigr] \\
& =  \Pp\bigl[\sup_{-\infty < x \le S}\{\sigma B(RS,R x) + m(R-S)x\} > \sup_{S \le x \le T}\{\sigma B(Rx) + m(R-S)x\}\bigr]
\\ &
= \Pp\bigl[\sup_{-\infty < x \le S}\{\sigma B(R x) + mRx\} > \sup_{S \le x \le T}\{\sigma B(Rx) + mRx\}\bigr],
\end{align*}
with the second equality holding because the event on the right-hand side is a continuity set for the joint vector on the right in~\eqref{jc}. 
\end{proof}

\section{Discrete-time M/M/1 queues}
\label{app:MM1}

Notational comment: the input and output sequences in our queuing setting are elements $\bq{x}=\{x(i)\}_{i\in\Z}$  of the space $\Qs_1=\{1,\infty\}^\Z$, where the value $x(i)=\infty$ signifies that site (time point) $i$ is empty.  For the purpose of counting particles  it is convenient to replace $\infty$ with zero. We use  bracket  notation $x[i]$ to denote the  corresponding  $\{0,1\}$-valued configuration and to count the number of particles in the interval $[i,j]$ as follows: 
\be\label{notat8}   x[i]= \ind_{x(i)=1} =\begin{cases} 0, &x(i)=\infty \\1, &x(i)=1 \end{cases} 
\qquad\text{and}\qquad 
x[i,j]=\sum_{k=i}^j x[k].   \ee
Obviously then also $x[i,i]=x[i]$. 
Define the usual coordinatewise partial order $\preceq$ on $\Qs_1$ by
\begin{equation*}
	\bq{x}_1\preceq \bq{x}_2 \quad \iff \quad \bigl[\, \forall i\in\Z:  \  x_1(i)=1 \implies x_2(i)=1\,\bigr] \quad \iff \quad  \bigl[\, \forall i\in\Z:  \  x_1[i]\le x_2[i]\,\bigr].
\end{equation*}
Introduce also notation for truncating sequences by setting them empty to the left of time $n$: 
\begin{equation}\label{notat67} 
	x_{n,0}(i)=
	\begin{cases}
		x(i), & i\geq n\\
		\infty, & i\leq n-1
	\end{cases}
	\qquad\iff\qquad x_{n,0}[i]=x[i]\cdot\ind_{i\ge n}. 
\end{equation}


	\begin{lemma}[Burke property] \label{lm:Burke}
		Let $0<\alpha<\alpha+\beta<1$ and  $(\arrv,\srvv)\sim \nu^{(\alpha, \beta)}$. Let  $\depv=D(\arrv,\srvv)$ and $\bq{r}=R(\arrv,\srvv)$. Then  for any $i_0\in\Z$, the random variables 
		\begin{equation}\label{eq11}
			\{d(j)\}_{j\leq{i_0}}, \,\,\{r(j)\}_{j\leq{i_0}} \,\, \text{ and } \,\, Q_{i_0}(\arrv,\srvv)
		\end{equation} 
		are mutually independent with marginal distributions $d(j)\sim{\rm Ber}(\alpha)$, $r(j)\sim {\rm Ber}(\alpha+\beta)$, and $Q_{i_0}\sim{\rm Geom}(\gamma)$ with $\gamma=\frac{\beta}{(1-\alpha)(\alpha+\beta)}$.   
		Furthermore,   
		\be\label{disId}(\depv,\bq{r})\sim(\arrv,\srvv)\sim \nu^{(\alpha, \beta)}.\ee 
	\end{lemma}
\begin{proof}   Here is a sketch  of a simple proof. The structure of the queuing mappings together with the independent Bernoulli product  inputs $(\arrv, \srvv)$  imply that $\{(Q_j, d(j), r(j))\}_{j\in\Z}$  is a stationary, irreducible,  recurrent   Markov chain. Observe that the joint product distribution ${\rm Geom}(\gamma)\otimes{\rm Ber}(\alpha)\otimes{\rm Ber}(\alpha+\beta)$ is preserved by  the mapping $(Q_{j-1}, a(j), s(j)) \mapsto (Q_j, d(j), r(j))$, and it is the stationary distribution of the Markov chain  $\{(Q_j, d(j), r(j))\}_{j\in\Z}$.  For any fixed $j_0$, the joint independence of 
$\{d(j)\}_{j\in\lzb j_0,i_0\rzb}$, $\{r(j)\}_{j\in\lzb j_0,i_0\rzb}$  and $ Q_{i_0}(\arrv,\srvv)$ can now be checked by induction on $i_0$. The base case $i_0=j_0$ comes from the stationarity of   the Markov chain.  Letting $j_0\searrow-\infty$ gives the full distributional claim for \eqref{eq11}.

The fixed-point property \eqref{disId} 
comes by letting $i_0\nearrow\infty$. A different proof is given in Theorem 4.1 of \cite{koni-ocon-roch-02}.
\end{proof}

We go through two auxiliary lemmas on the way to Proposition \ref{pr:q259}. 
 
\begin{lemma} \label{lm:qq99}  {\rm\cite[Lemmas 8.1 and 8.2]{martin2010fixed}} Consider two queues in tandem with arrivals $\arrv$, service sequences $\srvv_1$ and $\srvv_2$, and  departures  $\depv=D(\arrv, \srvv_1, \srvv_2)$.  

{\rm(i)}   Recall the notation  \eqref{notat67}  of  the truncated arrival sequence $\arrv_{n,0}$. 
 For $n\in\Z$  denote the departures by $\depv^{(n)}=D(\arrv_{n,0}, \srvv_1, \srvv_2)$.   Then for each $i\in\Z$ there exists $n_0(i)\in\Z$ such that  $d^{(n)}[i]=d[i]$ for all $n\le n_0(i)$. 

\smallskip 

{\rm(ii)}  Let $k\in\Z$.  Suppose $a[i]=0$ for all $i\le k-1$.  Then for all $t\ge k+1$, 
\be\label{qq95}  \sum_{i=k}^{t-1} d[i] = \min_{\ell,v: \, k\le \ell\le v\le t} \Bigl\{ \; \sum_{i=k}^{\ell-1} a[i] + \sum_{i=\ell}^{v-1} s_1[i] +   \sum_{j=v}^{t-1} s_2[j]  \Bigr\}. 
\ee

\end{lemma} 	

The proof of the next lemma relies on the  ideas from p.~16--17 of 	\cite{martin2010fixed}. 
			
\begin{lemma}   Let $\arrv$ and  $\srvv$ be arrival and  service sequences,   
$\depv=D(\arrv,\srvv)$ and $\bq{r}=R(\arrv,\srvv)$.
Then for all $s<t$ in $\Z$, 
\be\label{qq100}\begin{aligned}  
\min_{\ell\in\lzb s,t\rzb}\biggl\{  \; \sum_{i=s}^{\ell-1} a[i] + \sum_{j=\ell}^{t-1} s[j]\biggr\}
=
\min_{\ell\in\lzb s,t\rzb}\biggl\{  \; \sum_{i=s}^{\ell-1} r[i] + \sum_{r=\ell}^{t-1}  d[j]\biggr\}
\end{aligned} \ee

\begin{proof}  In the first step, we prove 
\be\label{qq108}\begin{aligned}  
\min_{\ell\in\lzb s,t\rzb}\biggl\{  \; \sum_{i=s}^{\ell-1} a[i] + \sum_{j=\ell}^{t-1} s[j]\biggr\}
=
\min_{\ell\in\lzb s,t\rzb}\biggl\{  \; \sum_{i=s}^{\ell-1} a[i] + \sum_{r=\ell}^{t-1}  d[j]\biggr\}
\end{aligned} \ee
If there are no unused services in $\lzb s, t-1\rzb$, then $\srvv=\depv$ throughout the interval and \eqref{qq108} holds.  In general, since $\srvv\ge \depv$,  we have $\ge$ in \eqref{qq108}. 

It remains to consider the case where  there are unused services.  Let $n$ be the time of the last unused service in $\lzb s, t-1\rzb$.  Then for each $k\in\lzb s,n\rzb$, 
\[   \sum_{i=k}^{n} a[i] \le  \sum_{j=k}^{n} d[j] ,  \quad\text{equivalently, }\quad   0 \ge  \sum_{i=k}^{n} a[i]  -  \sum_{j=k}^{n} d[j]   \]
because otherwise one of the arrivals would still be in the system after time $n$ and  there could not have been an unused service at time $n$.  Then for  $\ell\le n$, the  above inequality gives 
\begin{align*}    \sum_{i=s}^{\ell-1} a[i] + \sum_{r=\ell}^{t-1}  d[j]  
\ge   \sum_{i=s}^{\ell-1} a[i] + \sum_{r=\ell}^{t-1}  d[j]   
+ \biggl\{  \; \sum_{i=\ell}^{n} a[i]  -  \sum_{j=\ell}^{n} d[j]\biggr\}  
=  \sum_{i=s}^{n} a[i] + \sum_{r=n+1}^{t-1}  d[j]  . 
\end{align*} 
This implies that there is a  minimizer $\ell^*$ of the right-hand side of \eqref{qq108} that satisfies $\ell^*\in\lzb n+1,t\rzb$. Since $\srvv=\depv$ throughout $\lzb n+1,t-1\rzb$,  at $\ell=\ell^*$ the two expressions in braces in \eqref{qq108} agree. Hence   we have $\le$ in \eqref{qq108}, and \eqref{qq108} has been verified.  

Next from \eqref{qq108} we derive 
\be\label{qq118}\begin{aligned}  
\min_{\ell\in\lzb s,t\rzb}\biggl\{  \; \sum_{i=s}^{\ell-1} a[i] + \sum_{j=\ell}^{t-1} d[j]\biggr\}
=
\min_{\ell\in\lzb s,t\rzb}\biggl\{  \; \sum_{i=s}^{\ell-1} (a[i]+u[i])  + \sum_{r=\ell}^{t-1}  d[j]\biggr\}
\end{aligned} \ee
which completes the proof of the lemma.  Again if $\bq{u}=0$ throughout the interval then \eqref{qq118} holds, and in general we have $\le$ in \eqref{qq118}. 

Suppose now that $m$ is the time of the first unused service in $\lzb s,t-1\rzb$.  This implies that the queue is empty after the service at time $m$, and afterwards the departures cannot outnumber the arrivals:  for each $n\ge m+1$, 
$   \sum_{i=m+1}^{n} a[i] \ge  \sum_{j=m+1}^{n} d[j]. $
Furthermore, the unused service forces $d[m]=0$, and hence for all $n\ge m$, 
\[   \sum_{i=m}^{n} a[i] \ge  \sum_{j=m}^{n} d[j]. \]  
Now consider $\ell\in\lzb m,t-1\rzb$ on the left-hand side of \eqref{qq118}: 
\begin{align*} 
\sum_{i=s}^{\ell-1} a[i] + \sum_{j=\ell}^{t-1} d[j]  = \sum_{i=s}^{m-1} a[i] + \sum_{j=m}^{t-1} d[j] + \biggl\{ \; \sum_{i=m}^{\ell-1} a[i] -  \sum_{j=m}^{\ell-1} d[j] \biggr\}    \ge \sum_{i=s}^{m-1} a[i] + \sum_{j=m}^{t-1} d[j] . 
\end{align*} 
Thus there is a  minimizer $\ell'$ of  the left-hand side of \eqref{qq118} that satisfies $\ell'\in\lzb s,m-1\rzb$. On this range $\bq{u}=0$. We conclude that $\ge$ holds in \eqref{qq118}.
\end{proof} 

\end{lemma} 			
			
%
%
%

\begin{proposition} \label{pr:q259}  
The tandem queuing maps have these properties.  

{\rm (i)}   For all $n\ge 3$ and  $k\in\lzb1,n-1\rzb$, 
\be\label{q258}
D^n(\bq{x}_1,\bq{x}_2,\dotsc,\bq{x}_n)=D^{k+1}\big(D^{n-k}(\bq{x}_1,\dotsc,\bq{x}_{n-k}), \bq{x}_{n-k+1},\dotsc, \bq{x}_n\big). 
\ee

{\rm (ii)}  For all $n\ge 3$ and  $k\in\lzb2,n-1\rzb$, 
\be\label{q259}  D^n(\bq{x}_1,\bq{x}_2,\dotsc,\bq{x}_n) = D^n\bigl(\bq{x}_1,\dotsc, \bq{x}_{k-1}, \,  R(\bq{x}_k,\bq{x}_{k+1}), \,  D(\bq{x}_k,\bq{x}_{k+1}), \, \bq{x}_{k+2},\dotsc, \bq{x}_n \bigr). 
\ee
Note that the initial segment $\bq{x}_1,\dotsc, \bq{x}_{k-1}$ is not allowed to be empty, but the final segment $\bq{x}_{k+2},\dotsc, \bq{x}_n$ does disappear in the case $k=n-1$. 
\end{proposition} 

\begin{proof}   Part (i).  The case $k=1$ is the definition of $D^n$  and the case $k=n-1$ is a tautology. Hence the case $n=3$ holds.   Let $n\ge4$ and assume that part (i) holds for $n-1$.  Let $k\in\lzb2,n-2\rzb$.
\begin{align*}
D^n(\bq{x}_1,\bq{x}_2,\dotsc,\bq{x}_n) 
&\overset{\eqref{eq18}}=  D\big(D^{n-1}(\bq{x}_1,\bq{x}_2,\dotsc,\bq{x}_{n-1}),\bq{x}_n\big)   \\
&=  D\big( D^k\bigl[   D^{n-k}(\bq{x}_1,\dotsc,\bq{x}_{n-k}),  \bq{x}_{n-k+1},\dotsc, \bq{x}_{n-1})\bigr] ,  \bq{x}_n\big) \\
&\overset{\eqref{eq18}}= D^{k+1}\big(D^{n-k}(\bq{x}_1,\dotsc,\bq{x}_{n-k}), \bq{x}_{n-k+1},\dotsc, \bq{x}_n\big). 
\end{align*} 

\smallskip 

Part (ii).  \textit{Step 1.}  We prove  the case $n=3$ of \eqref{q259}.    The task is to show 
\begin{equation}\label{q250} 
 D(\bq{x}_1,\bq{x}_2,\bq{x}_3)=D\big(\bq{x}_1,\, R(\bq{x}_2,\bq{x}_3),\,D(\bq{x}_2,\bq{x}_3)\big). 	
			  \end{equation}

By part (i) of Lemma \ref{lm:qq99} it suffices to treat the case where there exists $k\in\Z$ such that  $x_1[i]=0$ for $i\le k-1$ and then let $k\searrow-\infty$. 	Then by \eqref{qq95}, for $t\ge k+1$, 
\begin{align*}
\sum_{i=k}^{t-1} D_i(\bq{x}_1,\bq{x}_2,\bq{x}_3) 
&=  \min_{\ell: \, k\le \ell\le  t} \Bigl\{ \; \sum_{i=k}^{\ell-1} x_1[i] + \min_{v: \,\ell\le v\le t}  \Bigl[ \;  \sum_{i=\ell}^{v-1} x_2[i] +   \sum_{j=v}^{t-1} x_3[j] \Bigr]  \Bigr\} \\
&\overset{\eqref{qq100}}=  \min_{\ell: \, k\le \ell\le  t} \Bigl\{ \; \sum_{i=k}^{\ell-1} x_1[i] + \min_{v: \,\ell\le v\le t}  \Bigl[ \;  \sum_{i=\ell}^{v-1} R_i(\bq{x}_2,\bq{x}_3) +   \sum_{j=v}^{t-1} D_j(\bq{x}_2,\bq{x}_3) \Bigr]  \Bigr\} \\
&=\sum_{i=k}^{t-1} D_i\big(\bq{x}_1,\, R(\bq{x}_2,\bq{x}_3),\,D(\bq{x}_2,\bq{x}_3)\big). 
\end{align*}

\textit{Step 2.}   We prove  the case $k=n-1$   for all $n\ge 3$: 
\be\label{q260}  D^n(\bq{x}_1,\bq{x}_2,\dotsc,\bq{x}_n) = D^n\bigl(\bq{x}_1,\dotsc, \bq{x}_{n-2}, \,  R(\bq{x}_{n-1},\bq{x}_{n}), \,  D(\bq{x}_{n-1},\bq{x}_{n}) \bigr). 
\ee
 The case $n=3$ is in \eqref{q250}.  By \eqref{q258}, the case $n=3$ of \eqref{q260}, and again by \eqref{q258}, 
 \begin{align*}
 D^n(\bq{x}_1,\bq{x}_2,\dotsc,\bq{x}_n) 
 &=D^3\bigl(   D^{n-2}(\bq{x}_1,\bq{x}_2,\dotsc,\bq{x}_{n-2}) , \bq{x}_{n-1}, \bq{x}_{n} \bigr)  \\
 &=D^3\bigl(   D^{n-2}(\bq{x}_1,\bq{x}_2,\dotsc,\bq{x}_{n-2}) , \,  R(\bq{x}_{n-1},\bq{x}_{n}), \,  D(\bq{x}_{n-1},\bq{x}_{n}) \bigr)  \\
 &= D^n\bigl(\bq{x}_1,\dotsc, \bq{x}_{n-2}, \,  R(\bq{x}_{n-1},\bq{x}_{n}), \,  D(\bq{x}_{n-1},\bq{x}_{n}) \bigr). 
 \end{align*} 
 
 \textit{Step 3.}   We complete the proof of \eqref{q259} by taking $n\ge4$ and   $k\in\lzb2,n-2\rzb$: 
 \begin{align*}
 D^n(\bq{x}_1,\bq{x}_2,\dotsc,\bq{x}_n) 
 &\overset{\eqref{q258}}=D^{n-k}\bigl(   D^{k+1}[\bq{x}_1,\dotsc ,\bq{x}_{k},\bq{x}_{k+1}] , \bq{x}_{k+2}, \dotsc, \bq{x}_{n} \bigr)  \\
 &\overset{\eqref{q260}}=D^{n-k}\bigl(   D^{k+1}[\bq{x}_1,\dotsc , \bq{x}_{k-1}, R(\bq{x}_{k},\bq{x}_{k+1}), D(\bq{x}_{k},\bq{x}_{k+1}) ]  , \bq{x}_{k+2}, \dotsc, \bq{x}_{n} \bigr)  \\
 &\overset{\eqref{q258}}=D^n\bigl(\bq{x}_1,\dotsc, \bq{x}_{k-1}, \,  R(\bq{x}_k,\bq{x}_{k+1}), \,  D(\bq{x}_k,\bq{x}_{k+1}), \, \bq{x}_{k+2},\dotsc, \bq{x}_n \bigr). 
 \end{align*} 
\end{proof}

For $\bq{x}\in\Qs_1$, $\rho \in(0,1)$ and integers $i\leq j$ we denote the sequence centered at $\rho$ by 
\begin{equation}\label{rw}
	{x}^\rho[i,j]={x}[i,j]-\rho(j-i+1).
\end{equation} 
Recall the truncation notation \eqref{notat67}. For 
$\arrv,\srvv\in\Qs_1$ define the  queuing map that ignores arrivals and services before time $i$:  
\begin{equation*}
	D_{i,0}(\arrv,\srvv)=D(\arrv_{i,0},\srvv_{i,0}). 
\end{equation*}
More generally, for $n\in\N$ and $\bq{x}^1,\dotsc,\bq{x}^n\in\Qs_1$, 
\begin{equation}\label{eq42}
	D^n_{i,0}(\bq{x}^1,\dotsc,\bq{x}^n)=D_{i,0}\big(D^{n-1}_{i,0}(\bq{x}^1,\dotsc,\bq{x}^{n-1}),\bq{x}^n\big)=D^n(\bq{x}_{i,0}^1,\dotsc,\bq{x}_{i,0}^n).
\end{equation}
\begin{lemma}\label{lem:q}
	For $\arrv,\srvv\in \Qs_1$, let $\depv=D(\arrv,\srvv)$ and $\bq{r}=R(\arrv,\srvv)$. 
	\begin{enumerate}  [label={\rm(\roman*)}, ref=({\rm\roman*})]   \itemsep=1pt  
		\item  We have the inequalities 
		\begin{equation}\label{bq1}
			\depv\preceq \srvv 
   \qquad\text{and}\qquad  \arrv\preceq \bq{r}.
		\end{equation}
		\item \label{bq0}Fix $i\in\Z$. For any $n\in\N$ and $l\in [i,\infty)$, the map $D^n_{i,0}(\aabullet,\dotsc,\aabullet)[i,l]$ from \eqref{eq42} is non-decreasing in all its variables.
		\item Fix  $i\in\Z$. Let $n\geq 2$ and $\bq{x}_1,\bq{x}_2,\dotsc,\bq{x}_n\in\Qs_1$ be such that the departure processes  $\depv_n=D(\bq{x}_1,\bq{x}_2,\dotsc,\bq{x}_n)$ and $\bq{f}_{n}=D_{i,0}(\bq{x}_1,\bq{x}_2,\dotsc,\bq{x}_n)$ are well-defined. Then
		\begin{equation}\label{bq4}
				\bq{f}_n\preceq \bq{d}_n.
		\end{equation}
	Moreover, the following bounds hold 
		\begin{align}
			\max_{l\in [i,j]}\bigl\{ -{f}_n^{\rho}[i,l] \bigr\} &\leq  2\sum_{k=1}^n\max_{l\in [i,j]}|{x}_{k}^{\rho}[i,l]|\label{bq3b},\\
			\max_{l\in [i,j]}|{d}_n^{\rho}[i,l]|&\leq  2\sum_{k=1}^n\max_{l\in [i,j]}|{x}_{k}^{\rho}[i,l]|\label{bq3a}.
		\end{align}
	\end{enumerate}
\end{lemma}
\begin{proof}
 The inequalities in \eqref{bq1} 
 follow directly from the definitions. The proof of  Item \ref{bq0} is by induction, 
 note that
 \begin{equation}\label{eq43}
 	\begin{aligned}
 		{f}_2(\arrv,\srvv)[i,j]&= a[i,j]-Q_j(\arrv_{i,0},\srvv_{i,0})= a[i,j]-\max_{l\in [i,j]}\big[ a[l,j]- s[l,j]\big]^+\\
 		&= a[i,j]\wedge\min_{l\in[i,j]}\big[ a[i,l-1]+ s[l,j]\big].
 	\end{aligned}
 \end{equation}
 It is not hard to see from the display above that ${f}_2(\cdot,\cdot)[i,j]$ is indeed non-decreasing in both its variables, proving the base case $n=2$. For the induction step, assume Item \ref{bq0} holds for $n-1$, then 
  \begin{equation*}
 		{f}_n(\bq{x}_1,\dotsc,\bq{x}_n)[i,j]={f}_2\big(\bq{f}_{n-1}(\bq{x}_1,\dotsc,\bq{x}_{n-1}),\bq{x}_n\big)[i,j]
 		={f}_{n-1}[i,j]\wedge\min_{l\in[i,j]}\big[{f}_{n-1}[i,l-1]+
 		{x}_n[l,j]\big],
 \end{equation*}
which implies that $\bq{f}_n$ is indeed non-decreasing in all its variables.
    We continue to prove \eqref{bq4} by induction. Clearly, as $f_2(l)=\infty$ for $l\leq i-1$, we only need to verify that for $l\in [i,\infty]$, if $f_2(l)=1$ then $d_2(l)=1$.  Observe that the only difference between the output in $[i,\infty)$ of the queue $D_{i,0}(\arrv,\srvv)$ and $D(\arrv,\srvv)$ is that the first $Q_{i-1}(\arrv,\srvv)$ unused services after $i-1$ of the former, are replaced with departure times for the latter. This proves the base case $n=2$. For the induction step, assume \eqref{bq4} holds for $n-1$. Then
 \begin{equation*}
 \begin{aligned}
 		D(\bq{x}_1,\dotsc,\bq{x}_n)&=D\big(D(\bq{x}_1,\dotsc,\bq{x}_{n-1}),\bq{x}_n\big) \succeq D\big(D_{i,0}(\bq{x}_1,\dotsc,\bq{x}_{n-1}),\bq{x}_n\big)\\
 		&\succeq D_{i,0}\big(D_{i,0}(\bq{x}_1,\dotsc,\bq{x}_{n-1}),\bq{x}_n\big)=D_{i,0}(\bq{x}_1,\dotsc,\bq{x}_n),
 \end{aligned}
 \end{equation*}
where in the first inequality we used that $D(\cdot,\cdot)$ is increasing in the first variable. This proves \eqref{bq4}.
  Next we show \eqref{bq3b}. From \eqref{eq43}
	\begin{equation}\label{eq15}
		\begin{aligned}
			& \rho(j-i+1)-{f}_2[i,j]= \rho(j-i+1)+\bigl\{- a[i,j]\vee\max_{l\in [i,j]}\big[- a[i,l-1]- s[l,j]\big]\bigr\}\\
			&=\rho(j-i+1)+\bigl\{- a[i,j]\vee\max_{l\in [i,j]}\big[- a[i,l-1]-\big( s[i,j]- s[i,l-1]\big)\big]\bigr\}\\
			&\leq \max_{l\in [i,j]}\bigl\{-a^{\rho}[i,l]\bigr\}+2\max_{l\in [i,j]}|s^{\rho}[i,l]|
		\end{aligned}
	\end{equation}
	Applying \eqref{eq15} repeatedly gives 
	\[
				\max_{k\in [i,j]}-{f}_2^{\rho}[i,k]\leq \max_{l\in [i,j]}\bigl\{-a^{\rho}[i,l]\bigr\}+2\max_{l\in [i,j]}|s^{\rho}[i,l]|. 
	\]
	The base case $n=2$ has been verified. 
	Next suppose \eqref{bq3b} holds for $n-1$. Then by the base case, 
	\[ 
			\max_{l\in [i,j]}-{f}_n^{\rho}[i,l]\leq  \max_{l\in [i,j]}-{f}_{n-1}^{\rho}[i,l]+2\max_{l\in [i,j]}|{x}_n^{\rho}[i,l]|,
		\]
	which implies the induction step, and proves \eqref{bq3b}. To show \eqref{bq3a}, it is enough to show
	\begin{align}
		\max_{l\in [i,j]}{d}_n^{\rho}[i,l]&\leq  \max_{l\in [i,j]}\bq{x}_{n}^{\rho}[i,l]\label{eq17}\\	\text{and} \quad \max_{l\in [i,j]}-{d}_n^{\rho}[i,l]&\leq \label{eq16} 2\sum_{k=1}^n\max_{l\in [i,j]}|{x}_{k}^{\rho}[i,l]|.
	\end{align}
		Inequality  \eqref{eq17}  follows from  \eqref{bq1} and 
	$
		D(\bq{x}_1,\bq{x}_2,\dotsc,\bq{x}_n)=D\big(D(\bq{x}_1,\bq{x}_2,\dotsc,\bq{x}_{n-1}),\bq{x}_n\big)\preceq \bq{x}_n
		$. Inequality  \eqref{eq16} follows from \eqref{bq4} and \eqref{bq3b}.
\end{proof}

For $n\in\N$, $\bq{x}\in\Qs_n$, $k\in\lzb1,n\rzb$ and a strictly increasing vector  $\bar{i}=(i_1,i_2,\dotsc,i_k)$ of integers  in $\lzb1,n\rzb$, let $\Phi[\bq{x};\bar{i}]\in\Qs_k$ be defined through
\begin{equation}\label{Phi}
\Phi[\bq{x};\bar{i}](j)\leq 
l  \text{ if and only if }\,\,\bq{x}(j)\leq i_l \qquad \forall l\in\lzb1,k\rzb.
\end{equation}
In other words, the process $\Phi[\bq{x};\bar{i}]$ relabels the classes as follows: $\lzb1, i_1\rzb\longrightarrow$ class 1, $\lzb i_1+1, i_2\rzb\longrightarrow$ class 2, and so on, up to new class $k$. 

\begin{lemma}\label{lem:Phi}
	Let $\bar{\lambda}=(\lambda_1,\dotsc,\lambda_n)$, $\bar{\bq{x}}=(\bq{x}_1,\dotsc,\bq{x}_n)\sim \nu^{\bar{\lambda}}$ and  $\bq{v}_n=\Vmap_n(\bar{\bq{x}})$. Let $0=i_0<i_1<i_2<\dotsm<i_k\le n$ and $\bar{i}=(i_1,i_2,\dotsc,i_k)$. Then
	\begin{equation*}
		\Phi[\bq{v}_n;\bar{i}]\sim\Vmap_k(\bq{x}_{i_1},\bq{x}_{i_2},\dotsc,\bq{x}_{i_k}).
	\end{equation*}
\end{lemma}
\begin{proof}
By Theorem \ref{thm:FM},  the distribution of $\bq{v}_n$ is $\mu^{\bar\lambda}$, the unique spatially ergodic invariant distribution of $n$-type TASEP. 
	The map $\Phi$ preserves shift-ergodicity and  commutes with the TASEP dynamics (Lemma~\ref{lem:FU}). Thus $\Phi[\bq{v}_n;\bar{i}]$ has  the unique shift-ergodic stationary distribution of $k$-type TASEP, with density   $\sum_{l=i_{m-1}+1}^{i_{{m}}}\lambda_{l}$ of particles of class $m\in\{1,\dotsc,k\}$. This distribution must be that of   $\Vmap_k(\bq{x}_{i_1},\bq{x}_{i_2},\dotsc,\bq{x}_{i_k})$.
\end{proof}

\section{$D$ space}

We review first general facts about the space $D(\R,S)$ of cadlag functions from $\R$ into a complete, separable metric space $(S,d)$.  In the next section we specialize to the path space relevant for this paper where 
  $S=C(\R)$ with its Polish topology of uniform convergence on compact subsets, metrized by  the metric $d$ in \eqref{d}. 

\subsection{$D(\R,S)$}

First    we recall  the complete separable metric for the 
Polish Skorokhod topology on the space $D(\R,S)$ and then state a criterion for distributional convergence on this space.   Let $\Lambda$ be the set of  continuous bijections $\lambda:\R\to\R$ such that  
\begin{align*}
	\gamma(\lambda)=\sup_{s<t}\Big|\log\frac{\lambda(s)-\lambda(t)}{s-t}\Big|<\infty.
\end{align*}
For $x,y\in D(\R,S)$, $\lambda\in\Lambda$ and $u>0$  define
\begin{align*}
	r(x,y,\lambda,u)=\sup_{ t\ge0}d\big(x(t\wedge u),y(\lambda(t)\wedge u)\big)
 \vee \sup_{t\le 0}d\big(x(t\vee (-u)),y(\lambda(t)\vee (-u))\big). 
\end{align*}
Then a complete separable metric on $D(\R,S)$ is given by 
\[
	r(x,y)=\inf_{\lambda\in\Lambda}\big[\gamma(\lambda)\vee \int_0^\infty e^{-u}r(x,y,\lambda,u)\,du \big]. 
\]

We state the weak convergence criterion that we utilize   for processes with paths in $D(\R,S)$. The ingredients are standard and spelled out in Lemma A.17 in \cite{Busani-2021}.   For $X\in D(\R,S)$, define
\begin{align}\label{theta}
	\theta_X[a,b)=\sup_{s,t\in[a,b)}d\big(X(t),X(s)\big),
\end{align}
and then  the modulus of continuity $\omega:D(\R,S)\times \R_+ \times \R_+\rightarrow \R_+$  as  
\be\label{moc}\begin{aligned}
	\omega(X,t,\delta)=\inf\bigl\{\max_{1\leq i \leq n} \theta_X[t_{i-1},t_i): \; &\exists n\geq 1,\; -t=t_0<t_1< \dots<t_n=t \\
	&\text{such that $t_i-t_{i-1}>\delta$ for all $i\leq n$}\bigr\}.
\end{aligned}\ee

\begin{lemma}\label{lem:conv}
	Let $\{X^N\}_{N\in\N}$ be a random sequence in $D(\R,S)$.  Let $\mathbf T\subseteq \R$ be dense. Assume conditions \eqref{ss}--\eqref{fdc} below. 
	\begin{enumerate} [label={\rm(\roman*)}, ref={\rm\roman*}]   \itemsep=3pt  
		\item \label{ss} For each $t\in \mathbf T$ and $0<\delta, \epsilon<1$, there exist finite  $C(t,\epsilon)$   and $N_1(t,\epsilon,\delta)$  such that  
		\begin{align*}
			\Pp\Big(\sup_{u,v\in(t-\delta,t+\delta]}d(X^N_u,X^N_v)>\epsilon\Big)<C\delta \qquad\text{ for } N>N_1.
		\end{align*}
		\item \label{cfd} For each $k\in\N$ and $k$-tuple  $(t_1, \dotsc,t_k)\in \mathbf T^k$,  there exists a  probability distribution $p_{t_1, \dotsc,t_k}$ on $S^k$ such that $(X^N_{t_1}, \dotsc,X^N_{t_k})\Rightarrow p_{t_1, \dotsc,t_k}$.
		
		\item \label{fdc}For every $\epsilon>0$ and $T>0$,\; $\ddd\lim_{\delta \rightarrow 0} \limsup_{N\rightarrow \infty} \Pp\big(\omega(X^N,T,\delta)>\epsilon\big)=0.$
	\end{enumerate}
	Then there exists a unique process $X\in D(\R,S)$ with finite-dimensional distributions $\{p_{t_1, \dotsc,t_k}\}$  such that $X^N\Rightarrow X.$
	 
\end{lemma}

Next we collect various basic facts related to jumps of cadlag paths.  

\begin{lemma}\label{lm:D567}     Let $\eta\in D(\R, S)$.   
Suppose $t_n\to t$, $t_n\ne t$ for all $n$. Then $d(\eta(t_n-), \eta(t_n))\to 0$. 
\end{lemma} 
\begin{proof}
If $t_n<t$ along a subsequence,  then $\eta(t_n\pm)\to \eta(t-)$ along this subsequence. The other possibility is that $t_n>t$ along a subsequence.   Then $\eta(t_n\pm)\to \eta(t)$ along this subsequence.
\end{proof} 

The next two lemmas concern converging sequences 
 $\eta_n\to\eta$ in $D(\R, S)$.  This convergence is equivalent to the existence of a sequence of  strictly increasing bijections $\lambda_n:\R\to\R$ such that $\forall T<\infty$, 
\[ \lim_{n\to\infty} \sup_{t\in[-T,T]}  |\lambda_n(t)-t|=0
\quad\text{and}\quad 
\lim_{n\to\infty} \sup_{t\in[-T,T]}  d\bigl(\eta_n(t), \eta(\lambda_n(t))\bigr)=0. 
\]

\begin{lemma}\label{lm:D557}   Let $\eta_n\to\eta$ in $D(\R,S)$ and $a>0$. Then  for each $T<\infty$ there exists $\delta>0$ such that in each $\eta_n$, jumps of size $\ge a$ in $[-T,T]$  are separated from each other by at least $\delta$.
\end{lemma} 

\begin{proof}  In any given $\eta\in D(\R,S)$,  jumps of size $\ge a$  in $[-T,T]$  are finite in number and  separated from each other. 
 If  the lemma fails, then along some subsequence (still denoted by $n$) there exist $u_n<v_n$ in $[-T,T]$ such that $v_n-u_n\to0$,  $d(\eta_n(u_n-), \eta_n(u_n))\ge a$,  and $d(\eta_n(v_n-), \eta_n(v_n))\ge a$. Pass to a further subsequence so that $u_n\to r$ and $v_n\to r$.  Let $u'_n=\lambda_n(u_n)$ and $v'_n=\lambda_n(v_n)$.   Since the $\lambda_n$   are strictly increasing bijections of $\R$ that converge to the identity uniformly on $[-T,T]$, we also have  $u'_n<v'_n$ and $u'_n\to r$ and $v'_n\to r$.

The local uniformity given by  $D$-convergence gives 
\[     d( \eta_n(u_n\pm),  \eta(u'_n\pm) ) \to 0  
\quad\text{and}\quad 
d( \eta_n(v_n\pm),  \eta(v'_n\pm) ) \to 0. \] 
Hence for large enough $n$, 
\[  d(\eta(u'_n-), \eta(u'_n))\ge a/2  \quad\text{and}\quad   d(\eta(v'_n-), \eta(v'_n))\ge a/2. \]  
By Lemma \ref{lm:D567},  this is possible only if $u'_n=v'_n=r$ for large enough $n$, contradicting $u'_n<v'_n$. 
\end{proof}

For $R\in\R$ and $a>0$,  define the nondecreasing sequence  $\{\tau^R_{k,a}(\eta)\}_{k\in\Z_+}\subset[R,\infty]$ by  
\be\label{tRka} \begin{aligned}
 \tau^R_{k,a}(\eta) =   \inf\bigl\{  s\in[R,\infty):  & \;  d(\eta(s-),\eta(s))\ge a, \   \exists  s_0<s_1<\dotsm<s_{k-1}\in[R,s)  \\[3pt]
    &\quad  \text{ such that } d(\eta(s_j-),\eta(s_j))\ge a \; \forall j\in\lzb0,k-1\rzb \bigr\}  . 
\end{aligned}\ee 
The finite values in $\{\tau^R_{k,a}(\eta)\}_{k\in\Z_+}$ are exactly the locations in $[R,\infty)$ of 
the  jumps of $\eta$ of size $\ge a$, including a possible jump at $R$.  

\begin{lemma} \label{lm:D573}  Each   $\tau^R_{k,a}:D(\R,S)\to [R,\infty]$ is a lower semicontinuous function      and hence in particular  Borel measurable. 
\end{lemma} 

\begin{proof}   Fix $R\in\R$ and  $a>0$ and abbreviate $\tau_k=\tau^R_{k,a}$.  
Suppose $\eta_n\to\eta$ in $D(\R,S)$.  
Begin by checking  this claim  for a compact interval $[u,v]\subset\R$: 
\be\label{D534} \begin{aligned} 
&\text{Suppose $s_n\in[u,v]$ satisfy  $d(\eta_n(s_n-), \eta_n(s_n))\ge a$ for all $n\in\N$. }\\
 &\text{Then every subsequence of $\{s_n\}$ has a further subsequence with a limit } \\
 &\text{$s_n\to  s\in[u,v]$ such that $d(\eta(s-), \eta(s))  \ge a$.}
\end{aligned}\ee
Pass to a subsequence such that $s_n\to s$ and thereby also $\lambda_n(s_n)\to s$. 
 By the local uniformity implied by $D$-convergence, along a subsequence, 
 \begin{align*}
\lim_{n\to\infty} d(\eta(\lambda_n(s_n)-), \eta(\lambda_n(s_n))) 
= \lim_{n\to\infty}  d(\eta_n(s_n-), \eta_n(s_n))  \ge a. 
 \end{align*} 
 By Lemma \ref{lm:D567}, for all large enough $n$, we must have $\lambda_n(s_n)=s$.   Claim \eqref{D534} has been verified.  

\medskip 

For $k\ge0$ we have to show  
\be\label{D530} \tau_k(\eta)\le\varliminf_{n\to\infty}\tau_k(\eta_n).  \ee
  We can assume $\varliminf_{n\to\infty}\tau_k(\eta_n)<T<\infty$.   Restrict to a subsequence $n_j$ along which the liminf is realized for each $p\in\lzb0,k\rzb$:  
\[  s^p   =   
 \lim_{j\to\infty}  \tau_p(\eta_{n_j}) 
=  \varliminf_{n\to\infty}\tau_p(\eta_n) \; \in \; [R,T]. 
\] 
By Lemma \ref{lm:D557}, $\exists\delta>0$ such that in the limit $s^{p-1} \le s^{p}-\delta$ for $p\in\lzb1,k\rzb$.   By claim \eqref{D534},  $s^0<\dotsm<s^k$ are  locations  in $[R,T]$ of jumps of $\eta$ of size $\ge a$. 
Thus $\tau_p(\eta)\le s^p$ for each $p\in\lzb0,k\rzb$.   Lower semicontinuity has been verified.   
\end{proof}

Note that for $T<\infty$,  $\{\eta\in D(\R, S): \eta \text{ has a jump of size } b \text{ in } [-T,T] \}$ is a closed  subset of $D(\R, S)$, by adapting the proof of statement \eqref{D534}.   Hence 
\begin{align*} 
&\{\eta\in D(\R, S): \eta \text{ has a jump of size } b\} \\
&\qquad  \qquad  = \bigcup_{T=1}^\infty 
\{\eta\in D(\R, S): \eta \text{ has a jump of size } b \text{ in } [-T,T] \}  
\end{align*}  
   is an $F_\sigma$ set and thereby a Borel subset of $D(\R, S)$.   


\medskip 

\subsection{The path space $D(\R, C(\R))$ of SH} 
\label{sec:DSH} 

We specialize now to the space $D(\R, C(\R))$ relevant for the present study. A generic element of $D(\R, C(\R))$ is denoted by $\R\ni\mu\mapsto\psi_\mu(\bbullet)\in C(\R)$ and $\psi_\mu(x)\in\R$ denotes the value of the function $\psi_\mu(\bbullet)\in C(\R)$ at $x\in\R$.    
The convergence $\psi^n\to\psi$ in $D(\R, C(\R))$ means that there exist strictly increasing bijections $\lambda_n:\R\to\R$ such that the following locally uniform limits hold for all $\mu_0,M\in\R_+$: 
\be\label{psi347} \begin{aligned} 
\lim_{n\to\infty} \sup_{\mu\in[-\mu_0, \mu_0]}  |\lambda_n(\mu)-\mu|=0\,, 
\quad  
&\lim_{n\to\infty} \sup_{\mu\in[-\mu_0, \mu_0]}   \sup_{x\in[-M,M]} 
\bigl\vert \psi^n_{\mu\pm}(x) - \psi_{\lambda_n(\mu)\pm}(x)\bigr\vert =0\,, 
\\[3pt] 
\text{and}\qquad  &\lim_{n\to\infty} \sup_{\mu\in[-\mu_0, \mu_0]}   \sup_{x\in[-M,M]} 
\bigl\vert \psi^n_{\lambda_n^{-1}(\mu)\pm}(x) - \psi_{\mu\pm}(x)\bigr\vert =0. 
\end{aligned}\ee 
The notation $\lambda_n$ appears below always in this same meaning, in reference to a particular instance of $\psi^n\to\psi$.

For $\psi\in D(\R, C(\R))$ define the jump set 
\[  \Xi(\psi)=\{\mu\in\R:   \psi_\mu\ne \psi_{\mu-} \} \] 
and the difference function 
\[  J_{\mu, \psi}(x) =\psi_\mu(x)-\psi_{\mu-}(x), \qquad \mu, x\in\R. \]
The composition below shows that  $(\mu,\psi)\mapsto J_{\mu, \psi}$ is a Borel mapping of $\R\times D(\R, C(\R))$ into $C(\R)$:
  \begin{align*}
 (\mu,\psi)  \overset{\rm(a)} \mapsto (\mu,(\psi_{\mu-}, \psi_{\mu}))\overset{\rm(b)} \mapsto (\mu, \psi_{\mu}-\psi_{\mu-})  \mapsto   \psi_{\mu}-\psi_{\mu-} . 
 \end{align*}
 Step (a) takes $\R\times D(\R, C(\R))$ into $C(\R)\times C(\R)$ and is measurable  because projections are measurable on $D$-space.   Step (b) is subtraction from $C(\R)\times C(\R)$ into $C(\R)$.

For real $a>0$   set 
\be\label{smua57}   \sigma_{\mu,a}(\psi) = \inf\bigl\{ r\ge0:   \sup_{x\in[-r,r]} |J_{\mu, \psi}(x)|\ge a\bigr\}. 
\ee 
Then $ \sigma_{\mu,a}(\psi)<\infty$ implies   $\mu\in\Xi(\psi)$, while   $\mu\in\Xi(\psi)$ implies that $ \sigma_{\mu,a}(\psi)<\infty$ at least for small enough $a>0$.


\begin{lemma} \label{lm:psi3.87}  Measurability of  $\sigma_{\mu,a}(\psi)$:

{\rm (a)}   For fixed $a>0$,  the  
$\R\times D(\R, C(\R))\to[0,\infty]$  function $(\mu, \psi)\mapsto \sigma_{\mu,a}(\psi)$   is lower semicontinuous and hence   jointly Borel measurable in $(\mu, \psi)$. 

{\rm (b)}    For fixed $a>0$ and $\mu\in\R$,  the function  
$\sigma_{\mu,a}:  D(\R, C(\R))\to[0,\infty]$ is lower semicontinuous and hence  Borel measurable.
\end{lemma} 

\begin{proof}    We show that for $M\in(0,\infty)$,  $\{(\mu, \psi):  \sigma_{\mu,a}(\psi)\le M\}$ is a closed subset of $\R\times D(\R, C(\R))$.   This proves part (a). Part (b) follows by fixing $\mu$. 

  Let $\mu_n\to\mu$ in $\R$, $\psi^n\to\psi$ in $D(\R, C(\R))$, and $\sigma_{a, \mu_n}(\psi^n)\le M$.   Then $\exists x_n\in[-M,M]$ such that $|\psi^n_{\mu_n}(x_n)- \psi^n_{\mu_n-}(x_n)|\ge a$. 
 From \eqref{psi347} with $\alpha_0>|\mu|$, 
\[ \lim_{n\to\infty}  \sup_{\nu\in[-\alpha_0, \alpha_0]}  
 \sup_{x\in[-M,M]} 
\bigl\vert \psi^n_{\nu\pm}(x) - \psi_{\lambda_n(\nu)\pm}(x)\bigr\vert =0 
\ \implies \  
\lim_{n\to\infty}  | \psi_{\lambda_n(\mu_n)}(x_n)- \psi_{\lambda_n(\mu_n)-}(x_n) | \ge a.  
 \]
  We have   $\lambda_n(\mu_n)\to\mu$ but the size of the jump of $\psi_\bbullet$ at $\lambda_n(\mu_n)$ does not decay to zero. 
By Lemma \ref{lm:D567},  this is possible only if $\lambda_n(\mu_n)=\mu$ for all large enough $n$. This  turns the above into  $\lim_{n\to\infty}  |\psi_{\mu}(x_n)- \psi_{\mu-}(x_n)| \ge a$.    Since  $x_n\in[-M,M]$  we have $\sigma_{\mu,a}(\psi)\le M$. 
\end{proof}

The object of interest is the point  measure  $\Lambda_a$ on 
$\R\times\R_+\times C(\R)$ defined for $\psi\in D(\R, C(\R))$ and  $a>0$: 
\be\label{La87}  \Lambda_a(\psi) = \sum_{\mu\tsp\in\tsp\Xi(\psi)} \delta_{(\mu,  \,\sigma_{\mu,a}(\psi),\, J_{\mu, \psi})} . \ee 
We argue that $\Lambda_a(\psi)$ is finite on bounded sets. 
Let $M\in(0,\infty)$ and consider the projection $D(\R, C(\R))\to D(\R, C[-M,M])$ by restriction: for $\mu\in\R$,  $C(\R)\ni\psi_\mu\mapsto \psi_{\mu,M}= \psi_\mu\vert_{[-M,M]}\in C[-M,M]$.    As in \eqref{dn}, let $d_{M}$ denote the uniform metric on $C[-M,M]$. 
 Then the set 
\be\label{psi267} 
\{ \mu\in\R:      \sigma_{\mu,a}(\psi)\le M\} =  \{ \mu\in\R:    d_{M}(\psi_{\mu-,M}, \psi_{\mu,M})\ge a  \}
\ee
is discrete because  $\mu\mapsto\psi_{\mu,M}$ is a $C[-M,M]$-valued cadlag path and large jumps cannot accumulate in a cadlag path.  
Thus $\Lambda_a$ in \eqref{La87} is 
an element of the space $\M(\R\times\R_+\times C(\R))$ of locally finite Borel measures, which is a Polish  with its vague topology.   

  

\begin{lemma}
$\Lambda_a: D(\R, C(\R))\to \M(\R\times\R_+\times C(\R))$ is a Borel mapping. 
\end{lemma} 

\begin{proof}    The Borel $\sigma$-algebra on the measure space $\M(\R\times\R_+\times C(\R))$ is generated by evaluation of the measures on bounded Borel sets.  
  Let $R\in\R$ and $M\in(0,\infty)$.   By virtue of \eqref{psi267}, we can utilize definition \eqref{tRka} to enumerate the locations in $[R, \infty)$ of jumps of size $\ge a$  and express the restriction of $\Lambda_a(\psi)$ to $[R,\infty)\times[-M,M]\times C(\R)$ as 
\[    \Lambda_a(\psi, B)  = \sum_{k=0}^\infty \ind_B\bigl(\tau^R_{k,a}(\psi),   \sigma_{a, \tau^R_{k,a}(\psi)}(\psi), J_{\tau^R_{k,a}(\psi), \psi}\bigr) 
\quad \text{for Borel } B\subset [R, \infty)\times[-M,M]\times C(\R). 
\]
All three components 
of the point measure are Borel functions of $\psi$:   $\tau^R_{k,a}(\psi)$ by Lemma \ref{lm:D573},  $\sigma_{a, \tau^R_{k,a}(\psi)}(\psi)$ by Lemma \ref{lm:psi3.87},  and   $J_{\tau^R_{k,a}(\psi), \psi}$ because $(\mu,\psi)\mapsto J_{\mu, \psi}$ is jointly measurable in $(\mu, \psi)$. 
\end{proof}

For the final piece of the argument, we restrict to 
  the following closed subspace  $D_{SH}$ of  $D(\R, C(\R))$:
\be\label{DSH7}\begin{aligned} 
D_{SH}=\bigl\{ & \psi\in D(\R, C(\R)) :   \psi_\mu(0)=0 \text{ for all $\mu\in\R$, and } \;   \\
 &\qquad\qquad   
 \text{for each pair $\mu<\nu$ in $\R$, }  x\mapsto \psi_\nu(x)- \psi_\mu(x) \text{  is nondecreasing\tsp}
 \bigr\} .
\end{aligned}\ee
This is the path space of the stationary horizon and the processes  $H^{v,N}$ in \eqref{H138}. 
For  $\psi\in D_{SH}$  we can write   the definition \eqref{smua57} of   $\sigma_{\mu,a}$ without  absolute values: 
\be\label{smua7}   \sigma_{\mu,a}(\psi)= \inf\bigl\{ y\ge0:  [\psi_\mu(y)- \psi_{\mu-}(y)] \vee [\psi_{\mu-}(-y) -\psi_\mu(-y) ] \, \ge \, a\bigr\}  \qquad\text{for } \psi\in D_{SH} .\ee   
Furthermore, $J_{\mu, \psi}$  is a nondecreasing function on $\R$.

\begin{lemma}\label{lm:psi87}   Fix $a>0$ and $\psi\in D_{SH}$. 
Suppose there exists a symmetric,  dense subset $\mathcal Z$  of $\R$ such that 
$|\psi_\mu(z)- \psi_{\mu-}(z)|\ne a$ 
for all $\mu\in\R$ and $z\in\mathcal Z$.  
Then whenever  $\psi^n\to\psi$ in 
 $D_{SH}$, also  $\Lambda_a(\psi^n)\to\Lambda_a(\psi)$ in the space $\M(\R\times\R_+\times C(\R))$.  
\end{lemma}

\begin{proof}

Let  $\mu_0>0$ be such  that $\mu\mapsto\psi_\mu$ is continuous at $\pm\mu_0$.   Let $M>0$ satisfy   $\pm M\in\mathcal Z$. 
Let $(\mu_1, \sigma_1, J_1)$, $\dotsc$, $(\mu_k, \sigma_k, J_k)$ with $\sigma_i=\sigma_{a, \mu_i}(\psi)$ and $J_i=\psi_{\mu_i}-\psi_{\mu_i-}$  be an enumeration of the finite set 
\be\label{psi337} \{ (\mu, \sigma_{\mu,a}(\psi), J_{\mu, \psi}):    \mu\in\Xi(\psi)\cap[-\mu_0,\mu_0], \,   \sigma_{\mu,a}(\psi)\le M\}  .   \ee
   We claim that for large enough $n$, 
\be\label{psi377} \begin{aligned}
&\{ (\mu, \sigma_{\mu,a}(\psi^n), J_{\mu, \psi^n}):    \mu\in\Xi(\psi^n)\cap[-\mu_0,\mu_0], \,   \sigma_{\mu,a}(\psi^n)\le M\}    \\[4pt]  &\qquad \qquad 
=\{ (\mu^n_1, \sigma^n_1, J^n_1),\dotsc,(\mu^n_k, \sigma^n_k, J^n_k) \} 
\end{aligned}  \ee 
such that as $n\to\infty$,   $(\mu^n_i, \sigma^n_i, J^n_i) \to (\mu_i, \sigma_i, J_i)$ for $i\in\lzb1,k\rzb$.  Since $\mu_0$ and $M$ can be taken arbitrarily large, this implies the vague convergence $\Lambda_a(\psi^n)\to\Lambda_a(\psi)$ of simple point measures.   The rest of this proof verifies the claim. 



\smallskip 

Since $\sigma_i\le M$  while $\psi_{\mu_i}(M)- \psi_{\mu_i-}(M)  \ne a$ and $\psi_{\mu_i-}(-M)- \psi_{\mu_i}(-M)  \ne a$, we must have 
\be\label{psi897} 
\sigma_i< M \quad \text{and}\quad  
[\psi_{\mu_i}(M)- \psi_{\mu_i-}(M)] \vee [\psi_{\mu_i-}(-M)- \psi_{\mu_i}(-M)] > a. 
\ee
 
 Set 
\be\label{psi927} \begin{aligned} 
&\mu^n_i=\lambda_n^{-1}(\mu_i) \quad\text{for } \ i\in\lzb1,k\rzb,  \quad \text{which determines }  \\[2pt]
& \quad  \sigma^n_i=\sigma_{a,\mu^n_i}(\psi^n)=\sigma_{a, \lambda_n^{-1}(\mu_i)}(\psi^n) \ \text{ and }\ 
J^n_i =  J_{\mu^n_i, \psi^n}=J_{\lambda_n^{-1}(\mu_i), \,\psi^n}.
\end{aligned} \ee 
  The first limit in \eqref{psi347}  gives $\mu^n_i\to\mu_i$.  The third limit in \eqref{psi347} together with \eqref{psi897} ensures that,  for large enough $n$, $\mu^n_i\in\Xi(\psi^n)$  and  
 \[  [\psi^n_{\mu^n_i}(M)- \psi^n_{\mu^n_i-}(M)] \vee [\psi^n_{\mu^n_i-}(-M)- \psi^n_{\mu^n_i}(-M)]  > a \] 
  and thereby also   $\sigma^n_i<M$.  The third limit in \eqref{psi347} gives also the locally uniform convergence 
 $ 
   J^n_i=\psi^n_{\lambda_n^{-1}(\mu_i)}-\psi^n_{\lambda_n^{-1}(\mu_i)-} \to \psi_{\mu_i}-\psi_{\mu_i-}=J_i$. 
    An  application of  Lemma \ref{lm:f677} (to be proved below)    to the functions $f_n=J^n_i$ and $f=J_i$ 
    gives  the limits $\sigma^n_i \to \sigma_i$.  
  
To summarize,   \eqref{psi927}  defines the set on the right of \eqref{psi377} which converges element by element to the set in \eqref{psi337} and which is a subset of the set on the left of \eqref{psi377}.  It remains to verify that for large enough $n$ the set on the left of \eqref{psi377} has no elements besides $(\mu^n_1, \sigma^n_1, J^n_1),\dotsc,(\mu^n_k, \sigma^n_k, J^n_k)$.  

  Suppose on the contrary that along some subsequence (denoted again by $n$)  there exists $\nu^n\in\Xi(\psi^n)\cap[-\mu_0,\mu_0]$ such that 
$\nu^n\notin\{\mu^n_1, \dotsc,\mu^n_k\}$ and 
 $\rho^n= \sigma_{a, \nu^n}(\psi^n)\le M$.   The latter condition forces 
 $[\psi^n_{\nu^n}(M)-\psi^n_{\nu^n-}(M)]\vee[\psi^n_{\nu^n-}(-M)-\psi^n_{\nu^n}(-M)]\ge a$.
 The limit in \eqref{psi347}  then implies that 
 \be\label{psi987} \varliminf_{n\to\infty} [\psi_{\lambda_n(\nu^n)}(M)-\psi_{\lambda_n(\nu^n)-}(M)]\vee[\psi_{\lambda_n(\nu^n)-}(-M)-\psi_{\lambda_n(\nu^n)}(-M)] \, \ge\, a. \ee
Pass to a further subsequence (still denoted by $n$) such that $\lambda_n(\nu^n)\to\bar\nu\in[-\mu_0,\mu_0]$.   Then  by Lemma \ref{lm:D567}, it must be that $\lambda_n(\nu^n)=\bar\nu$ for all large enough $n$ in the subsequence.
   We have established the existence of  
\[   \bar\nu\in[-\mu_0,\mu_0] \ \ \text{ such that } \ \    [\psi_{\bar\nu}(M)-\psi_{\bar\nu-}(M)]\vee[\psi_{\bar\nu-}(-M)-\psi_{\bar\nu}(-M)] \, \ge \, a  . 
\] 
 Thus $(\bar\nu, \sigma_{a, \bar\nu}(\psi), J_{\bar\nu,\psi})$ is an element of the set \eqref{psi337} and hence must equal $(\mu_j, \sigma_j, J_j)$ for some $j\in\lzb1,k\rzb$. 
 Now $\nu^n$ and $\mu^n_j$ are different  locations in $\psi^n$ of jumps of size $\ge a$ but both converge to $\bar\nu=\mu_j$.   This contradicts Lemma \ref{lm:D557}.   
\end{proof}



It remains to provide the technical lemma appealed to above in the proof of Lemma \ref{lm:psi87}:

\begin{lemma} \label{lm:f677}    Let $f_n\in C(\R)$ be nondecreasing functions such that  $f_n(0)=0$ and $f_n\to f$ locally uniformly. Let $a>0$, $\sigma_n=\inf\{ x\ge0: f_n(x)\vee[-f_n(-x)] \ge a\}$ and $\sigma=\inf\{ x\ge0: f(x)\vee[-f(-x)] \ge a\}$.   Then $\sigma\le \varliminf\sigma_n$.   

Assume further that $|f(z)|\ne a$  for $z$ in some symmetric dense subset $\mathcal Z$  of $\R$. Then $\sigma= \lim\sigma_n$.   
\end{lemma} 

\begin{proof}
Suppose   $\sigma_{n_j}\to y<\infty$.  Then by the local uniform convergence,   
$a= f_{n_j}(\sigma_{n_j}) \vee [-f_{n_j}(-\sigma_{n_j})]  \to  f(y)\vee[-f(-y)]$, which implies $\sigma\le y$.  Thus we have 
$\sigma\le \varliminf\sigma_n$. 

To prove the remaining part we can assume $\sigma<\infty$.   Pick $z\in\mathcal Z$ such that $z>\sigma$.  Then $f(z)\vee[-f(-z)] >a$, and the limit forces $f_n(z)\vee[-f_n(-z)] >a$ for large enough $n$, implying $\sigma_n<z$. Thus we have $\varlimsup\sigma_n\le\sigma$. 
\end{proof}

\section{Stationary horizon} \label{sec:stat_horiz}
Consider the following map from ~\cite{Seppalainen-Sorensen-21b} (an equivalent yet somewhat cumbersome version of this map was used in \cite{Busani-2021})  
defined for functions that satisfy $f(0) = g(0) = 0$:  
\be \label{Phialt}
\Phi(f,g)(y) = f(y) + \sup_{-\infty <x \le y }\{g(x) - f(x)\} - \sup_{-\infty < x \le 0}\{g(x) - f(x)\}
\ee
We note that the map $\Phi$ is well-defined only on the appropriate space of functions where the suprema are all finite.
This map extends to maps $\Phi^k:C(\R)^k \to C(\R)$ as follows. \begin{enumerate}
    \item $\Phi^1(f_1)(x) = f_1(x)$. 
    \item $\Phi^2(f_1,f_2)(x) = \Phi(f_1,f_2)$,\qquad\text{and for }$k \ge 3,$ 
    \item $\Phi^k(f_1,\ldots,f_k) = \Phi(f_1,\Phi^{k - 1}(f_2,\ldots,f_k))$.
\end{enumerate}
We may drop the superscript and simplify to  $\Phi(f_1,\ldots,f_k) = \Phi^k(f_1,\ldots,f_k)$. 
As throughout the paper, $C(\R)$ has the Polish topology of uniform convergence on compact sets.

\begin{definition} \label{def:SH}
The stationary horizon $\{G_\mu\}_{\mu \in \R}$ is a process with state space $C(\R)$ and with paths in the Skorokhod space $D(\R,C(\R))$ of right-continuous functions $\R \to C(\R)$ with left limits.  The law of the stationary horizon is characterized as follows: For real numbers $\mu_1 < \cdots < \mu_k$, the $k$-tuple  $(G_{\mu_1},\ldots,G_{\mu_k})$ of continuous functions   has the same law as $(f_1,\Phi^2(f_1,f_2),\ldots,\Phi^k(f_1,\ldots,f_k))$, where $f_1,\ldots,f_k$ are independent two-sided Brownian motions with drifts $2\mu_1,\ldots,2\mu_k$, and each with diffusion coefficient $\sqrt 2$. 
\end{definition}


The following theorem  collects  facts about the stationary horizon from \cite{Busani-2021,Seppalainen-Sorensen-21b,Busa-Sepp-Sore-22arXiv}. For notation, let $G_{\mu+} = G_\mu$, and let $G_{\mu -}$ be the limit of $G_{\alpha}$ as $\alpha \nearrow \mu$. 
\begin{theorem}[\cite{Busani-2021}, Theorem 1.2; \cite{Seppalainen-Sorensen-21b}, Theorems 3.9, 3.11, 3.15, 7.20 and Lemma 3.6] \label{thm:SH10}  The following hold for the stationary horizon.
\begin{enumerate} [label=\rm(\roman{*}), ref=\rm(\roman{*})]  \itemsep=3pt
    \item\label{itm:SHpm} For each $\mu \in \R$, $G_{\mu -} = G_{\mu +}$ with probability one,  and $G_\mu$ is a two-sided Brownian motion with diffusion coefficient $\sqrt 2$ and drift $2\mu$
    \item \label{itm:SH_sc} For $c > 0$ and $\nu \in \R$, \ 
    $ 
    \{cG_{c (\mu + \nu)}(c^{-2}x) - 2\nu x  : x\in \R\}_{\mu \in \R} \,\deq\, \{G_\mu(x): x \in \R\}_{\mu \in \R}.
    $ 
    \item\label{itm:SH_sc2} Spatial stationarity holds in the sense that,  for $y \in \R$, 
    \[\{G_{\mu}(x):x \in \R\}_{\mu \in \R} \deq \{G_{\mu}(y,x + y): x \in \R\}_{\mu \in \R}.\]
    
\item \label{SHrelfect}Reflection property:  
        $
        \{G_{(-\mu)-}(-\tspb\aabullet)\}_{\mu \in \R} \deq \{G_\mu(\aabullet)\}_{\mu \in \R}.
        $
    
    \item \label{itm:SHdist} Fix $x > 0$ , $\mu_0 \in \R$,  $\mu > 0$, and $z \ge 0$. Then, 
    \begin{align*}
    &\Pp\bigl(\sup_{a,b \in [-x,x]}|G_{\mu_0 + \mu}(a,b) - G_{\mu_0 }(a,b)| \le z\bigr) = \Pp\bigl(G_{\mu_0 + \mu}(-x,x) - G_{\mu_0}(-x,x) \le z\bigr) \\
    &\qquad = \Pp\bigl(G_{\mu_0 + \mu}(2x) - G_{\mu_0}(2x) \le z\bigr) \\
    &\qquad = \Phi\Bigl(\f{z - 2\mu x}{2\sqrt {2x}}\Bigr) + e^{\f{\mu z}{2}}\biggl(\Bigl(1 + \tfrac12{\mu}z + \mu^2 x \Bigr)\Phi\Bigl(-\f{z + 2\mu x}{2 \sqrt{2 x}}\Bigr) - \mu\sqrt{{x}/{\pi}\tspa} \tspb e^{-\f{(z + 2\mu x)^2}{8x}}\biggr)
    \end{align*}
     where $\Phi$ is the standard normal distribution function. This distribution has an atom at $z=0$ and no other atoms. 
    \item \label{itm:exp} For $x < y$ and $\alpha < \beta$, with $\#$ denoting the cardinality, 
    \[
    \E[\#\{\mu \in (\alpha,\beta): G_{\mu-}(x,y) < G_{\mu +}(x,y) \}] = 2\sqrt{{2}/{\pi}\tsp}
    (\beta - \alpha)\sqrt{y - x}.
    \]
\end{enumerate}
Furthermore, the following holds on a single event of full probability.
\begin{enumerate} [resume, label=\rm(\roman{*}), ref=\rm(\roman{*})]  \itemsep=3pt
    \item \label{itm:SH_j} For $x_0 > 0$   define  the process $G^{x_0} \in D(\R,C[-x_0,x_0])$   by restricting each function $G_\mu$ to $[-x_0,x_0]$: $G^{x_0}_\mu=G_\xi\vert_{[-x_0,x_0]}$. Then, $\mu\mapsto G^{x_0}_\mu$ is a $C[-x_0,x_0]$-valued  jump process with finitely many jumps in any compact interval, but countably infinitely many jumps in $\R$. The number of jumps in a compact interval has finite expectation given in item \ref{itm:exp} above, and each direction $\mu$ is a jump direction with probability $0$. In particular, for each $\mu \in \R$ and compact set $K$, there exists a random $\ve = \ve(\mu,K)>0$ such that for all $\mu - \ve < \alpha < \mu < \beta < \mu  + \ve$, $\sigg \in \{-,+\}$, and all $x \in K$, $G_{\mu -}(x) = G_{\alpha}(x)$ and $G_{\mu +}(x) = G_{\beta}(x)$.
    \item \label{itm:SH_mont} For $x_1 \le x_2$, $\mu \mapsto G_{\mu}(x_1,x_2)$ is a non-decreasing jump process.
    \item Let $\alpha < \beta$. The function $x\mapsto G_\beta(x)-G_\alpha(x)$ is nondecreasing.  There exist finite $S_1 = S_1(\alpha,\beta)$ and $S_2 = S_2(\alpha,\beta)$ with $S_1 < 0 < S_2$ such that $G_{\alpha }(x) = G_{\beta }(x)$ for $x \in [S_1,S_2]$ and $G_{\alpha }(x) \ne G_{\beta }(x)$ for $x \notin [S_1,S_2]$.
    \item \label{itm:bad_dir_contained} 
    Let $\alpha < \beta$,  $S_1 = S_1(\alpha,\beta)$ and $S_2 = S_2(\alpha,\beta)$.  Then $\exists\tspa \zeta, \eta\in[\alpha, \beta]$ such that,     \begin{align*} 
    &\text{$G_{\zeta -}(x) = G_{\zeta +}(x)$ for $x \in [-S_1,0]$, and $G_{\zeta -}(x) > G_{\zeta +}(x)$ for $x < S_1$, and } \\
    &\text{$G_{\eta -}(x) = G_{\eta +}(x)$ for $x \in [0,S_2]$, and $G_{\eta -}(x) < G_{\eta +}(x)$ for $x > S_2$.}
    \end{align*} 
        In particular, the set $\{\mu \in \R: G_{\mu} \neq G_{\mu-}\}$ is dense in $\R$.
        
\end{enumerate}
\end{theorem}

\bibliographystyle{alpha}
\bibliography{references_file}
\end{document}